\documentclass[11pt,twoside]{article}
\usepackage{a4wide}
\usepackage{latexsym,amssymb}
\usepackage{amsmath,amsthm}
\usepackage{fourier} 
\usepackage[svgnames]{pstricks}  
\usepackage{pst-plot,pst-node}
\usepackage{graphicx}
\usepackage{epsf}
\usepackage{subfigure} 
\usepackage{booktabs}
\usepackage{titlesec}
\usepackage{url}
\usepackage{fancyhdr,caption}
\usepackage{hyperref}

\newrgbcolor{offgrey}{0.2 0.2 0.5}
\hypersetup{colorlinks=true,linkcolor=offgrey,citecolor=gray}
\psset{arrowscale=2,dotscale=1.5}

\titleformat*{\section}{\bfseries\scshape\Large}
\titleformat*{\subsection}{\bfseries\scshape\large}
\titleformat*{\subsubsection}{\itshape}

\newtheorem{theorem}{Theorem}[section]
\newtheorem{lemma}[theorem]{Lemma}
\newtheorem{proposition}[theorem]{Proposition}
\newtheorem{corollary}[theorem]{Corollary}
\theoremstyle{definition}
\newtheorem{definition}[theorem]{Definition}
\newtheorem{definitions}[theorem]{Definitions}
\newtheorem{example}[theorem]{Example}

\newtheorem{remark}[theorem]{Remark}
\newtheorem{remarks}[theorem]{Remarks}

\newcommand{\ee}{\mathsf{e}}
\newcommand{\ii}{\mathsf{i}}
\newcommand{\bb}[1]{\mathbb{#1}}
\newcommand{\braid}{\mathfrak{b}}
\newcommand{\br}{\mathfrak{b}}
\newcommand{\chor}{\mathfrak{c}}
\newcommand{\Chor}{\mathfrak{C}}
\newcommand{\CC}{\mathbb{C}}
\newcommand{\defn}[1]{{\bfseries\itshape #1}}

\newcommand{\divides}{\,|\,}
\newcommand{\EGamma}{\mathrm{E}\Gamma}
\newcommand{\half}{{\textstyle\frac12}}

\newcommand{\trivial}{\mathbf{1}}

\newcommand{\Fix}{\mathrm{Fix}}
\newcommand{\RR}{\mathbb{R}}
\newcommand{\ZZ}{\mathbb{Z}}
\newcommand{\nfrac}[2]{\,\raisebox{3pt}{\scriptsize$#1$}\kern-0.8pt/\hspace{-0.3mm}\raisebox{-2pt}{\scriptsize$#2$}}
\newcommand{\LX}{\Lambda X}
\newcommand{\Xn}{X^{(n)}}
\newcommand{\Xnhat}{X_*^{(\widehat{n})}}
\newcommand{\Sone}{{\mathsf{S}^1}}
\newcommand{\Sonehat}{\widehat{\mathsf{S}^1}}
\newcommand{\jhat}{\widehat\jmath}
\newcommand{\khat}{\widehat{k}}
\newcommand{\ellhat}{\widehat{\ell}}
\newcommand{\nhat}{{\widehat{n}}}
\newcommand{\uhat}{{\widehat{u}}}

\newcommand{\sigmahat}{{\widehat{\sigma}}}
\newcommand{\deltahat}{{\widehat{\delta}}}

\newcommand{\Gammahat}{{\widehat{\Gamma}}}
\newcommand{\Gt}{G_\tau}

\newcommand{\eps}{\varepsilon}
\newcommand{\A}{\mathcal{A}}
\newcommand{\SO}{\mathsf{SO}}
\newcommand{\OO}{\mathsf{O}}
\newcommand{\lra}{\longrightarrow}
\newcommand{\paths}{\mathcal{P}}

\newcommand{\braids}{B}

\newcommand{\zerobar}{\overline{0}}

\makeatletter
  \@addtoreset{figure}{section}
  \@addtoreset{table}{section}
  \@addtoreset{equation}{section}
\makeatother

  \newrgbcolor{headcolour}{0.8 0 0}

\begin{document}
\title{Classification of symmetry groups \\ for planar $n$-body choreographies}
\author{James Montaldi, Katrina Steckles \\ \small University of Manchester}
\date{October 2013}

\thispagestyle{empty}

\maketitle

\noindent\hrulefill 

\smallskip

\centerline{\large\scshape Abstract}

\medskip

\noindent Since the foundational work of Chenciner and Montgomery in 2000 there has been a great deal of interest in choreographic solutions of the n-body problem: periodic motions where the n bodies all follow one another at regular intervals along a closed path. The principal approach combines variational methods with symmetry properties. In this paper, we give a systematic treatment of the symmetry aspect.  In the first part we classify all possible symmetry groups of planar n-body, collision-free choreographies.  These symmetry groups fall in to 2 infinite families and, if n is odd, three exceptional groups.  In the second part we develop the equivariant fundamental group and use it to determine the topology of the space of loops with a given symmetry, which we show is related to certain cosets of the pure braid group in the full braid group, and to centralizers of elements of the corresponding coset. In particular, we refine the symmetry classification by classifying the connected components of the set of loops with any given symmetry.  This leads to the existence of many new choreographies in $n$-body systems governed by a strong force potential.

\medskip

{\small

 }

\noindent\hrulefill 

{\small 
\tableofcontents
}

\section{Introduction}

The problem of determining the motion of $n$ particles under gravitational interaction has long been of interest, and since Poincar\'e there has been particular interest in periodic motions.  In the last 20 years, renewed interest has followed the discovery of what are now called choreographies:  periodic motions where the particles, assumed to be of equal mass, follow each other around a closed path at regular intervals. In 1993, Moore \cite{Moore93} discovered the first of these, where 3 identical particles move along a figure 8 curve; he found this numerically.  Independently, Chenciner and Montgomery \cite{Fig8} (re)discovered this figure-8 solution a few years later, but they proved its existence using a clever combination of symmetry methods and variational techniques.

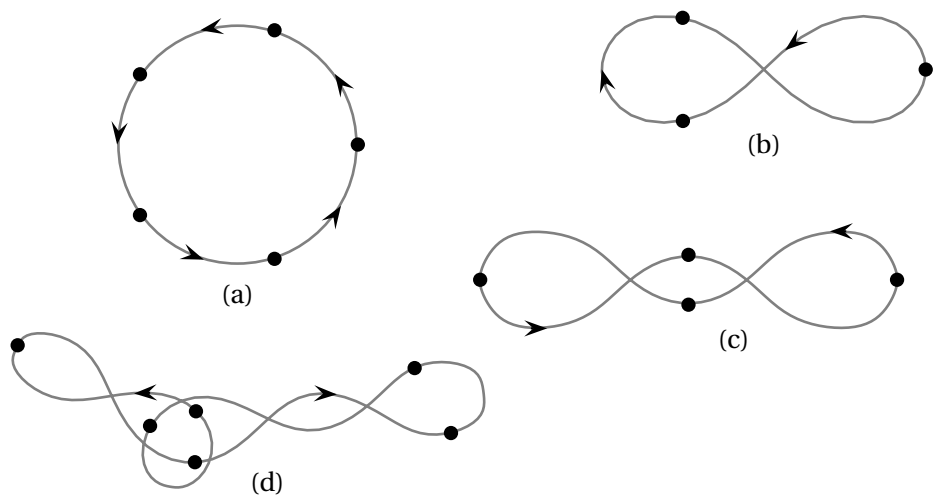
\begin{figure}[ht]
\begin{center}
\begin{pspicture}(-6,-3.5)(6,3.5)
\put(-3,1){
   \psset{unit=0.8}
   \pscircle[linewidth=1pt,linecolor=gray](0,0){2}
   \psdots(2,0)(0.618,1.9)(-1.618,1.17)(-1.618,-1.17)(0.618,-1.9)
   \psline{->}(1.721,1.018)(1.618,1.176)
   \psline{->}(-0.558,1.913)(-0.618,1.902)
   \psline{->}(-1.993,0.063)(-2,0)
   \psline{<-}(-0.558,-1.913)(-0.618,-1.89)
   \psline{<-}(1.721,-1.018)(1.618,-1.176)
   \rput(0,-2.5){(a)}
}
\rput(4,2){
{\psset{unit=2} 
    \parametricplot[linewidth=1pt,linecolor=gray]{0}{360}{%
        t sin 1.096 mul t 5 mul sin 0.0252 mul sub t 7 mul sin 0.0058 mul sub
        t 2 mul sin 0.3373 mul t 4 mul sin 0.0557 mul add}
    \psdots(1.077,0)(-.5400, .3437)(-.5384, -.3437)
    \psline{->}(.1911, .1676)(.1419, .1276)
    \psline{->}(-1.05,-0.1)(-1.077,0)
}
 \rput(0,-1){(b)}}
\rput(3,-0.8){
 \psset{unit=2} 
   \parametricplot[linewidth=1pt,plotpoints=150,linecolor=gray]{0}{360}{%
t 1 mul cos 1.37 mul   t 3 mul cos 0.024 mul add  t 5 mul cos 0.006 mul add 
   t 7 mul cos 0.023 mul sub  t 9 mul cos 0.0104 mul add
t 1 mul sin 0.167 mul    t 3 mul sin 0.232 mul add    t 5 mul sin 0.069 mul sub
  t 7 mul sin 0.013 mul add   t 9 mul sin 0.0088 mul sub }
   \psdots(1.386, 0)(0, -.166)(-1.386, 0)(0, .166)
   \psline{->}(1.0480, .32289)(.94007, .32302)
   \psline{->}(-1.0480, -.32289)(-.94007, -.32302)
\rput(0.3,-0.4){(c)}}
\rput(-3,-2.5){
 \psset{unit=2} 
   \parametricplot[linewidth=1pt,plotpoints=150,linecolor=gray]{0}{360}{%
t 1 mul cos -1.100 mul   t 2 mul cos 0.4339 mul add  t 3 mul cos 0.1958 mul add 
  t 4 mul cos 0.1119 mul add  t 5 mul cos 0.04111 mul add  t 7 mul cos 0.006463 mul add
  t 8 mul cos 0.02501 mul add  t 9 mul cos 0.01195 mul add  t 10 mul cos 0.0004111 mul add
  t 11 mul cos 0.006528 mul add
t 1 mul sin -0.7476 mul add  t 2 mul sin 0.1002 mul add  t 3 mul sin -0.06577 mul add
  t 4 mul sin -0.05939 mul add  t 5 mul sin -0.04463 mul add  t 7 mul sin 0.003164 mul add
  t 8 mul sin 0.001949 mul add  t 9 mul sin 0.004285 mul add  t 10 mul sin -0.004479 mul add
  t 11 mul sin -0.02171 mul add
t 1 mul cos -0.04743 mul  t 2 mul cos -0.1219 mul add  t 3 mul cos -0.04677 mul add
  t 4 mul cos 0.003176 mul add  t 5 mul cos 0.1449 mul add  t 7 mul cos 0.008940 mul add
  t 8 mul cos 0.003650 mul add  t 9 mul cos 0.004890 mul add  t 10 mul cos 0.01816 mul add
  t 11 mul cos 0.006594 mul add
t 1 mul sin 0.1201 mul add  t 2 mul sin 0.2715 mul add  t 3 mul sin 0.03161 mul add
  t 4 mul sin -0.03008 mul add  t 5 mul sin 0.09286 mul add  t 7 mul sin -0.05262 mul add
  t 8 mul sin 0.04184 mul add  t 9 mul sin 0.02347 mul add  t 10 mul sin 0.007619 mul add
   t 11 mul sin 0.002672 mul add
 }
 \psdots(-1.46, .415)(-.281, -.363)(1.42, -.169)(1.18, .263)(-.579, -.121)(-.275, -0.0249)
\psline{->}(.619, 0.09)(.662, 0.088)
\psline{->}(-.579, 0.099)(-.686, 0.094)
\rput(0.2,-.5){(d)}
}
\end{pspicture}
\end{center}
\caption{Examples of planar choreographies; (a) circular, (b) the figure eight, (c) the super-eight and (d) a non-symmetric choreography.}
\label{fig:choreo-egs}
\end{figure}

Since the work of Moore, Chenciner and Montgomery there have been many papers written on the subject of choreographies.  We restrict ourselves to the planar case, although interesting examples of choreographies have been shown to exist in higher dimensions \cite{hiphop1,CV2000,Ferrario07,FGN-Platonic}.  In the plane, the first choreography known (in hindsight) was the circular choreography of Lagrange, in which the particles are positioned at the vertices of a regular $n$-gon rotating with constant speed about
its centre. Soon after the work of Chenciner and Montgomery, J.~Gerver suggested a 4-particle choreography on what is called the `super-eight', a curve similar to the figure eight but with three internal regions and two crossings (\cite{CGMS02} p.~289 and Fig.~\ref{fig:choreo-egs}(c)). The papers \cite{CGMS02} and \cite{Simo01a} contains many examples of choreographies, found numerically, and it will be noticed that almost all have some geometric symmetry.

Some work approaches the questions using numerics and some use an analytic-topological approach, but almost all methods use a variational setting for the problem.  The original paper by Moore \cite{Moore93} was asking how the theory of braids could be used in the study of dynamical systems of $n$ interacting bodies in the plane---any periodic motion of $n$ particles can be represented by a braid (indeed a pure braid as the particles return to their original position after one period), and Moore's numerical approach was to use the braid as an `initial condition' for the variational problem and then `relax' the curve by decreasing the action.   Chenciner and Montgomery's approach was also variational, but they used explicitly the symmetries involved in the figure 8 solution, together with the variational setup, and the crux of their existence proof was to show that minimizing the action within the given symmetry class did not involve collisions. 

The idea of using symmetry methods in the variational problem was taken up in a very interesting paper by Ferrario and Terracini \cite{FT04}, where they gave, among other things, conditions on the symmetry under consideration guaranteeing that a minimizer of the action is free of collisions (their `rotating circle' condition, which we describe in Section~\ref{sec:classification}). 

An excellent review by Terracini was published in 2006 \cite{Terracini-review}, containing many more references. 

The principal aim of the present paper is to make systematic the combination of topological (braid) methods and symmetry methods. We begin by classifying all possible symmetry groups arising for (collision-free) choreographies in the plane, and then proceed to study symmetries in loop space, firstly in general and then referring specifically to choreographies. The work is an extension of the work presented in the second author's thesis \cite{Steckles}.

\subsection{Configurations and symmetries}

We are interested in the motion of an isolated system of $n$ identical particles in the plane. We identify the plane with $\bb{C}$---the complex numbers.  Under these assumptions, the centre of mass of the particles is given by $\frac1n \sum_jz_j$ and without loss of generality we can take this point to be fixed at the origin.  In addition we assume the particles do not collide. Later we assume they interact under a conservative attractive force.  Much of this section follows the work of Ferrario and Terracini \cite{FT04}.

The configuration space of the system is therefore
$$\Xn := \bigl\{(z_1,\dots,z_n)\in\mathbb{C}^n \;\mid\; {\textstyle\sum_j} z_j=0,\; z_i\neq z_j\;\forall i\neq j\bigr\},$$
which is a (non-compact) manifold of real dimension $2n-2$.  There is a natural symmetry group acting on $\Xn$, namely the product of the orthogonal group in the plane and the group of permutations of the $n$ points $\Gamma:=\OO(2)\times S_n$ acting by, 
\begin{equation}\label{eq:spatial action}
(A,\sigma)\cdot(z_1,\dots,z_n) = (Az_{\sigma^{-1}(1)},\dots, Az_{\sigma^{-1}(n)}).
\end{equation}
All group actions will be \emph{left} actions, whence the inverse on the permutation in \eqref{eq:spatial action}. 

Let $\Lambda=\Lambda\Xn$ be the space of all loops in $X^{(n)}$. A loop is by definition a continuous map 
$$u:\bb{T}\to X^{(n)},$$ 
where $\bb{T}$ is the time circle: we identify $\bb{T}=\mathbb{R}/\mathbb{Z}$, so loops are parametrized by $t\in[0,1]$.  Then
\begin{equation}\label{eq:z1...zn}
u(t)=(z_1(t),\dots,z_n(t))\in\Xn,
\end{equation}
where each $z_j:\bb{T}\to\CC$.  

Denote by $\Sonehat$ the symmetry group of rotations and reflections of the time circle $\bb{T}$, which is isomorphic to $\OO(2)$.   The action of $\Gamma$ on $\Xn$ extends to an action of $\Gamma\times\Sonehat$ on the loop space $\Lambda$:  
if $\tau\in\Sonehat$ we put
\begin{equation}\label{eq:temporal action}
((A,\sigma,\tau)\cdot u)(t) := (A,\sigma)\cdot u(\tau^{-1}(t)),
\end{equation}
where $(A,\sigma)$ acts as in (\ref{eq:spatial action}).  
A loop $u$ is said to have symmetry $G<\Gamma\times\Sonehat$ if $G$ is the isotropy subgroup of $u$ under the $\Gamma\times\Sonehat$-action (we use the notation $G<H$ to mean $G$ is a subgroup of $H$). Explicitly, this means that
$$ u(\tau(t)) = (A,\sigma)\cdot u(t), \qquad \forall\, (A,\sigma,\tau)\in G.$$
Notice that  $g=(I,\sigma,\tau)\in G$ means that
$$z_{\sigma(j)}(\tau(t)) = z_j(t)\quad(\forall j,t),$$
and consequently particles whose labels are within the same orbit of $\sigma$ follow the same path. In particular, if $\sigma$ is a cycle of order $n$ (the number of particles), then \emph{all} the particles follow the same path.

A particular subgroup of $\Gamma\times\Sonehat$ of central interest is the \defn{choreography group} $\Chor_n$ which is the cyclic group of order $n$ generated by $\chor=(I,\, \sigma_1,\,-\nfrac{1}{n})$, where $\sigma_1$ is the cycle $\sigma_1=(1\;2\;3\;\dots\; n)\in S_n$.  

Recall that, given any action of a group $G$ on a space $X$ the fixed point space is defined to be
$$\Fix(G, X) = \left\{x\in X\mid \forall g\in G,\; g\cdot x = x\right\}.$$

\begin{definition}
A \defn{choreography} is an element of the fixed point space $\Fix(\Chor_n,\Lambda\Xn)$.
\end{definition}

We denote this fixed point space by $\Lambda^\chor=\Lambda^\chor\Xn$.  
Explicitly, the loop (\ref{eq:z1...zn})  is a choreography if, for each $j=1,\dots,n$, 
\begin{equation} \label{eq:choreography defn}
z_{j+1}(t) = z_j(t+\nfrac1n),
\end{equation}
where the index is taken modulo $n$; in particular particle 1 follows particle 2 which in turn follows particle 3 etc., and all with the same time delay of $1/n$.  As already pointed out, this definition requires all the particles to move on the same curve. Such motions are sometimes called \emph{simple} choreographies, to distinguish from more general choreographies where more than one curve is involved, and possibly different numbers of particles on different curves: we only consider these simple choreographies. Note that the definition does not imply that the particles are in numerical order around the curve as Example~\ref{ex:circular motion} below shows.

It follows from the definition that the symmetry group $G$ of any choreography satisfies $\Chor_n<G$ (in Proposition~\ref{prop:kernels} we show it is in fact a normal subgroup).  Since $\Chor_n$ is of order $n$, it follows that $n$ divides the order of the symmetry group $G$ of any choreography.

For a given symmetry group $G<\Gamma\times\Sonehat$, we denote by $\rho,\sigma,\tau$ the  projections of $G$ to each component.  That is, given an element $g\in\Gamma\times\Sonehat$ we write its three components as $\rho(g)\in \OO(2),\;\sigma(g)\in S_n$ and $\tau(g)\in\Sonehat$.  

\begin{definition}
A subgroup $G<\Gamma$ is said to be \defn{non-reversing} if  $\tau(G)<\Sone$, otherwise it is \defn{reversing}.  
\end{definition}

These are what Ferrario and Terracini call symmetry groups of cyclic and dihedral type, respectively \cite{FT04}. 
Note that their `brake type' symmetry groups cannot occur in collision-free (simple) choreographies with more than 1 particle.

\begin{example} \label{ex:circular motion}
One important---but in a sense trivial---class of choreography is what we call the \defn{circular choreographies}, where the particles lie at the vertices of a regular $n$-gon which rotates uniformly about the centre of mass; they are the generalizations of the Lagrange solution to $n$ particles.  Explicitly, consider the parametrized circle $z(t)=\exp(2\pi\ii t)$ in the plane. Let $\ell$ be an integer coprime to $n$ and define a motion of $n$ particles by, for $j=1,\dots,n$,
\begin{equation} \label{eq:circular}
z_j(t) = \ee^{2\pi\ii\ell j/n}\,z(\ell t).
\end{equation}
This is easily seen to be a choreography, satisfying \eqref{eq:choreography defn}. The full symmetry group of this motion is isomorphic to a semidirect product $\OO(2)\ltimes\bb{Z}_n\simeq(\SO(2)\times\bb{Z}_n)\rtimes\bb{Z}_2$; an explicit description of the elements of this group is given in Eq.~(\ref{eq:speed ell}).  
In this motion, particle $j$ immediately follows particle $j+m$ around the circle, where $m\ell=1\bmod n$, with a time delay of $1/\ell m$; that is
$$z_{j+m}(t) = z_j(t+\nfrac{1}{m\ell}),$$
which implies the choreography condition (\ref{eq:choreography defn}).   
In this example, we have
$$\ker\rho=\Chor_n,\quad \ker\sigma\simeq\SO(2),\quad \ker\tau= \left<(R_{2\pi\ell/n},\sigma_1,0)\right>\simeq\bb{Z}_n.$$  
We assume $\ell$ is coprime to $n$ for otherwise this motion involves particles coinciding for all time. 
\end{example}

\paragraph{Classification} The main result of the first part of the paper (stated as Theorem~\ref{thm:classification}) is a complete classification of all possible symmetry groups of (simple) planar choreographies. In other words, we classify all those subgroups of $\OO(2)\times S_n\times \Sonehat$ that on the one hand contain $\Chor_n$ and on the other are realized as the symmetry group of \emph{some} collision-free $n$-body motion.   For a given number $n$ of particles, one finds that there are two infinite families of symmetry group and, if $n$ is odd, three exceptional symmetry groups.  Full details of the symmetry groups are given in Section~\ref{sec:classification}; here we give a brief description.  For the infinite families, the curve on which the particles move has the symmetry of a regular $k$-gon for some $k\geq1$.  As the particles move, they visit the `vertices' of the $k$-gon in some order. This order is similar to the difference between a pentagon and a pentagram: in the former the vertices are visited in geometric order, while in the latter the vertices are visited alternately (i.e., in the order 1, 3, 5, 2, 4, rather than 1, 2, 3, 4, 5).  There is a convenient notation used to distinguish these, the so-called Schl\"afli symbol. In this notation, the basic regular convex $k$-gon is denoted $\{k\}$, while the $k$-gon with every $\ell^{\mathrm{th}}$ vertex visited in sequence is denoted $\{k/\ell\}$ ---in particular $\{k/1\}=\{k\}$. Thus the pentagon is denoted $\{5\}$, while the pentagram is denoted $\{5/2\}$.  In order for the geometric object $\{k/\ell\}$ to consist of a single closed curve, it is necessary and sufficient that $k$ and $\ell$ are coprime (which we write throughout as $(k,\ell)=1$). 
We adapt this notation, and denote the symmetry groups for $n$ particles moving on a  curve of type $\{k\}$ or $\{k/\ell\}$ by $C(n,k)$ or $C(n,k/\ell)$ respectively if there is no time reversing symmetry, and by $D(n,k)$ or $D(n,k/\ell)$ respectively if there is such a symmetry.  See for example Fig.~\ref{fig:ell comparison} for choreographies illustrating the difference between $D(6,5)$ (pentagon) and $D(6,5/2)$ (pentagram).  In these infinite families, the time-reversing symmetries occur in combination with a reflection in the plane. As $k\to\infty$ the $k$-gon tends to a circle and so we denote the symmetry of the circular choreography introduced above by $D(n,\infty/\ell)$.

On the other hand, if $n$ is odd, there are three exceptional symmetry groups denoted $C'(n,2)$, $D'(n,1)$ and $D'(n,2)$ (where, as always, $n$ is the number of particles). For example, the figure eight choreography has symmetry $D'(3,2)$.  In these groups there is always an element which either acts as time-reversing symmetry but not a reflection (in $D'(n,1)$), or a reflection which is not acting as time-reversal in $C'(n,2)$, or both in the case of $D'(n,2)$.  Precise details are given in Section~\ref{sec:classification}, while a number of different choreographies are illustrated in Figure~\ref{fig:many choreographies}. 

\paragraph{Comparison to recent literature}  Stewart \cite{Stewart96} gives a classification of symmetry groups arising in many body problems which is different from ours. This difference arises for two reasons: firstly Stewart does not restrict attention to choreographies, and secondly his approach is local, and would apply to Hopf bifurcation or Lyapunov centre theorem scenarios: they are the symmetries that can arise for periodic orbits in a linear system.  On the other hand, Barutello, Ferrario and Terracini \cite{BFT08} do give a classification of symmetries for 3-body choreographies.  However, their classification is simpler than ours (even for $n=3$) as they consider the motion in a rotating frame (or modulo rotations), which has the effect of projecting out the rotational part of our symmetry groups, so effectively they consider subgroups of $\ZZ_2\times S_3\times\Sonehat$; in particular all the groups $D(3,k/\ell)$ in this paper are collapsed to the single `Lagrange type', which is the image of $D(3,\infty/\ell)$ in $\ZZ_2\times S_3\times\Sonehat$. 
See Table~\ref{table:comparison}. This is similar to viewing the motion on the shape sphere, see for example \cite{Montgomery-AMS} and references therein.

\begin{table}
\centering
\begin{tabular}{@{}lllll@{}}
\toprule
  This paper & $D(3,k/\ell) $ & $D'(3,2)$ & $D'(3,1)$ & $C'(3,2)$ \\
  Barutello et al. & `Lagrange'  & $D_6$ & $D_3$  & $C_6$ \\
\bottomrule
\end{tabular}
\caption{Comparison of  the classification in this paper with that of Barutello, Ferrario and Terracini \cite{BFT08} for 3-particle choreographies (note that in \cite{BFT08} the dihedral group of order $2m$ is denoted $D_{2m}$ while here we denote it $D_m$)}
\label{table:comparison}
\end{table}

\begin{figure}[p]
\begin{center}
\subfigure[$D(6,4)$]{\includegraphics[scale=0.16]{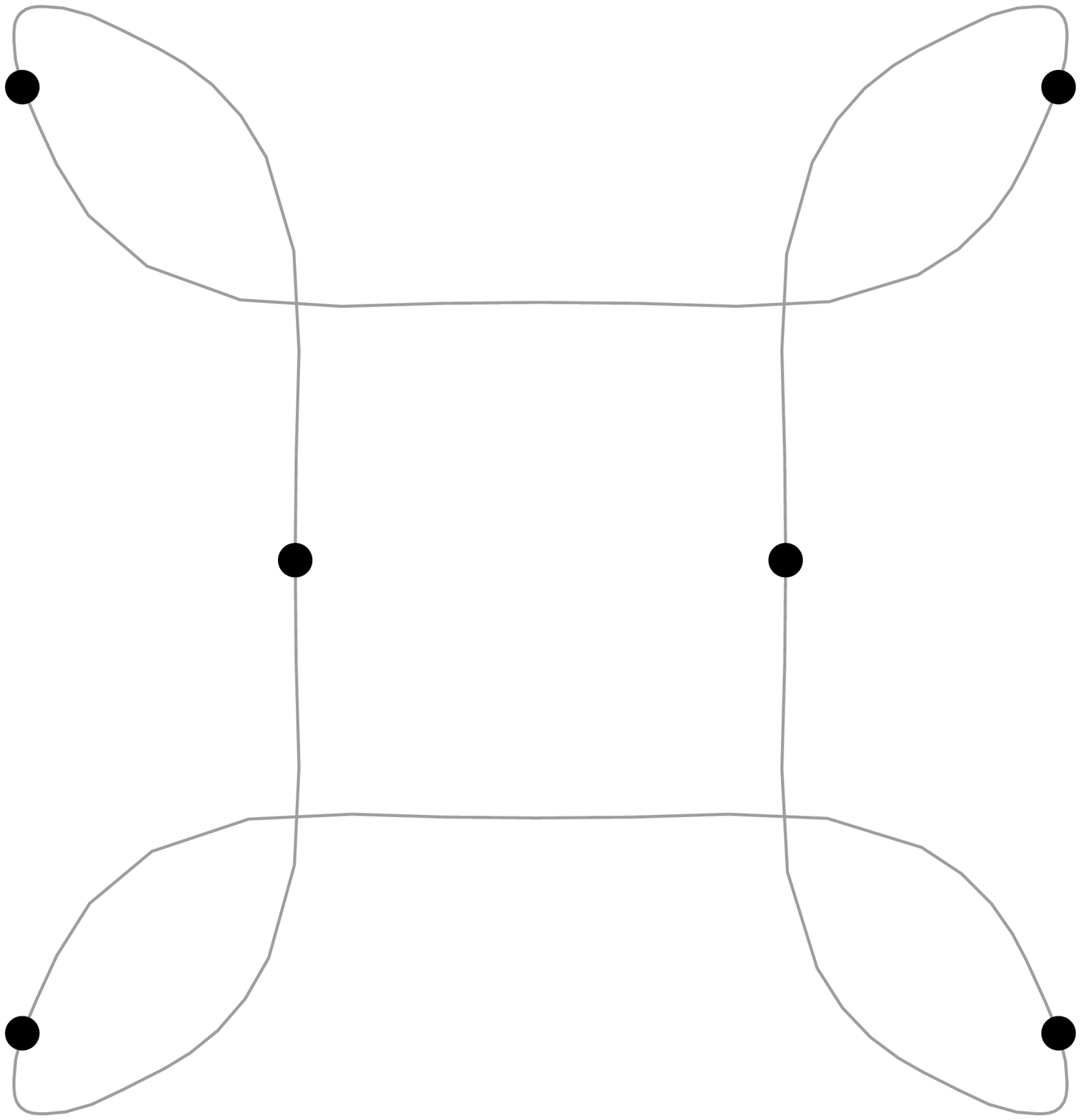}}
 \qquad 
\subfigure[$D(6,4)$]{\includegraphics[scale=0.16]{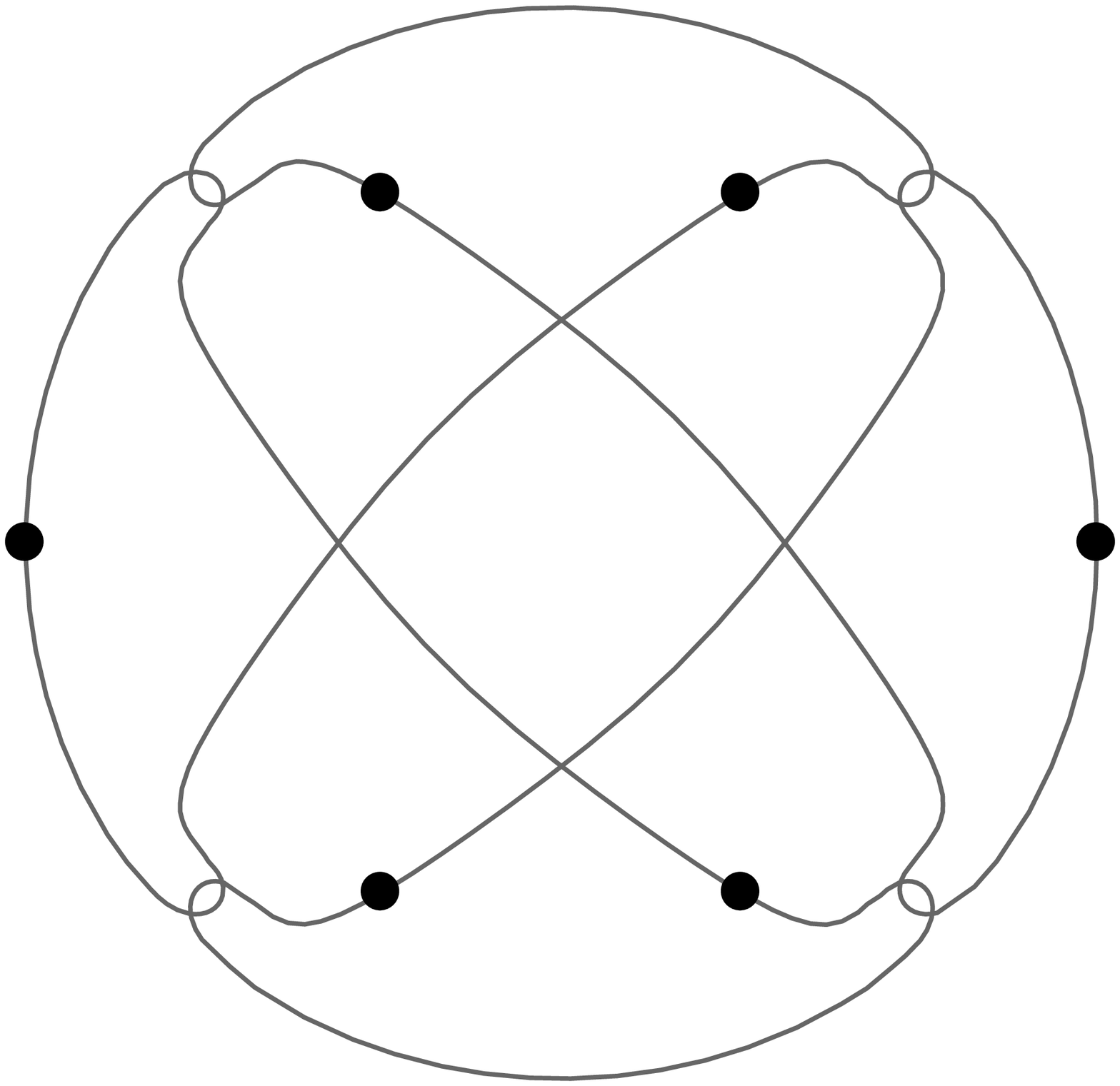}}
 \qquad 
\subfigure[$D(6,4)$]{\includegraphics[scale=0.16]{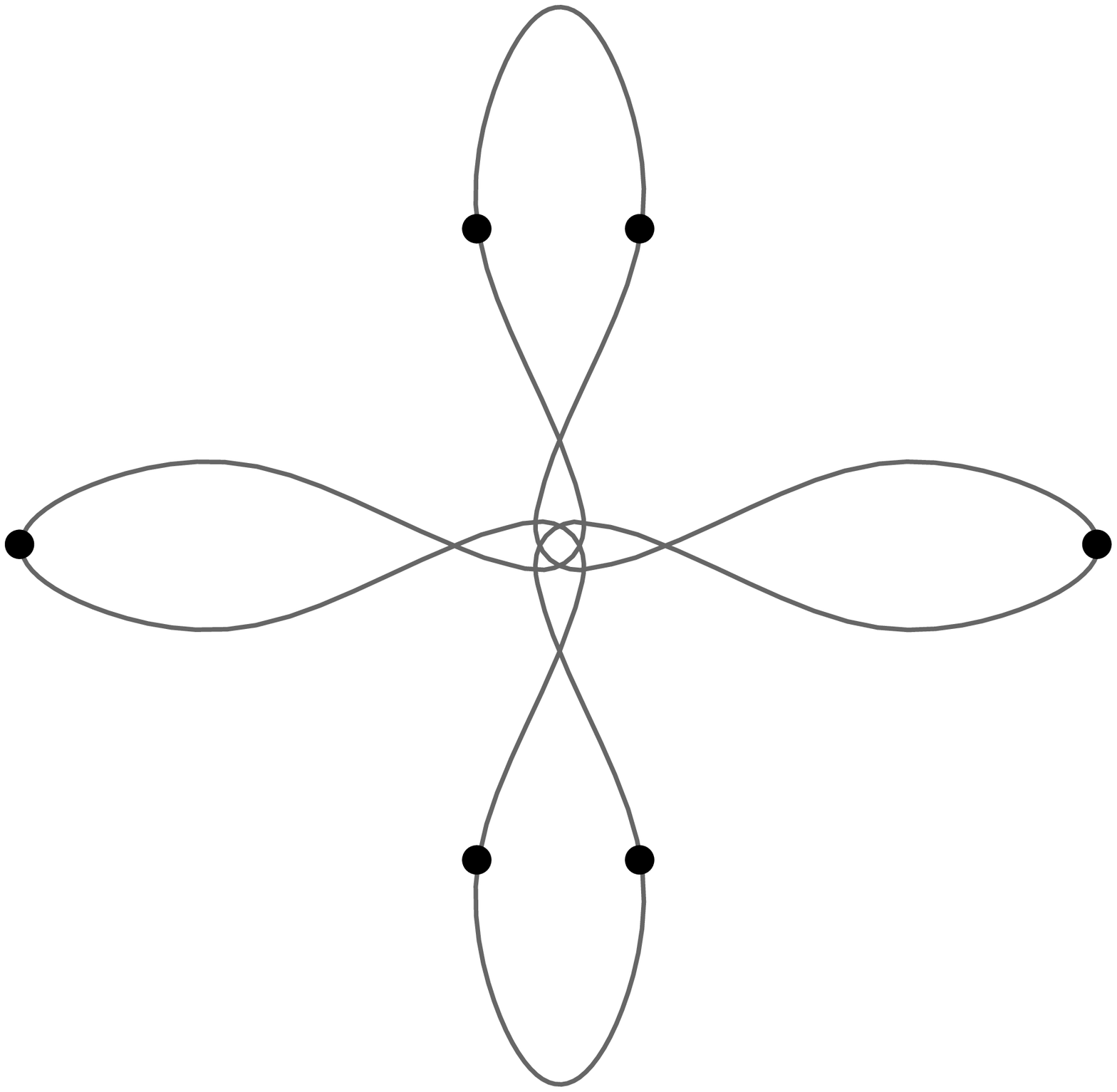}}
 \quad
\subfigure[$D(6,4)$]{\includegraphics[scale=0.16]{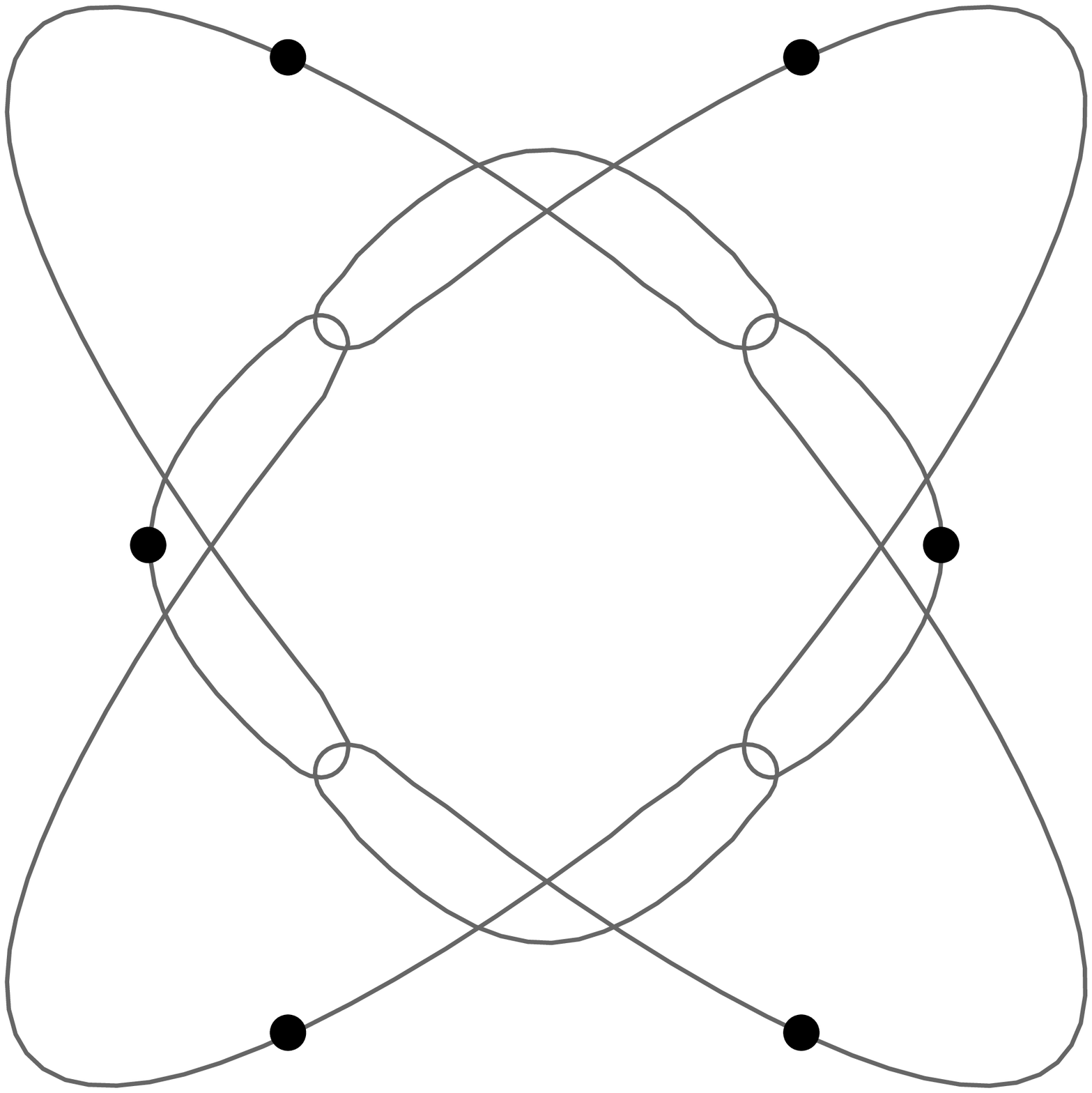}}
 \\[12pt] 
\subfigure[$D(8,3)$]{\includegraphics[scale=0.16]{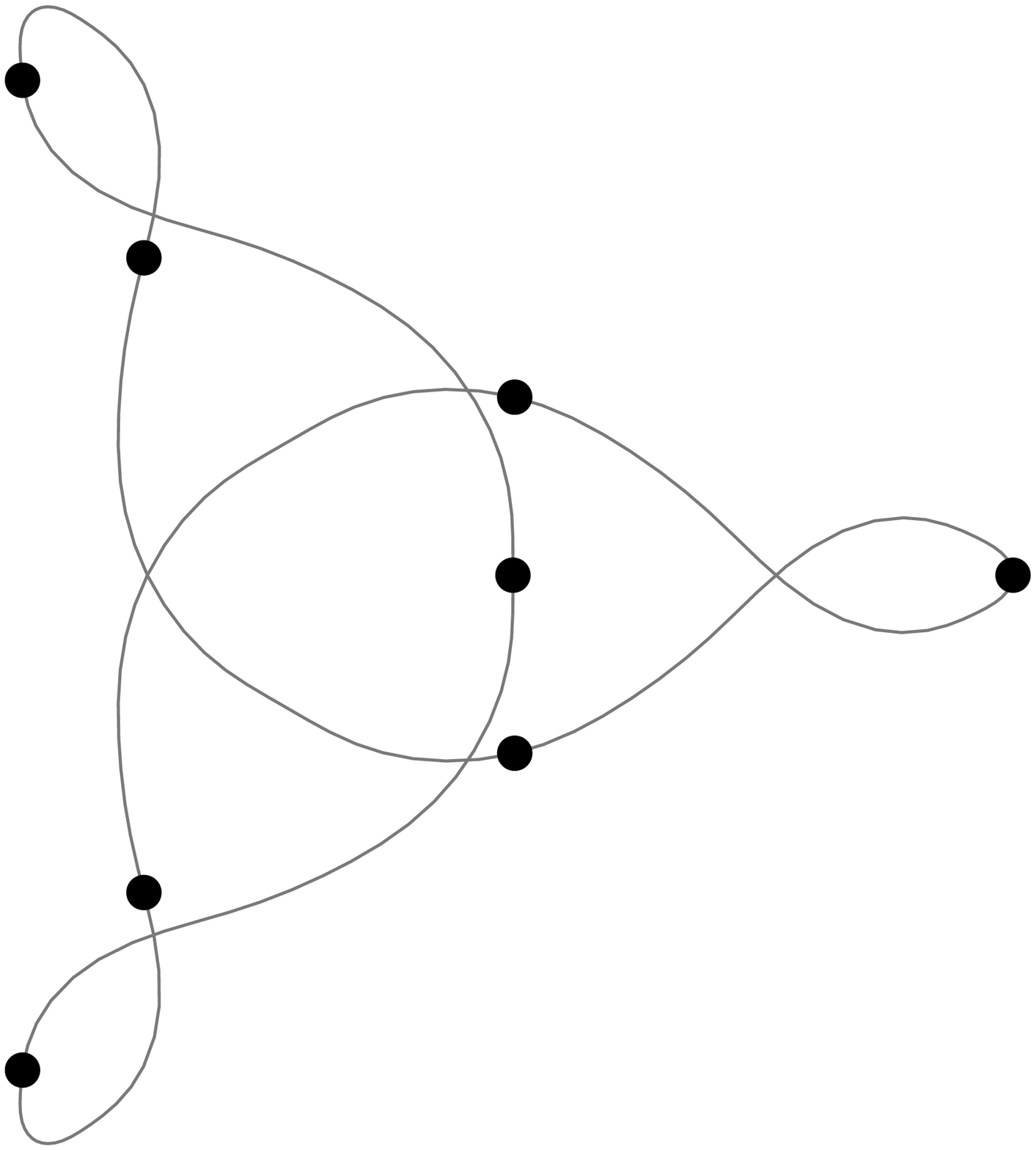}}
 \qquad 
\subfigure[$D(8,3)$]{\includegraphics[scale=0.16]{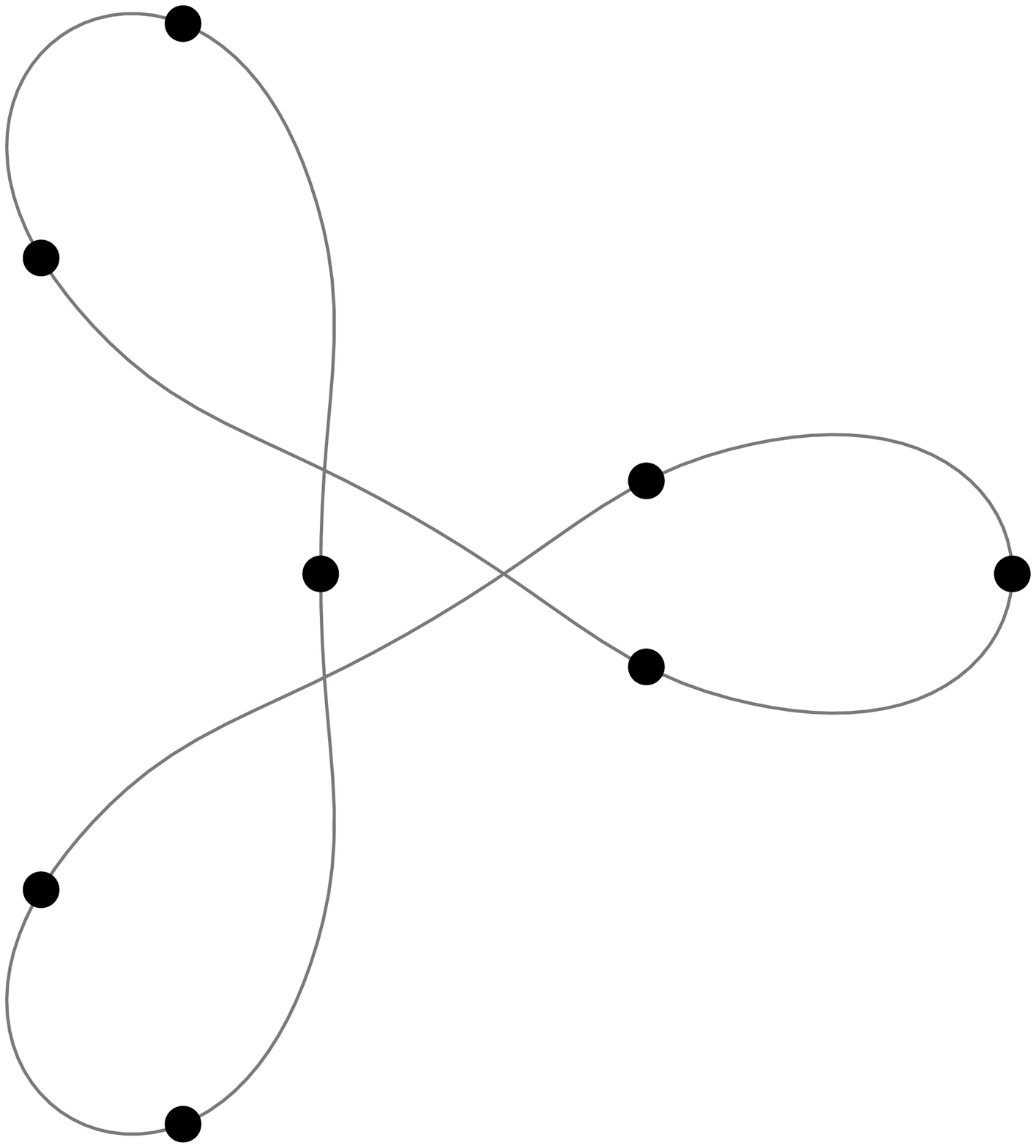}}
 \qquad 
\subfigure[$D(8,3)$]{\includegraphics[scale=0.16]{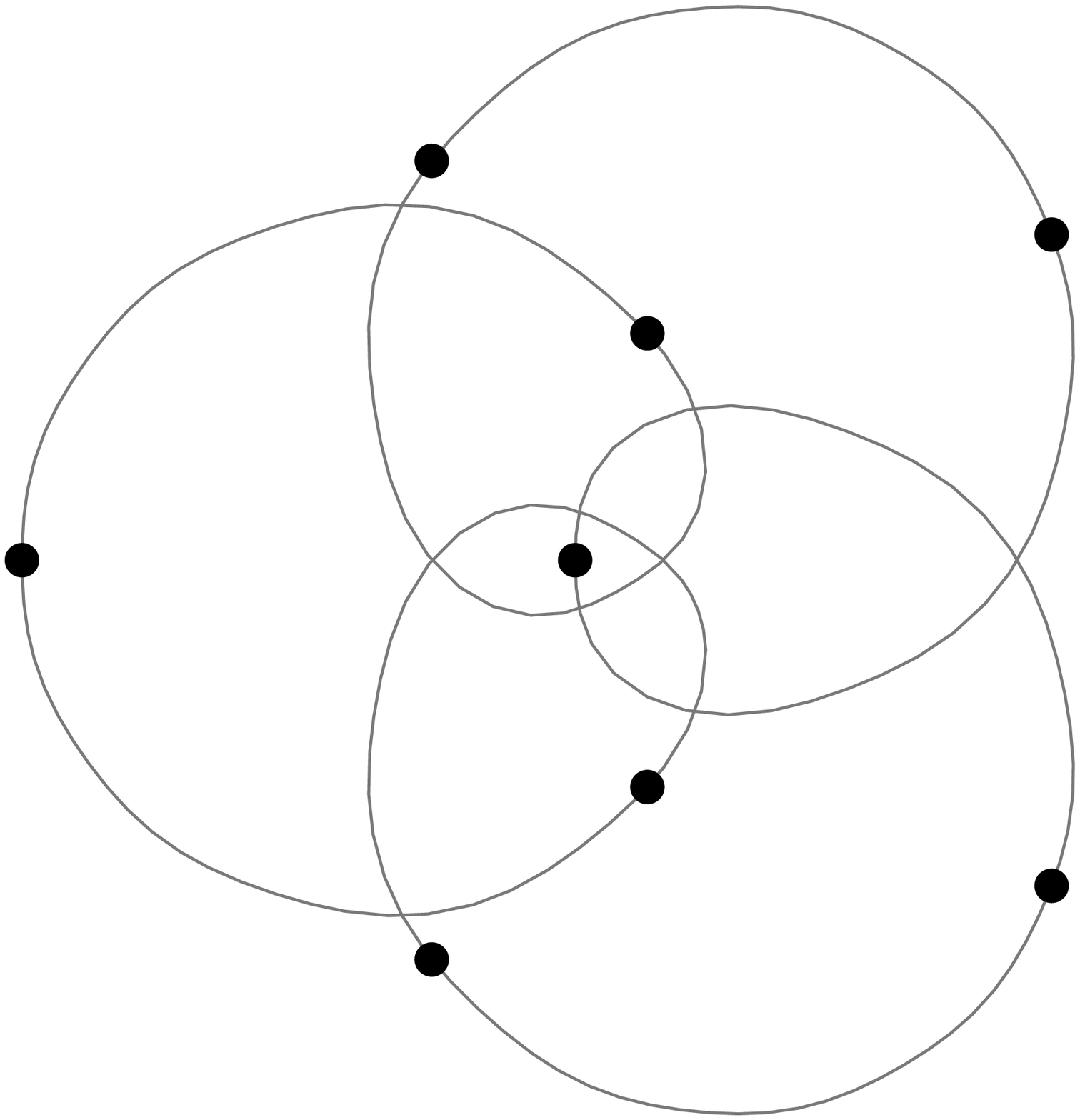}}
 \qquad 
\subfigure[$D(8,3)$]{\includegraphics[scale=0.16]{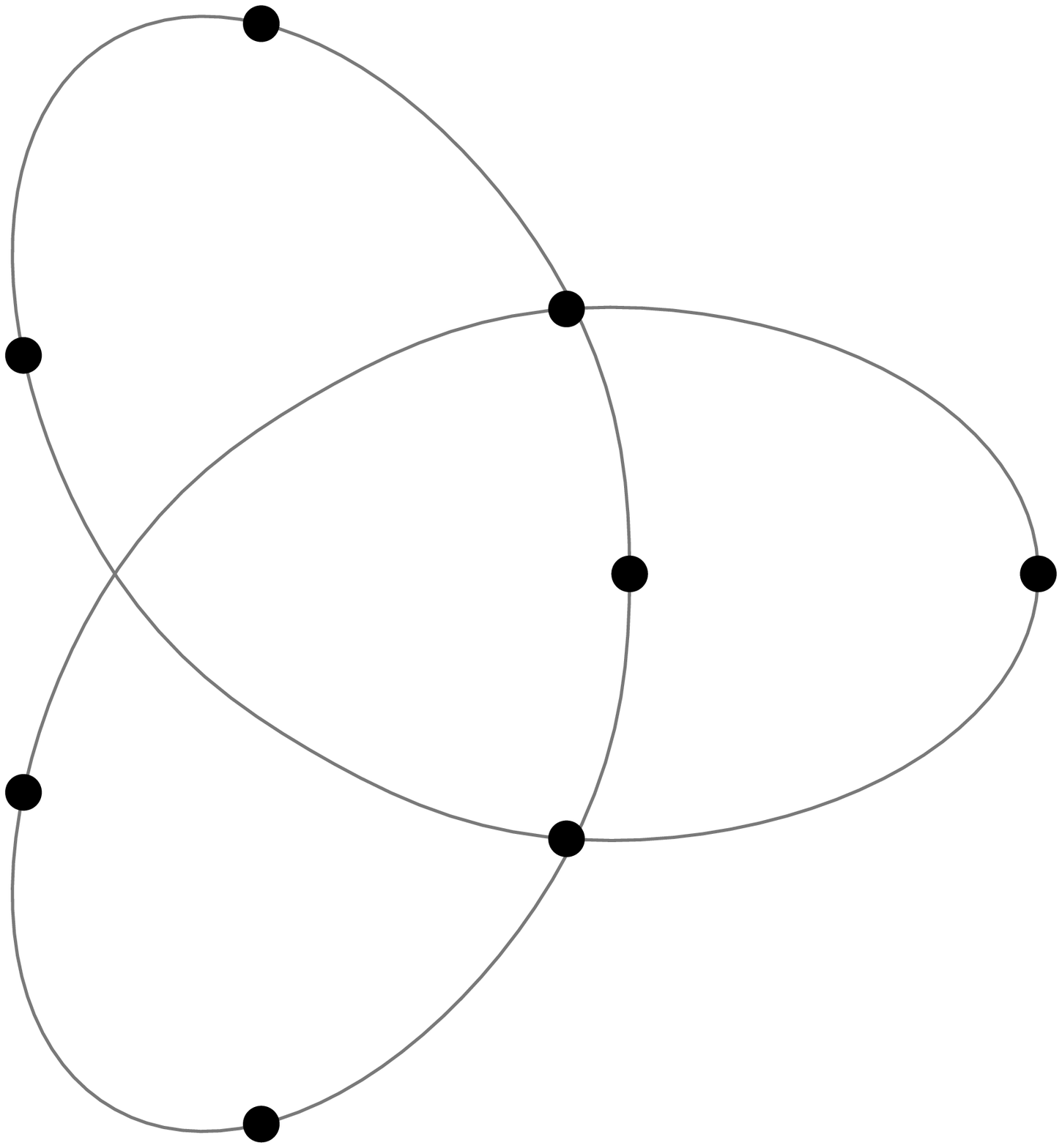}}
\\[12pt] 
\subfigure[$D(4,6)$]{\includegraphics[scale=0.16]{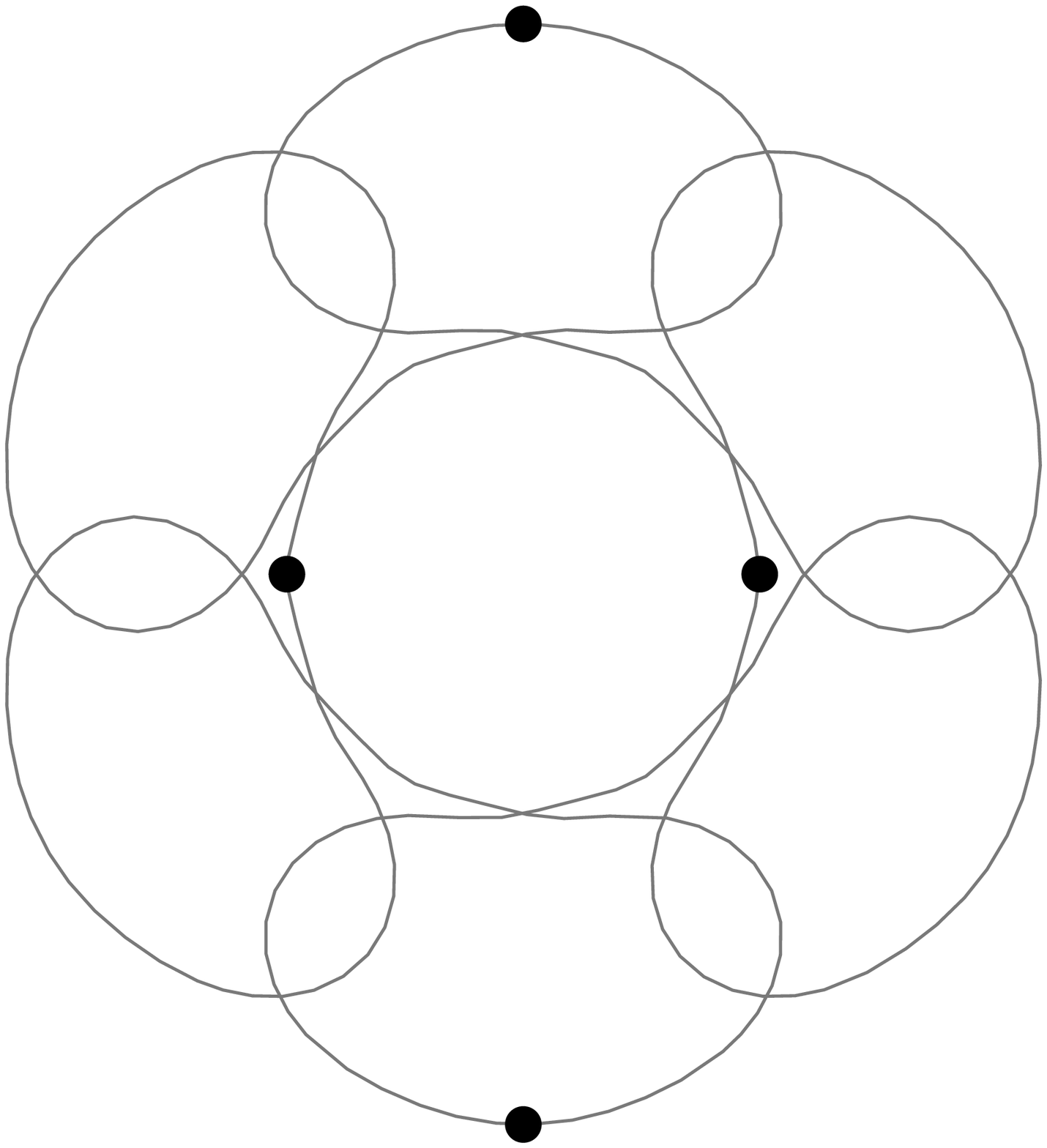}}
 \qquad
\subfigure[$D(5,8)$]{\includegraphics[scale=0.16]{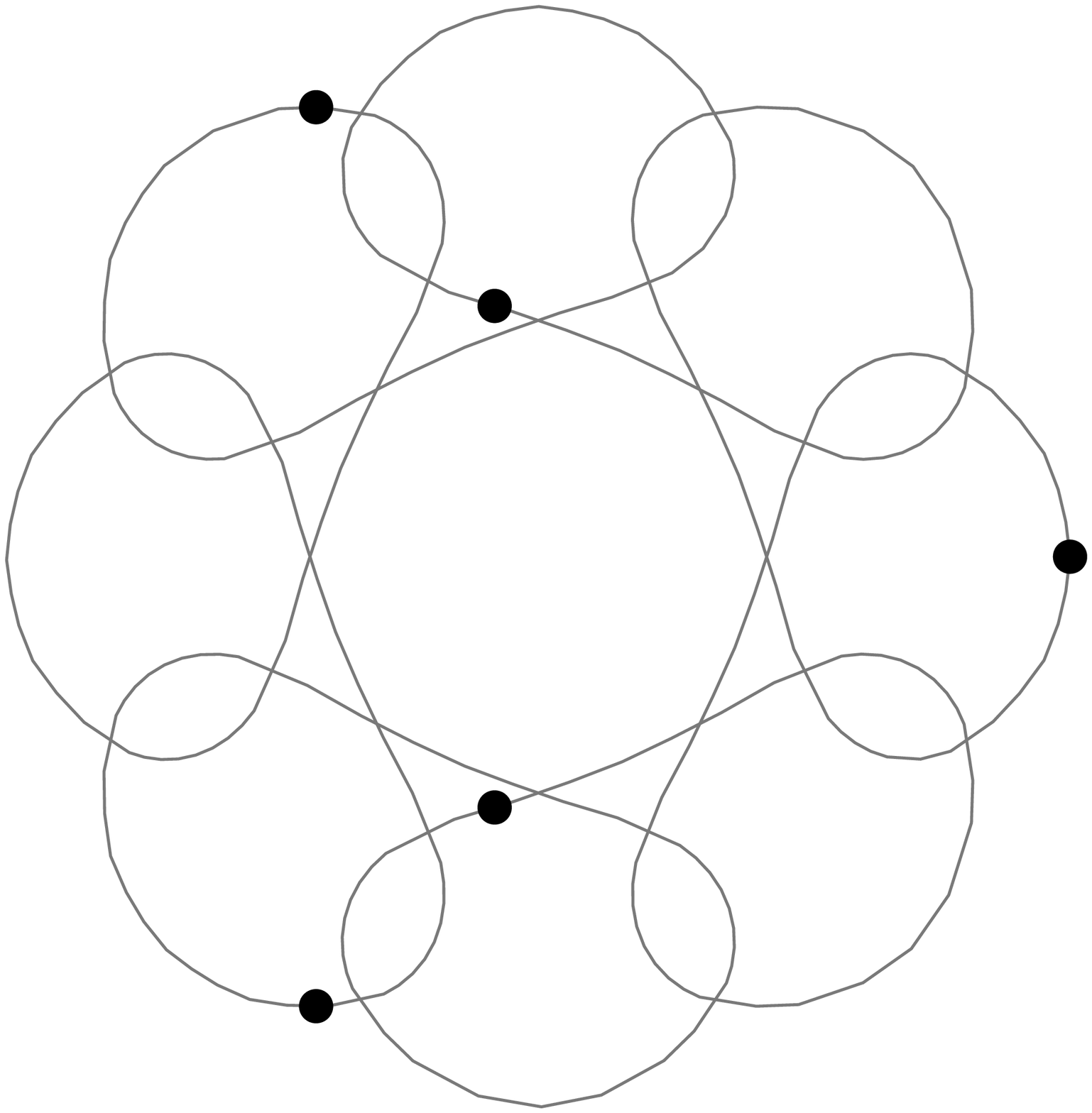}}
 \qquad
\subfigure[$D(8,9/2)$]{\includegraphics[scale=0.16]{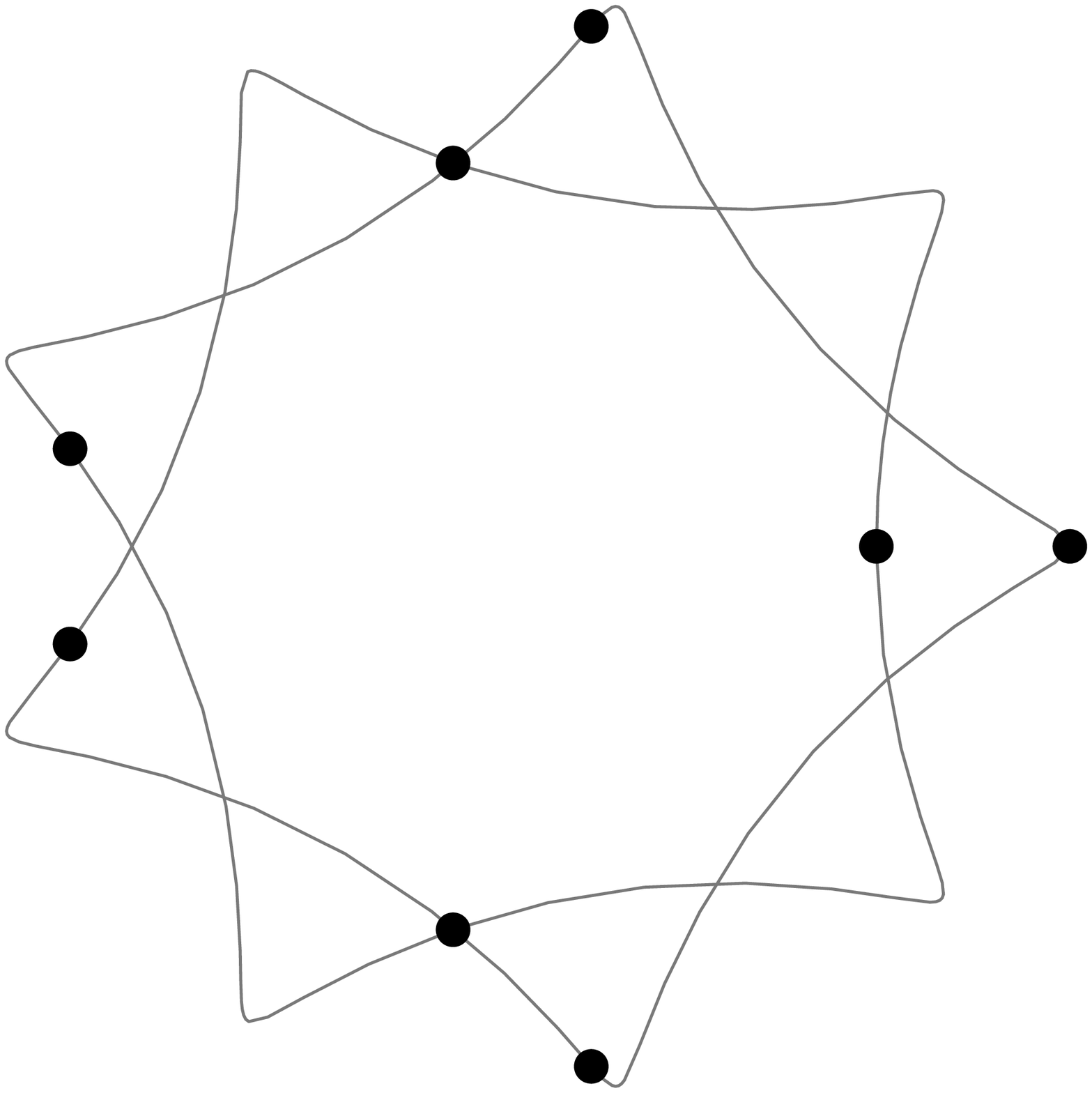}}
 \qquad 
\subfigure[$D(8,9/4)$]{\includegraphics[scale=0.16]{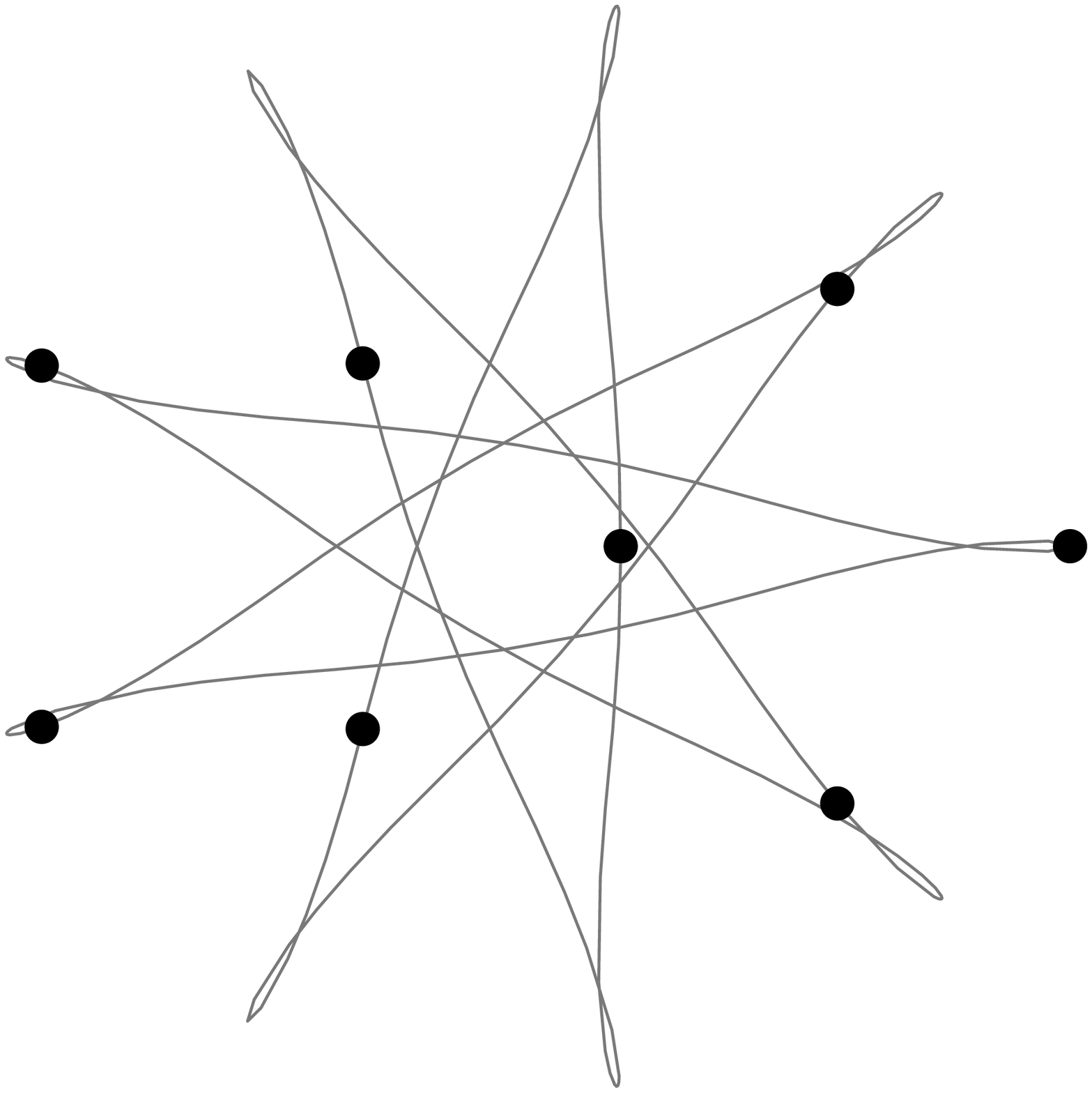}}
 \\[12pt]
\subfigure[$D(8,7)$]{\includegraphics[scale=0.16]{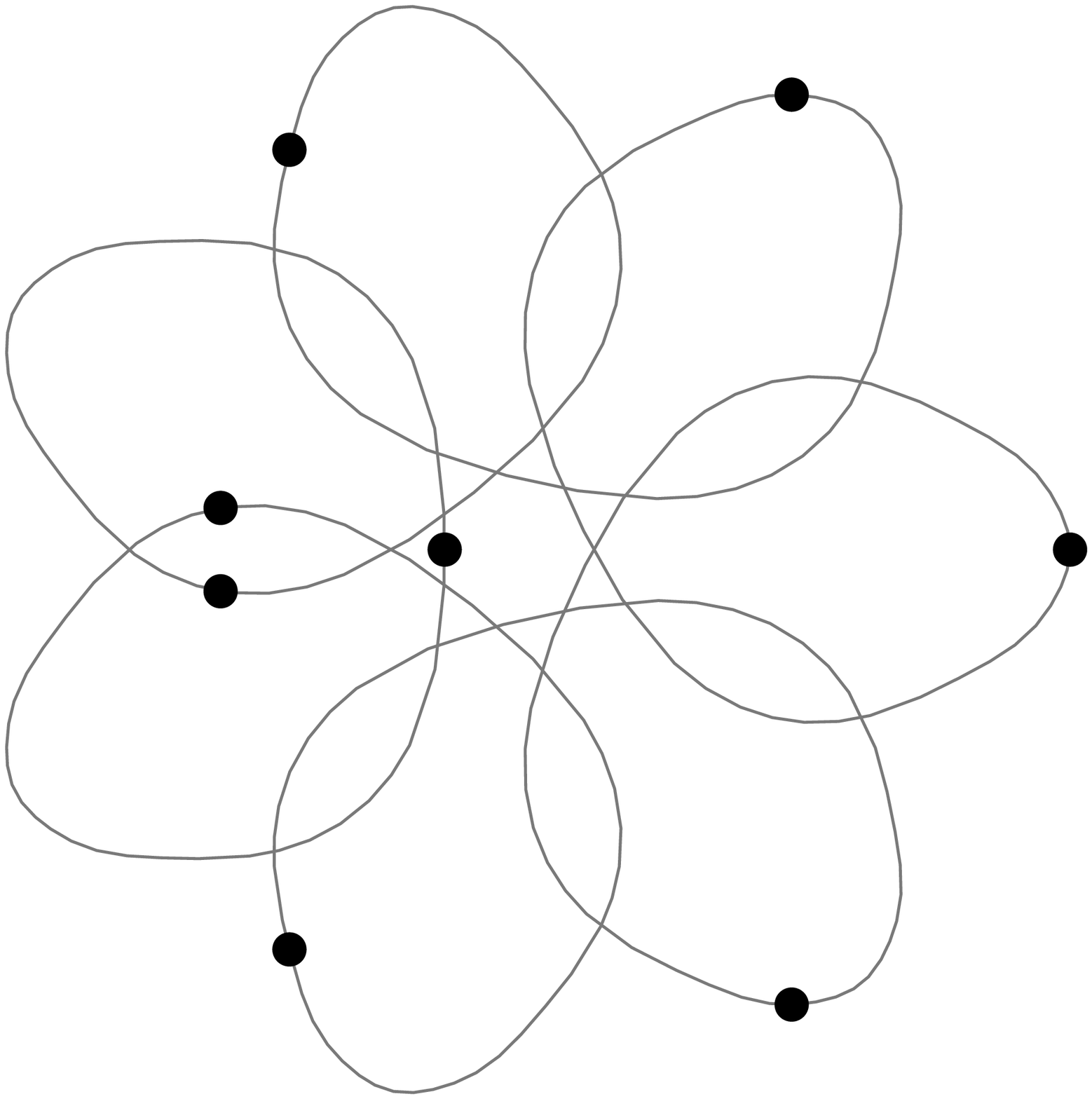}}
 \qquad 
\subfigure[$D(9,4)$]{\includegraphics[scale=0.16]{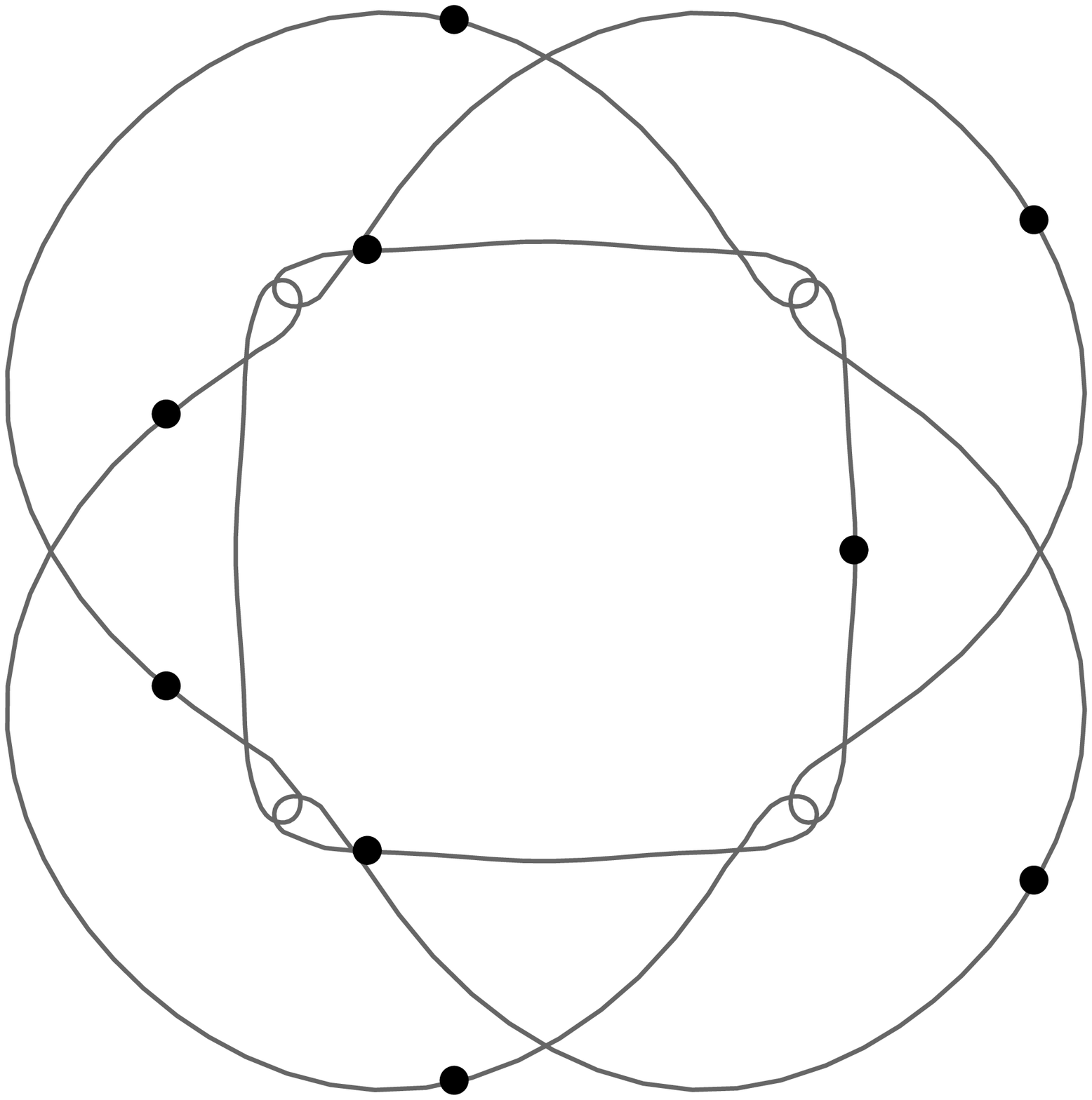}}
 \qquad 
\subfigure[$D(9,4)$]{\includegraphics[scale=0.16]{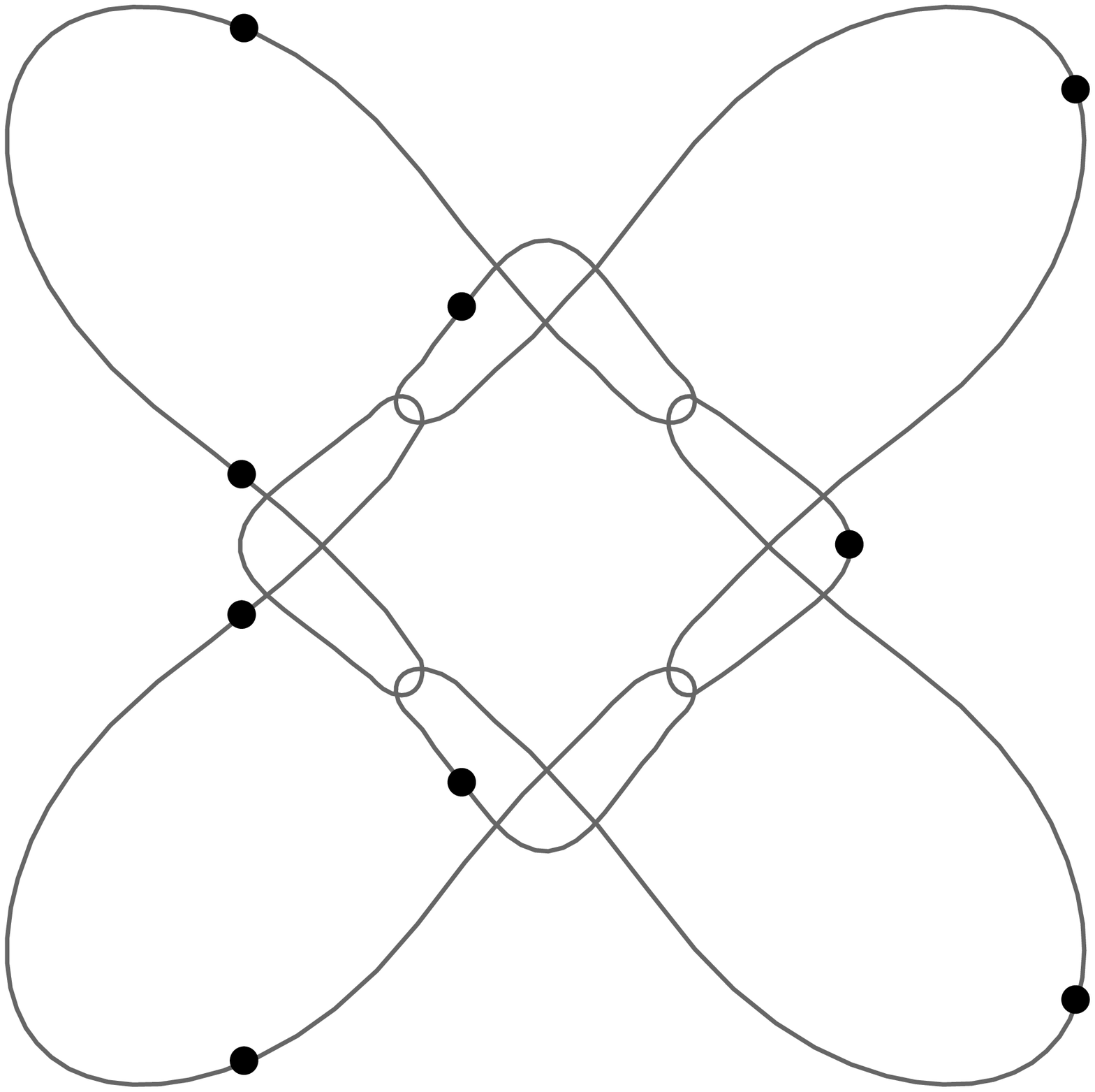}}
 \qquad 
\subfigure[$D(10,5/2)$]{\includegraphics[scale=0.16]{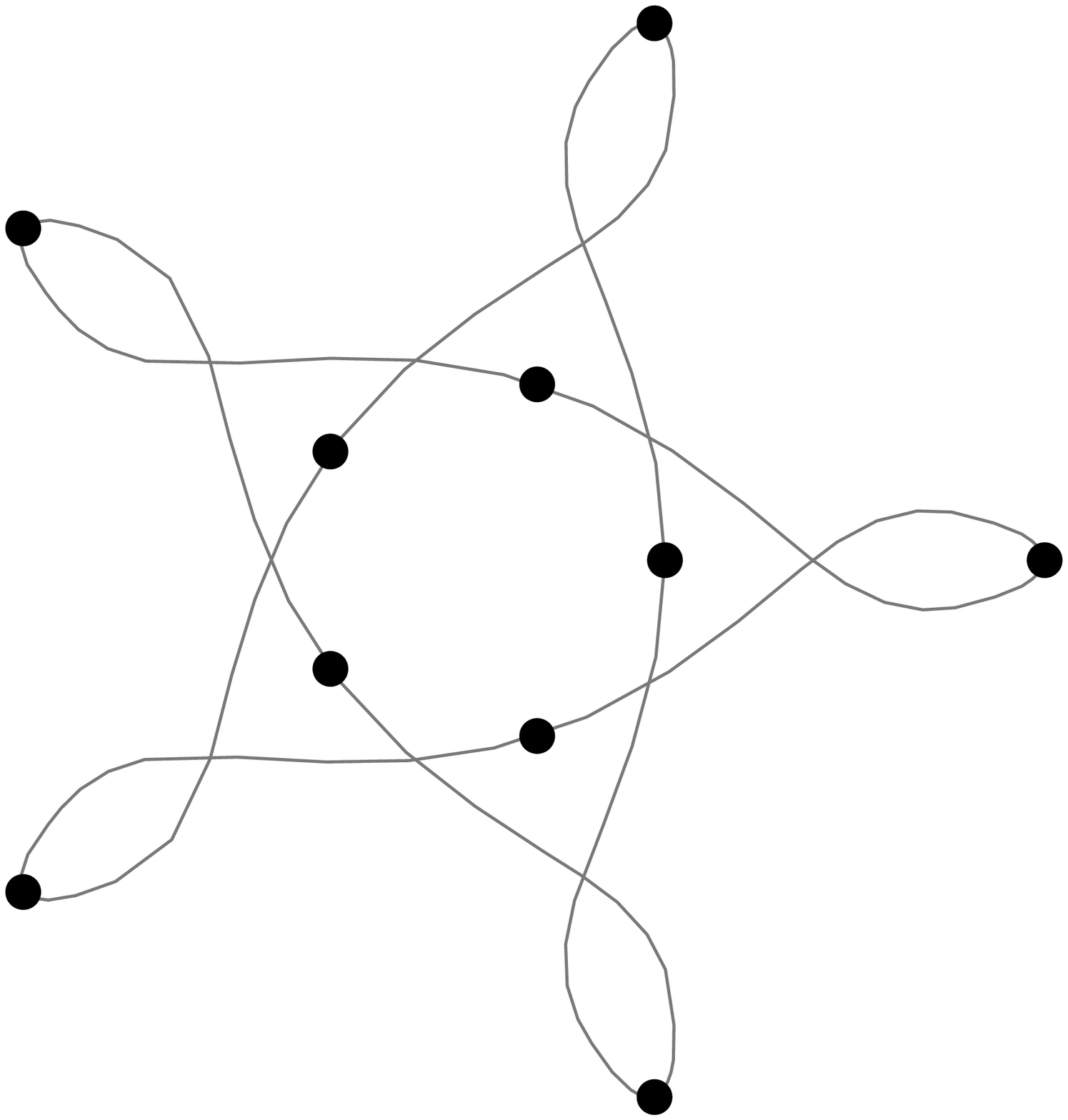}}
\end{center}
\caption{A selection of choreographies with their symmetry group.  See Remark~\,\ref{rmk:figures}.  These are all from the regular families of symmetry groups; for some examples illustrating the exceptional symmetry groups see Fig.~\ref{fig:exceptionals}. 
Where the same label appears for more than one figure, they correspond to different connected components of the corresponding $\Fix(G,\,\Lambda\Xn)$.
Animations can be viewed on the first author's website \cite{JMweb}.}
\label{fig:many choreographies}
\end{figure}

\subsection{Variational problem and topology of loop spaces}
\label{sec:variational}

The principal motivation for this paper is to apply the results to the variational problem describing the periodic motion of $n$ identical particles in the plane, interacting under a Newtonian potential. However, the Newtonian potential is renowned for its difficulty, as was evident even to Poincar\'e (in modern terms, because the action functional on the space of collision-free loops is not coercive). The proof of the existence of the figure-8 by Chenciner and Montgomery \cite{Fig8} uses some delicate arguments to show that the minimum of the action functional over the set of loops with the given symmetry cannot occur for a loop with collisions.  A different argument for avoiding collisions covering more general symmetry classes was given by Ferrario and Terracini \cite{FT04}, under the hypothesis that the symmetry group satisfy what they call the rotating circle condition (see Section\,\ref{sec:rcc} below).

Given an action of a (Lie) group $\Gamma$ on a manifold $X$ there is a natural action of $\Gamma\times\Sonehat$ on the loop space $\Lambda X$.  This action has been used a great deal in bifurcation theory, in particular for the Hopf bifurcation by Golubitsky and Stewart \cite{Golubitsky-Stewart_Hopf} and for the Hamiltonian Lyapunov centre theorem \cite{MRS88}.  However, it seems it has not been used as extensively, or as systematically, in variational problems.  The second half of this paper goes a little way to address this. 

Typically, one is looking for periodic solutions of a differential equation which can be expressed as a variational problem.  This includes the existence problem of closed geodesics, as well as periodic orbits for $n$-body problems and more general Lagrangian mechanical systems. Let $\A:\Lambda X\to\RR$ be the `action functional', whose critical points correspond to periodic solutions of a given fixed period, which we take to be 1, and assume it is invariant under the action of $\Gamma\times\Sonehat$. (This invariance occurs for example if the Lagrangian is invariant under the action of $\Gamma$ on $X$, or for the geodesic problem, if the metric is invariant under the $\Gamma$ action).  For each subgroup $G<\Gamma\times\Sonehat$ denote by $\A^G$ the restriction of $\A$ to $\Fix(G,\,\Lambda X)$. By Palais' principle of symmetric criticality \cite{Palais-PSC}, critical points of $\A^G$ coincide with critical points of $\A$ lying in $\Fix(G,\,\Lambda X)$, that is, to periodic solutions with spatio-temporal symmetry $G$.  

If the functional $\A^G$ is coercive, then it is guaranteed to achieve a minimum, and indeed a minimum on each connected component of $\Fix(G,\,\Lambda X)$.  Coercive means that for every sequence that has no point of accumulation (in the weak topology),  the functional tends to infinity, and it is a standard argument in variational calculus that provided the functional is lower semicontinuous and coercive then it achieves its minimum, see for example the book of Jost and Li-Jost \cite{Jost-Li}.  

For the geodesic problem on a compact Riemannian manifold $X$, the action functional, equal to the length of a loop, satisfies the Palais-Smale condition and is coercive, and the critical points are the closed geodesics \cite{Klingenberg}.   The topological techniques of this paper can be used to prove the existence of symmetric geodesics, so those satisfying
$u(t+\theta)=g\cdot u(t)$, for each $(g,\theta)\in G<\Gamma\times\Sone$. In particular, a symmetric geodesic is one that, as a curve in $X$, is invariant under those transformations of $X$ contained in the projection of $G$ to $\Gamma$.

For planar $n$-body problems, the space $X$ is $\Xn$ introduced above, which is not compact, nor even complete because of collisions, and separate arguments are required to deal with the two problems.  

Its completion $\overline{\Xn}\simeq\CC^{n-1}$ is not compact and the action functional is not coercive, as loops can move to infinity without the action increasing.  However, imposing restrictions on the types of loops considered can ensure coercivity of $\A$, and there are two types of restriction considered in the literature: topological and symmetry based. The topological constraints were introduced by Gordon \cite{Gordon-StrongForce} using the notion of \emph{tied} loops.  The symmetry approach was used in various ways by different authors and culminated in a beautifully simple result of Ferrario and Terracini \cite{FT04}, who showed that the restriction $\A^G$ of $\A$ to the subspace of loops in $\overline{\Xn}$ with symmetry $G$ is coercive 
if and only if $G<\Gamma\times\Sonehat$ is such that 
\begin{equation}\label{eq:isolated fixed point}
\Fix(G,\,\overline{\Xn})=\{0\}.
\end{equation}
This holds for a wide class of action functions $\A$, including the one derived from the Newtonian potential. For our purposes this condition \eqref{eq:isolated fixed point} holds for the choreography group $\Chor_n$, and a fortiori for any group containing $\Chor_n$ (the groups of our classification). 

There remains the issue of collisions. For the gravitational $1/r$ potential the action functional on $\Lambda\Xn$ fails to be coercive because, as was known to Poincar\'e, there are trajectories with collisions for which the action is finite.  Following Poincar\'e and others since, one can introduce the notion of a strong force (essentially with potential behaving like $1/r^a$ for $a\geq 2$ near collisions, rather than the Newtonian $1/r$), in which case a simple estimate shows that every loop with collisions has infinite action.  This idea was investigated by Gordon \cite{Gordon-StrongForce} where he combines it with his idea of tied loops in $\Lambda\Xn$ to ensure $\A$ is coercive on these connected components of $\Lambda\Xn$.  (See also the very interesting papers of Chenciner \cite{Chenciner-note-by-P,Chenciner-Poincare} describing the insights and contributions of Poincar\'e.)

This discussion leads to the following well-known result. 

\begin{theorem}\label{thm:SF}
Consider the $n$-body problem with a strong force potential and let $G$ be any subgroup of\/ $\Gamma\times\Sonehat$ containing $\Chor_n$.  Then in each connected component of\/ $\Fix(G,\,\Lambda\Xn)$ there is at least one choreographic periodic orbit of the system.
\end{theorem}

\begin{proof}
The strong force analysis by Gordon \cite{Gordon-StrongForce} and the coercivity result of Ferrario and Terracini \cite{FT04} mentioned above implies that each connected component of $\Fix(G,\,\Lambda\Xn)$ contains a local minimum of the action functional. This minimum is a periodic orbit, and necessarily a choreography since $G$ contains $\Chor_n$. 
\end{proof}

Connected components of the loop space $\Lambda\Xn$ correspond to (conjugacy classes of) pure braids and Montgomery  \cite{Montgomery-Braids} gives a very nice analysis of which components of $\Lambda\Xn$ are tied in Gordon's sense, in terms of the pure braids and their winding numbers. In the second half of this paper, we show that for each group $G<\Gamma\times\Sone$ (so not including time reversing symmetries which will be dealt with in a separate paper), the connected components of  $\Fix(G,\,\Lambda\Xn)$ are in 1--1 correspondence with certain conjugacy classes in the full braid group, and more precisely by the $P_n$-conjugacy classes in a certain coset in $B_n/P_n$, where $P_n$ is the pure braid group on $n$ strings (or twisted conjugacy classes in the case of $C'(n,2)$).  The precise formulation is given in Theorem~\ref{thm:connected components of choreographies}. In a sense, this can be seen as extending the work of Montgomery.

For the Newtonian (weak) potential, the action functional is not coercive at collisions, and the notion of tied loops does not apply as the set of collisions does not obstruct moving from one component of $\Lambda\Xn$ to another.  The proof by Chenciner and Montgomery \cite{Fig8} of the existence of the figure-8 solution involves showing that the minimum over all loops with collision is  greater than the action of a particular loop with the given symmetry (the class $D'(3,2)$ in our notation), and hence the minimum  over loops with that symmetry must be realized for a collision-free loop.  The important paper of Ferrario and Terracini \cite{FT04} gives a general perturbation argument, based on a technique of Marchal, showing that for many symmetry classes the minimum of $\A^G$ cannot be achieved at a trajectory with collisions. These symmetry classes are those satisfying their \emph{rotating circle condition}, a property we discuss in Sec.~\ref{sec:rcc}. 

However, even with the gravitational potential, the lack of coercivity does not of course imply that there is not a minimum on each connected component of $\Fix(G,\,\Lambda\Xn)$, and indeed numerics suggest that in many, or perhaps most, examples there are such minima (there is numerical evidence that on some connected components there is no minimum, see \cite{Simo01b,Simo01a}, but this evidence is also only numerical).  

In this paper, we make no claim to prove explicitly the existence of new choreographies for the Newtonian $n$-body problem, although we can make the following statement, which is an easy consequence of the results of Ferrario and Terracini.

\begin{theorem} Suppose $n$ is odd. Then for each of the symmetry types $D'(n,1)$ and $C'(n,2)$ there is a collision-free periodic orbit of the Newtonian $n$-body system with that symmetry. 
\end{theorem}

It is possible that the choreographies in question are those where $n$ particles move around a figure-8 curve---however to our knowledge it has not been shown that these minimize the action for the given symmetry type.

\begin{proof}
Ferrario and Terracini \cite{FT04} prove that for any symmetry satisfying the rotating circle condition there is a collision-free minimum in the set of loops with that symmetry.  We show in Proposition~\ref{prop:rcc} that the symmetry groups in question do satisfy this property.
\end{proof}

Note that the groups $C(n,k/\ell)$ also satisfy the rotating circle condition, but in that case it is known that the circular choreography minimizes the action \cite{BT04}.  

For the 3-body problem, many choreographies are known (numerically), almost all of which have just reflectional symmetry (many with one axis, and some with 2 such as the figure 8).  A new possibility raised here are the symmetry types $D(3,k)$ with $k>2$.  A particular case is the choreography of ``3 particles on a Celtic knot'', depicted in Figure~\ref{fig:D(3,4)},  with $D(3,4)$ symmetry.  While a motion similar to the figure exists for the strong force, by Theorem~\ref{thm:SF}, it would be particularly interesting to know if it exists for the Newtonian attraction. 

We have hitherto not been specific about exactly which space of loops we use.  For variational calculus one needs the Sobolev space $H^1(\bb{T},\Xn)$ with its usual topology, while for the topological part one uses continuous loops with the compact-open topology.  It is proved in \cite{Klingenberg} that the spaces are homotopic, and the argument can be adapted to show that the homotopy can be chosen to respect the  action of $\Gamma\times\Sonehat$.  Thus connected components of $\Fix(G,\,C^0(\bb{T},\Xn))$ correspond to those of $\Fix(G,\,H^1(\bb{T},\Xn))$, and the corresponding components are homotopy equivalent.

\paragraph{Organization} The paper is organized as follows.  In Section~\ref{sec:classification} we describe the symmetry groups arising for planar choreographies and some of their properties and state the classification theorem, Theorem~\ref{thm:classification}.  The proof of the theorem is the subject of Section~\ref{sec:main proof}.

The remainder of the paper investigates the topology of the space of choreographic loops with a given symmetry. Section~\ref{sec:components of loops space} describes a general approach to the question for the action of any Lie group $\Gamma$ on a manifold $X$.  We introduce the notion of `equivariant fundamental group' which can be used to compute the connected components and their fundamental group of the spaces of symmetric loops. These ideas are applied in Section~\ref{sec:choreography components} to the question for choreographies.  Finally, Section~\ref{sec:adjacencies} provides a method to describe which components of the space of loops with symmetry $G$ contain loops with symmetry group strictly greater than $G$, an important issue in applying variational techniques. In sections \ref{sec:components of loops space}--\ref{sec:adjacencies} we consider only the action of $\Gamma\times\Sone$, rather than allowing time-reversing symmetries in $\Gamma\times\Sonehat$.  Subgroups with time-reversing symmetries will be considered in a forthcoming paper.

\begin{remark} \label{rmk:figures} 
The figures showing choreographies are provided to illustrate the symmetry types of choreographies, and the differences between different connected components of the space of loops with a given symmetry type (a number of examples of this are shown in Fig.~\ref{fig:many choreographies}, and compare also Figs~\ref{fig:core} and \ref{fig:core2}).  The choreographies are all found using \textsc{Maple} or \textsc{Matlab} by a method of steepest descent to minimize the action for the $n$-body problem with Newtonian potential, although in most cases there is no guarantee such a solution exists. The method is applied on a space of finite Fourier series with coefficients satisfying conditions corresponding to the symmetry group in question, as described in Section~\ref{sec:Fourier}. Consequently, if there does exist a solution in a particular connected component of the space of loops with a given symmetry type, then it is reasonable to expect the solution to resemble the corresponding figure. Animations of the figures are available on the first author's website \cite{JMweb}; the programming for the animations was created by Dan Gries. 
\end{remark}

It has been observed by others before that it is not difficult to create numerical examples of choreographies by taking a parametrized curve with several self-intersections and placing on it a large number of particles at regular intervals, and then decreasing the action until it is minimized (using a computer programme of course). On the other hand, it appears to be much more of a challenge to produce examples where there are very few particles compared with the degree of symmetry of the curve, such as the examples $D(3,4),\; D(4,6)$ and $D(5,8)$ illustrated  in Figs~\ref{fig:many choreographies} and \ref{fig:D(3,4)}.  It would be particularly interesting to prove existence results for these examples.

\section{Classification of symmetry types}
\label{sec:classification}

\subsection{Notation}
\label{sec:notation}

Here we introduce some notation for certain subgroups and elements that will be useful throughout the paper. The symbol $n$ always denotes the number of particles, and we assume $n\geq3$.

\begin{itemize}
\item We denote by $R_{\theta}\in \SO(2)$ the rotation of the plane through an angle $\theta$ and by $\kappa\in\OO(2)$ the reflection in the horizontal axis. In complex coordinates, $R_\theta$ is multiplication by $\ee^{\ii\theta}$ and $\kappa$ is complex conjugation. We occasionally use $\kappa_\theta$ to denote reflection in the line at an angle $\theta$ with the horizontal (so $\kappa_0=\kappa$ and $R_\theta\,\kappa = \kappa_{\theta/2}$). 
\item Let $S_n$ denote the symmetric group on $n$ letters, with identity element denoted $e$, and consider two particular subgroups. Firstly, denote by $\sigma_1\in S_n$ the cycle $\sigma_1=(1\;2\;3\;\dots\; n)$ of order $n$, so $\sigma_1(j)=j+1\bmod n$, and denote by $\Sigma_n<S_n$ the cyclic subgroup generated by $\sigma_1$. Secondly, let $s_1$ be the order 2 permutation which fixes the element $1$ and reverses the numerical order: $s_1=(2\; n)(3\;(n-1))\cdots$, so $s_1(j) = 2-j\bmod n$, and denote by $\Sigma_n^+<S_n$ the dihedral group of order $2n$ generated by $\sigma_1$ and $s_1$.
\item For $k$ coprime to $n$ we denote by $\sigma_k\in\Sigma_n$ the unique permutation satisfying $\sigma_k^k=\sigma_1$. That is, $\sigma_k(j)=j+k'$ where $kk'\equiv 1\bmod n$. 
\item Recall that $\bb{T}$ is the circle $\RR/\ZZ$.  For $\theta\in\Sone$ (the group, also $\RR/\ZZ$) denote the transformation $t\mapsto \theta+t$ of $\bb{T}$ simply by $\theta$, and the reversing transformation $t\mapsto \theta-t$ by $\bar\theta$. Thus $0$ denotes the identity, and $\bar{0}$ the reflection about $t=0$. The reflection $\bar\theta$ is the reflection that fixes the points $\theta/2$ and $\theta/2 + 1/2$.
\item The choreography element is $\chor=(I,\sigma_1,-\nfrac{1}n)\in \Gamma\times\Sone$, and the choreography subgroup $\Chor_n$ is the cyclic subgroup of order $n$ generated by $\chor$.
\end{itemize}

With this notation, the symmetry group of the speed-$\ell$ circular choreography with $n$ particles described in Example~\ref{ex:circular motion}, is  generated by 
\begin{equation}\label{eq:speed ell}
(I,\sigma_1,-\nfrac{1}n),\;(R_{2\pi\ell\theta},e,\theta),\;(\kappa,s_1,\zerobar),
\end{equation}
with $\theta\in\Sone$.  

Since the space of choreographies is $\Fix(\Chor,\,\Lambda(\Xn))$, there is a natural action of the normalizer of $\Chor_n$ in $\Gamma\times\Sonehat$ on this space. This normalizer is in fact equal to $\OO(2)\times\Sigma_n^+\times\Sonehat$.  In the next proposition we show that the isotropy subgroup of any choreography is contained in this group: if collisions were allowed this would not be the case.

\subsection{Basic properties}

Associated to any choreography $u$ is its symmetry group $G$, and the projections $\rho,\sigma,\tau$ of $G$ to $\SO(2),S_n$ and $\Sonehat$ respectively. We now give a few basic properties of these projections.

\begin{proposition}\label{prop:kernels}
Let $u$ be a choreography with symmetry group $G<\Gamma\times\Sonehat$. Then,
\begin{enumerate}
\item $\ker\tau\cap\ker\rho=\ker\tau\cap\ker\sigma=\trivial$. 
\item $\ker\tau, \,\ker\sigma$ and $\ker\rho$ are cyclic groups. Indeed, $\tau(\ker\rho)$ and $\tau(\ker\sigma)$ are subgroups of\/ $\Sone$, while $\rho(\ker\tau)$ is a subgroup of\/ $\SO(2)$.
\item If\/ $\ker\rho\cap\ker\sigma\neq\trivial$ then the curve is multiply covered. Indeed, if 
$$\left|\ker\rho\cap\ker\sigma\right|=\ell>1$$
 then $u(t/\ell)$ is also a choreography, albeit with a different ordering of the particles.
\item  $\Chor_n$ is a normal subgroup of\/ $G$, and $\sigma(G)<\Sigma_n^+$.
\end{enumerate}
\end{proposition}

A simple example where $\ker\rho\cap\ker\sigma\neq\trivial$ is the circular choreography with speed $\ell$ for $\ell>1$, see Example~\ref{ex:circular motion}, where $\ker\sigma\cap\ker\rho\simeq\ZZ_\ell$.

\begin{proof}
(1) Let $\alpha \in \ker\tau\cap\ker\rho$, so $\alpha = (I, \sigma, 0)$ for some $\sigma$. Then $z_{\sigma(i)}(t) = z_i(t)$, for all $i,t$. This is not possible without collisions, and so we must have $\sigma=e$.

Now let $\alpha' \in \ker\tau\cap\ker\sigma$, so $\alpha' = (A, e, 0)$ for some $A\in\OO(2)$. Then $Au(t) = u(t)$ for all $t$. If $A$ is a rotational symmetry, then all $z_i$ are at the origin for all $t$, leading to collisions, and if $A$ is a reflection, $z_i(t) \in \RR$ for all $i$, which also implies collisions. Hence $A$ must be trivial.

(2) Assume $\tau(\ker\rho) \not< \Sone$. Then, after possibly reparametrizing time, there exists some $\alpha\in S_n$ such that $g = (I, \alpha, \overline{0})\in\ker\rho$. For this $g$, $z_{\alpha(j)}(-t) = z_j(t) \forall j,t$. If $\alpha(j) \neq j$, then there will be a collision at $t=0$. So we must have $\alpha=e$ and hence $g \in \ker\rho \cap \ker\sigma$, i.e. $g = (I, e, \overline{0})$. This implies $z_i(-t) = z_i(t)$ for all $i,t$ and so all the particles move in both directions on the same curve. This system has collisions, so is not allowed. 

Now assume $\tau(\ker\sigma)$ is not a subgroup of $\Sone$, and consider $h = (A, e, \overline{0}) \in \ker\sigma$. Then $h\chor h^{-1}\chor = (I, \sigma_1^2,0)$ which is not possible by part (1) (we assume $n>2$, so $\sigma_1^2\neq e$). Hence  $\tau(\ker\sigma)$ must be a subgroup of $\Sone$.

Now assume $\rho(\ker\tau)$ is not a subgroup of $SO(2)$; that is, there exists some element $k = (\kappa, \sigma, 0) \in \ker\tau$ (up to conjugacy), and we can assume $\sigma\neq e$ by part (1). 
Since the centre of mass is at the origin and all particles follow the same curve, this curve must intersect $\Fix(\kappa)=\RR$. Suppose $i$ is such that $\sigma(i)\neq i$, and let $t_0\in\bb{T}$ be such that $z_1(t_0)\in\RR$. The symmetry element $k$ then guarantees that $z_{\sigma(i)}(t_0) = \kappa z_i(t_0)=z_i(t_0)$ ---a collision! Hence  $\rho(\ker\tau)$ must be a subgroup of $SO(2)$. 

Then we have that $\ker\sigma$ and $\ker\rho$ must be cyclic since their images under $\tau$ are cyclic (for $\ker\tau$, its image under $\rho$) and their kernels under these maps are trivial.

(3) If $\ker\rho \cap \ker\sigma$ is not trivial, then it must be cyclic as it is contained in $\ker\rho$. Then there exists $(I, e, \nfrac{1}{\ell}) \in G$, so we have $u(t+\nfrac{1}{\ell}) = u(t)$ for all $t$, so $u$ is $\ell$-times covered. Moreover, let $\tilde{u}(t) = u(\nfrac{t}{\ell})$ for $t \in [0,1]$. This satisfies $\tilde{u}(t+1) = \tilde{u}(t)$, and so it is 1-periodic, and $\tilde{u}$ has choreography symmetry $(I, \sigma_\ell, \nfrac{1}{n})$. Indeed, set $s=\nfrac{t}{l}$, then $\tilde{z}_j(s+\nfrac{1}{n}) = z_j(t + \nfrac{1}{n\ell}) = z_{\sigma_\ell(j)}(t) = \tilde{z}_{\sigma_\ell(j)}(s)$. 

(4)  Let $\alpha = (A, \pi, a) \in G$ and conjugate $\chor$ by $\alpha$:
$$\chor^\alpha= (I,\, \pi\sigma_1 \pi^{-1},\, \pm\nfrac{1}{n}).$$
The product $\pi \sigma_1 \pi^{-1} := \sigma'$ will be of the same cycle type as $\sigma_1 = (1 \ldots n)$, namely a cycle of length $n$. 
Combine this element $\chor^\alpha = (I, \sigma', \pm\nfrac{1}{n})\in G$ with $\chor$ or its inverse as needed, depending on the sign of the $\tau$ component, to obtain $(I, \sigma'\sigma_1^{\pm 1}, 0)\in G$. Since the intersection $\ker\tau\cap\ker\rho$ is trivial by part (1), we must have that  $\sigma'\sigma_1^{\pm 1} = e$ or $\sigma' =\sigma_1^{\mp1}$. Then, $\chor^\alpha = (I, \sigma_1^{\pm1}, \mp\nfrac{1}{n})$ and hence is in $\Chor_n$, and so $\alpha$ is in the normalizer of $\Chor_n$ as required.  Finally since $\sigma_1^\pi = \sigma_1^{\pm1}$ the only possibility is that $\pi\in\Sigma_n^+$ (other elements of the normalizer of $\Sigma_n^+$ conjugate $\sigma_1$ to other generators of $\Sigma^+$). 
\end{proof}

\begin{corollary}\label{cor:kerrhochoreo}
Let $u$ be a choreography with symmetry $G$. Then either the kernel of $\rho$ is the choreography group $\Chor_n$, or the choreography is multiply covered.
\end{corollary}

\begin{proof}
Firstly, note that $\Chor_n < \ker\rho$, by its definition.

Assume $\ker\rho \neq \Chor_n$. Then there exists an element $g = (I, \alpha, \nfrac{1}{r})$ in $G\setminus\Chor_n$. We know from the proposition that the $\tau$ component of $g$ must not be time reversing, since the spatial component is trivial and hence there cannot be a reversal of the direction of motion.

Let $m=\mathrm{lcm}(r,n)$ so that $m>n$. Combination of elements gives $h=(I, \beta, -\nfrac{1}{m})$. We know $m$ is a multiple of $n$, so set $m=an$, and then $h^a = (I, \beta^a, -\nfrac{1}{n})$.

Multiplying by $\chor^{-1}$ gives $(I, \sigma_1^{-1}\beta^{a}, 0)$ which cannot occur by Proposition \ref{prop:kernels}, except in the case where $\beta^a = \sigma_1$, or equivalently $\beta = \sigma_a$. In this case, we have $h=(I, \sigma_a, -\nfrac{1}{m})$ and $m=an$, so we obtain an element of $\ker\rho \cap \ker\sigma$, and hence the curve is multiply covered by Proposition \ref{prop:kernels}.
\end{proof}

\subsection{Symmetry types}
\label{sec:symmetry types}

\begin{figure}
\centering
\subfigure[$D(6,5)$]{\includegraphics[scale=0.2]{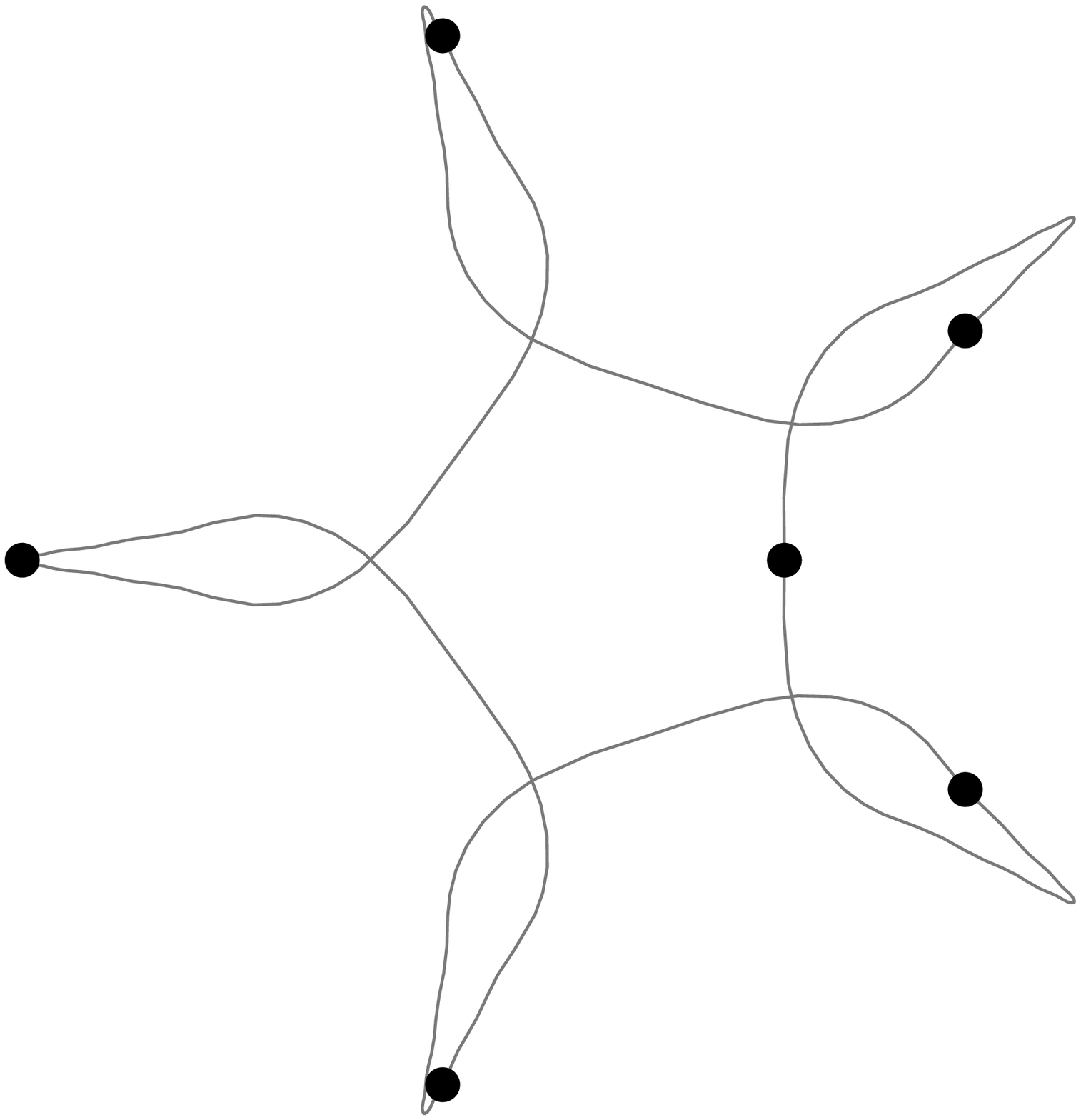}}
 \qquad \qquad
\subfigure[$D(6,5/2)$]{\includegraphics[scale=0.2]{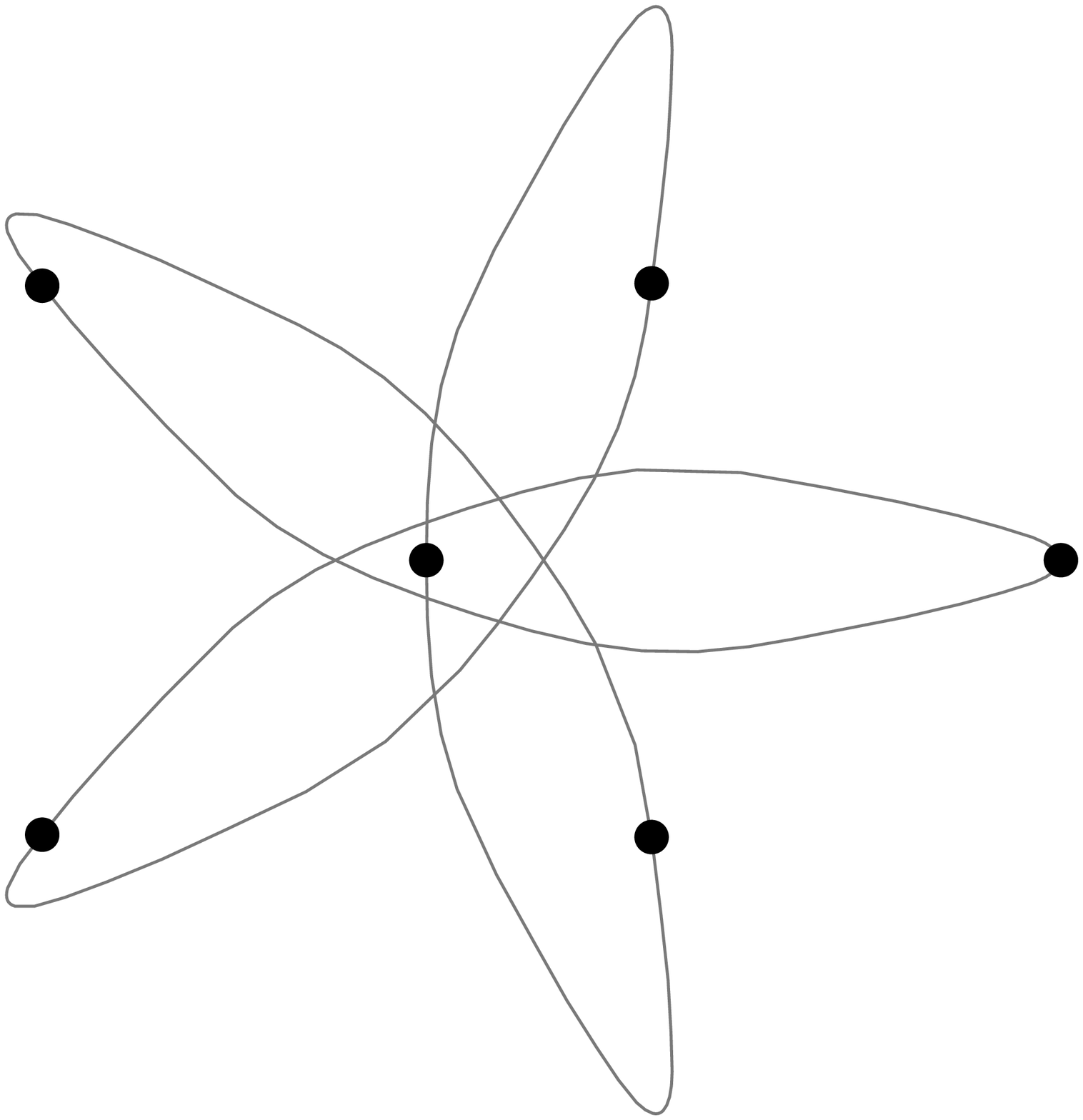}}
\caption{Comparison of different values of $\ell$: $D(6,5)$ and $D(6,5/2)$ is the difference between the symmetries of a pentagon and a pentagram.}
\label{fig:ell comparison}
\end{figure}

We define some finite subgroups of $\Gamma\times\Sonehat$, which turn out to be all possible finite symmetry types of planar choreographies (Theorem~\ref{thm:classification} below); a \emph{symmetry type} is a conjugacy class of symmetry groups, or isotropy subgroups.  As usual, let $n$ denote the number of particles and assume $n\geq 3$ and recall that $\chor=(I,\,\sigma_1,\, -\nfrac{1}n)$.  There are two regular families and three exceptional subgroups, the latter only for odd values of $n$.

The two regular families are:
\begin{description}
\item[$C(n,k/\ell)$] with $k\geq 1$, $\ell\geq 1$ coprime to $k$;  this is the subgroup generated by $\chor$ and $g_0=(R_{2\pi\ell/k},e,\nfrac1k)$. If $\ell=1$ we write this as $C(n,k)$.  Clearly $g_0$ and $\chor$ commute, so $C(n,k/\ell)\simeq \bb{Z}_n\times \bb{Z}_k$, which is cyclic if $n$ and $k$ are coprime. In particular, $C(n,1)=\Chor_n$, and if $k>1$ we can restrict to  $1\leq\ell < k/2$ ($\ell=0$ would contradict Proposition~\ref{prop:kernels}, and see Remarks~\ref{rmks:symmetry groups} below).  

\item[$D(n,k/\ell)$] with the same conditions on $k,\ell$: this is generated by the two generators $\chor$ and $g_0$ of $C(n,k/\ell)$ together with the `reflection' $(\kappa,s_1,\bar 0)$ (see Sec.~\ref{sec:notation} for the notation).  The group is isomorphic to $(\bb{Z}_n\times\bb{Z}_k)\rtimes\bb{Z}_2$, which is an index 2 subgroup of $D_n\times D_k$ and of order $2nk$. Indeed,
$$D(n,k/\ell) \simeq \bigl\{(g,h) \in D_n \times D_k \mid g, h\textrm{ are both reflections or both rotations}\bigr\}.$$
 If $\ell=1$ we write this as $D(n,k)$.  In particular $D(n,1) = \Chor_n\rtimes\bb{Z}_2\simeq D_n$, the dihedral group of order $2n$.
\end{description}

The three exceptional subgroups, arising only for odd $n$, are,
\begin{description}
\item[$C'(n,2)$]  --- the cyclic group of order $2n$ generated by $g=(\kappa,\,\sigma_2,\, -\nfrac{1}{2n})$. Notice that $g^2=\chor$, and $g^n=(\kappa,e,\nfrac12)$.
\item[$D'(n,1)$] --- the dihedral group of order $2n$ generated by $\chor$ and the `reflection' $(R_\pi,s_1,\zerobar)$.
\item[$D'(n,2)$]  --- the dihedral group of order $4n$ generated by $(\kappa,\, \sigma_2,\, -\nfrac{1}{2n})$ and $(R_\pi,s_1,\zerobar)$.  This group contains both $C'(n,2)$ and $D'(n,1)$ as index-2 subgroups.
\end{description}

Note that when dealing with general statements, writing $C(n,k/\ell)$ includes the case $\ell=1$, and clearly $C(n,k/1)=C(n,k)$, and similarly $D(n,k/1)=D(n,k)$ (as described in the introduction). See Figs~\ref{fig:many choreographies}, \ref{fig:ell comparison} and \ref{fig:D(3,4)} for examples of choreographies with symmetry $D(n,k/\ell)$.

We are now ready to state the classification theorem.  Recall that the circular choreographies are defined in Example~\ref{ex:circular motion} and their symmetry is given in \eqref{eq:speed ell}.  Consistent with the notation for symmetry groups above, we denote the symmetry group of the circular choreographies by $D(n,\infty/\ell)$.

\begin{theorem} \label{thm:classification}
Let $n\geq3$. The symmetry group of any planar $n$-body choreography is either that of a circular choreography $D(n,\infty/\ell)$ or is conjugate to one of the symmetry groups  $C(n,k/\ell)$, $D(n,k/\ell)$ and if $n$ is odd, $C'(n,2)$, $D'(n,1)$ and $D'(n,2)$.
\end{theorem}

The proof of this theorem is the subject of Section~\ref{sec:main proof}.  For the remainder of this section, we make some elementary observations about the symmetry groups in the following remarks, and describe a few consequences of the classification.

\begin{figure} [t]
\centering
\subfigure[$D(3,4)$ at $t=0$]%
 {\includegraphics[scale=0.18]{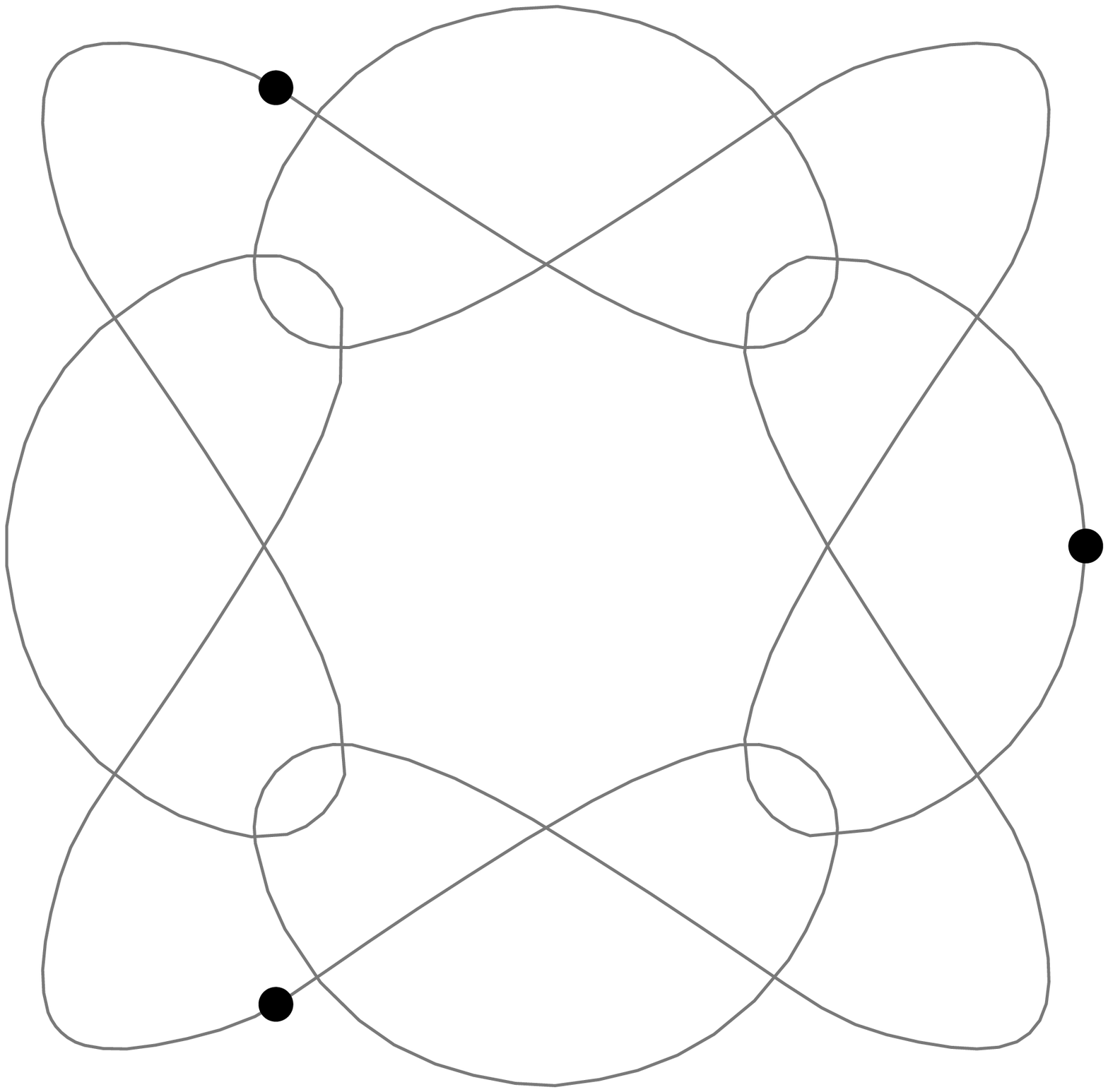}}
 \qquad
\subfigure[$D(3,4)$ at  $t=1/24$]%
 {\includegraphics[scale=0.18]{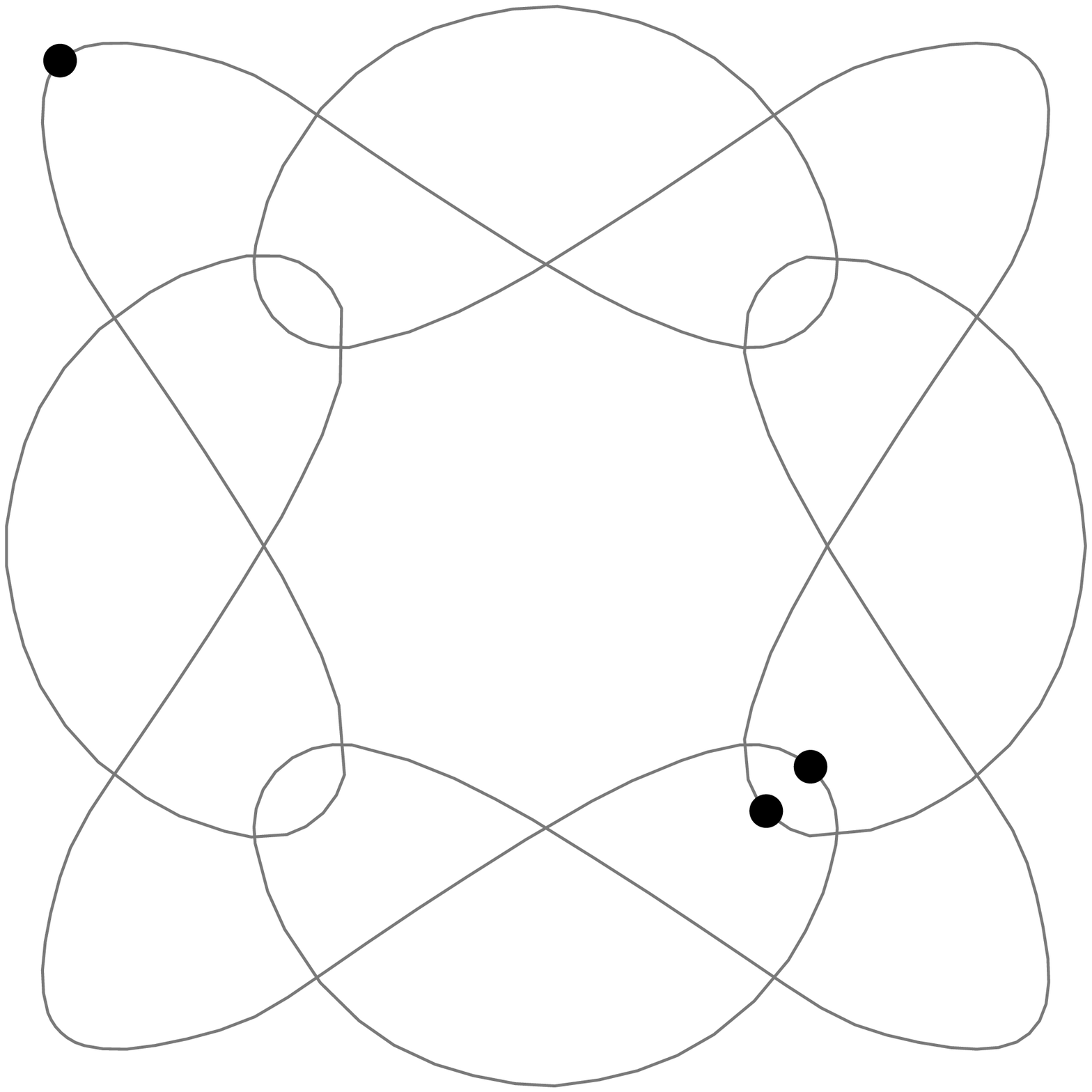}}
 \qquad 
\subfigure[$D(3,4)$ at  $t=1/12$]%
 {\includegraphics[scale=0.18]{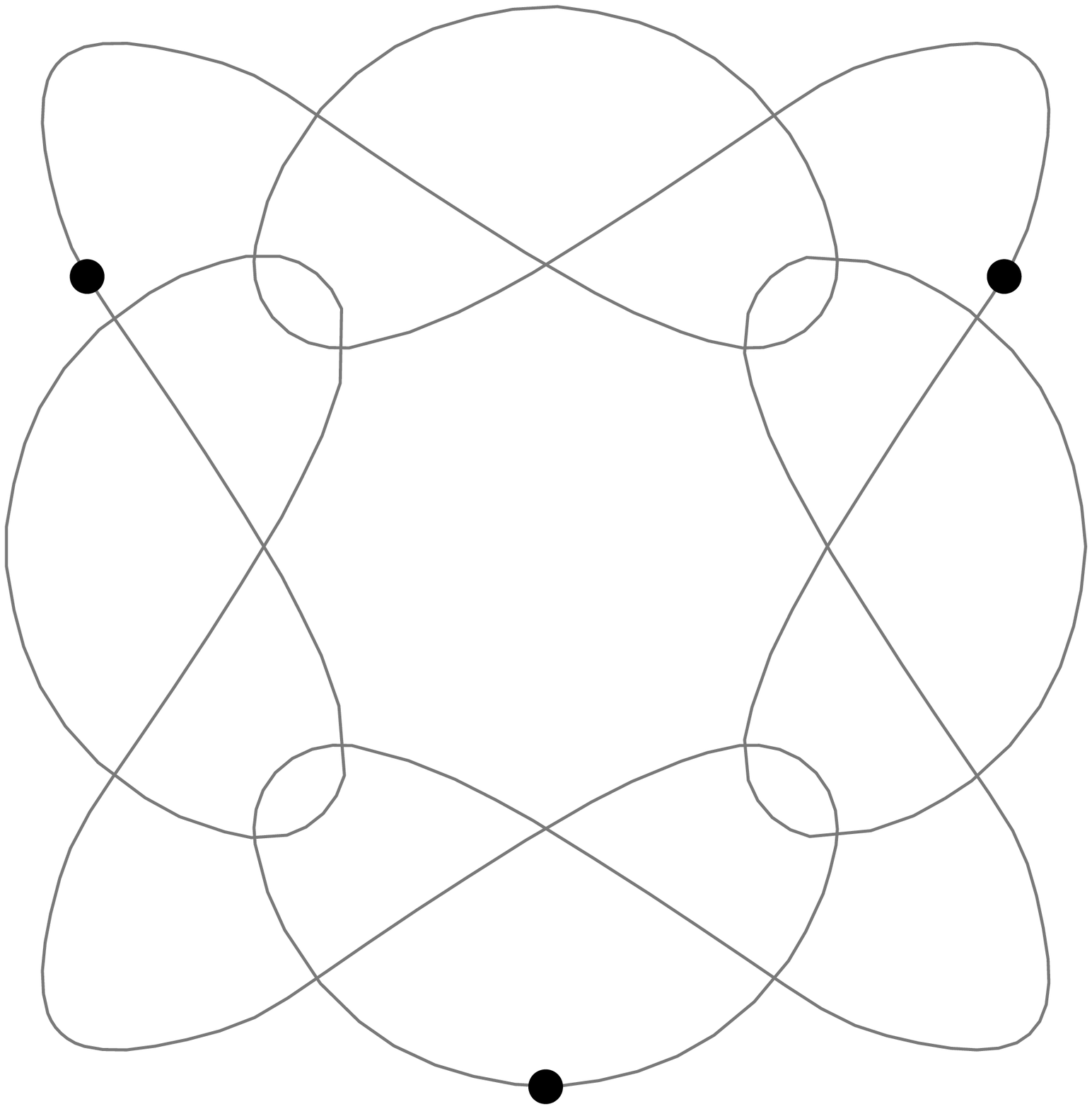}}
\caption{A choreography for the 3-body problem with 4-fold symmetry (``3 particles on a Celtic knot''). Note that the particles at $t=1/12$ are obtained from those at $t=0$ by rotating by $-\pi/2$ and relabelling by $\sigma_1^{-1}=(1\;3\;2)$. This corresponds to the symmetry element $g=(R_{3\pi/2},\,\sigma_1^2,\, 1/12)$, the generator of $C(3,4)$.  Fig.~(a) demonstrates the element $(\kappa,\,(2\;3),\zerobar)\in D(3,4)$, while Fig.~(b) illustrates $(\kappa_{\pi/4},\,(1\;3),\,\overline{1/12})$ and Fig.~(c) illustrates $(\kappa_{\pi/2},\,(1\; 2),\,\overline{1/6})$. See  \cite{JMweb} for the animation.
}
\label{fig:D(3,4)}
\end{figure}

\begin{remarks}\label{rmks:symmetry groups}
\begin{enumerate}
\item For the groups $C(n,k/\ell)$ and $D(n,k/\ell)$ it is useful to note that the cyclic groups $\ker\rho,\;\ker\sigma$ and $\ker\tau$ (see Proposition~\ref{prop:kernels}) are generated as follows:
\begin{itemize}
\item $\ker\rho=\Chor_n$ by $\chor=(I,\,\sigma_1,\,-\nfrac1n)$,
\item $\ker\sigma$ by $g_0=(R_{2\pi\ell/k},\,e,\,\nfrac1k)$, and
\item $\ker\tau$ (the core), which is cyclic of order $c:=(n,k)$, by
\begin{equation}\label{eq:core generator}
g_0^{k/c}\chor^{n/c} = (R_{2\pi \ell/c},\, \sigma_1^{n/c},\,0).
\end{equation}
\item It is also useful to note that if $c=(n,k)=1$ then $\tau$ is an isomorphism so that $C(n,k/\ell)$  is cyclic of order $nk$ generated by
\begin{equation}\label{eq:coprime generator}
g_0^a\chor^b  = (R_{2\pi a\ell/k},\,\sigma_1^b,\,\nfrac1{nk}),
\end{equation}
where $a,b$ are such that $an-bk=1$. More generally, if $(n,k)=c$ then $\tau(G)$ is generated by $c/nk$ and one has
\begin{equation}\label{eq:c/nk generator}
 g_0^a\chor^b  = (R_{2\pi a\ell/k},\,\sigma_1^b,\,\nfrac{c}{nk}),
\end{equation}
where now $an-bk=c$, but $G$ itself is not cyclic.
\end{itemize}

\begin{figure}  
\centering
\subfigure[$D(6,4)$]{\includegraphics[scale=0.2]{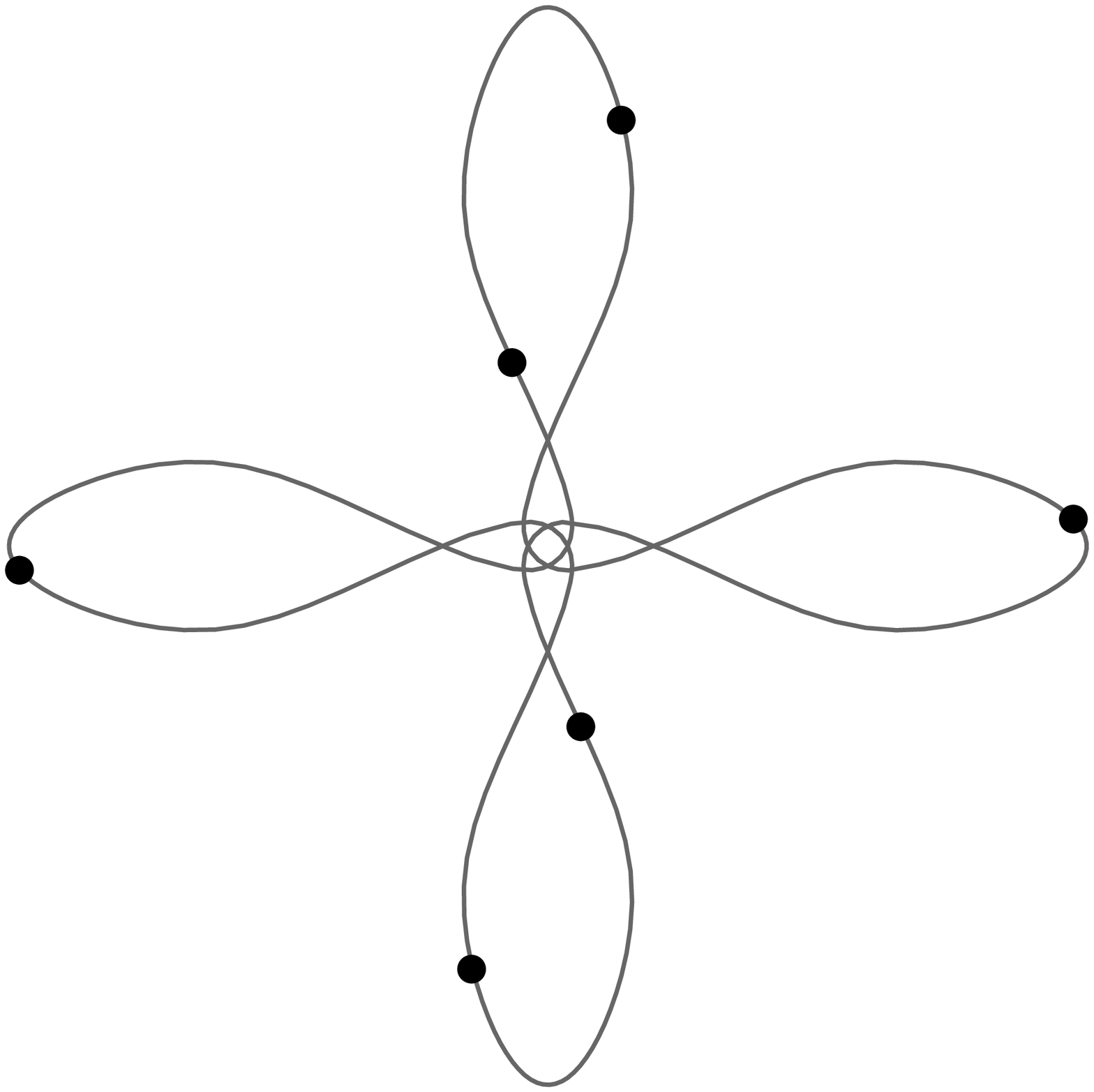}}
 \qquad \qquad
\subfigure[$D(10,5/2)$]{\qquad\includegraphics[scale=0.22]{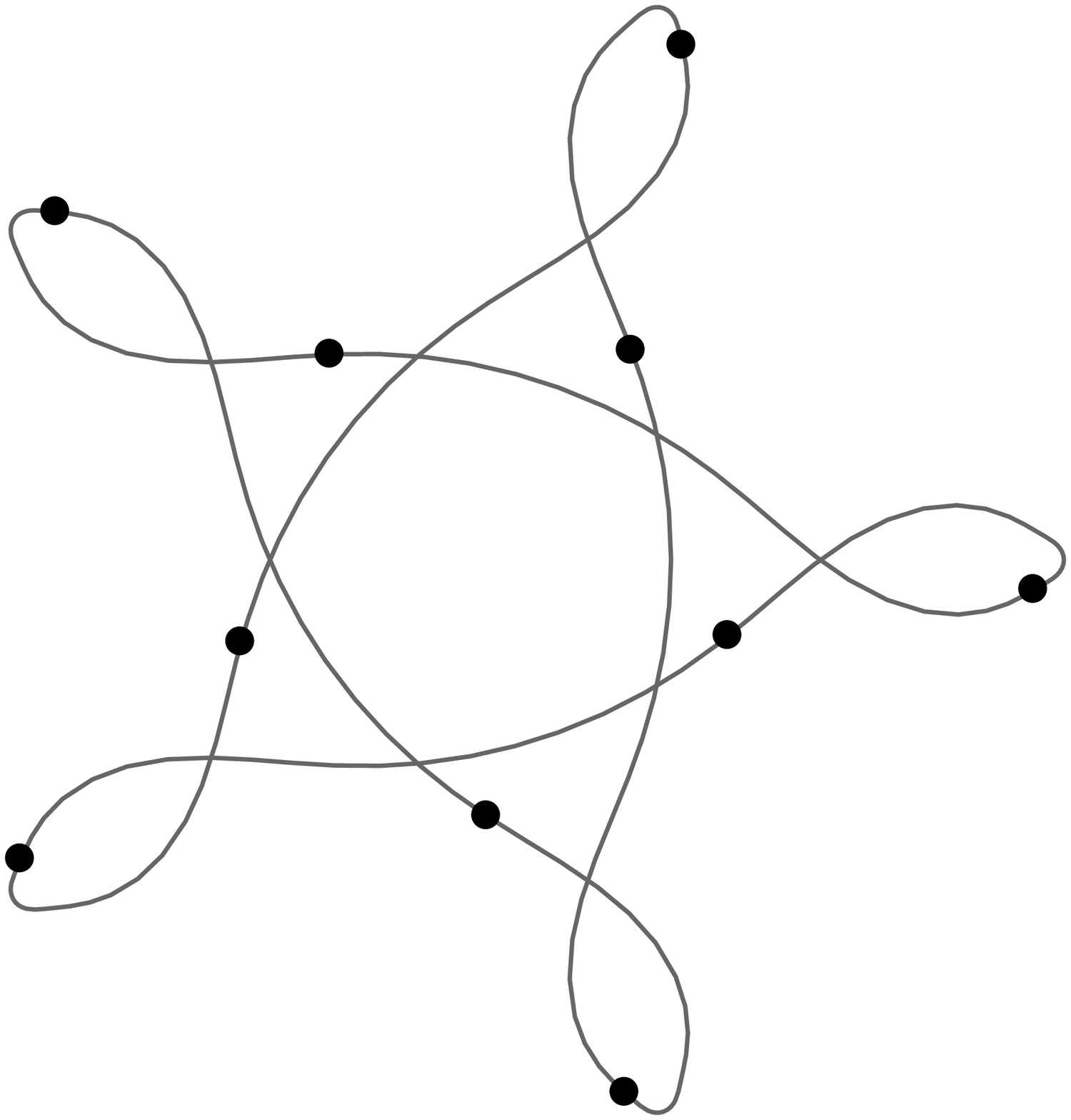}}
\caption{Two examples of choreographies with nontrivial core: in the first the order of the core is $c=2$ and in the second $c=5$ (generated by rotations through $\pi$ and $2\pi/5$ respectively). They are shown for generic values of $t$ illustrating the core symmetry which is valid for all $t$. Compare with special configurations $t=0$ in Fig.~\ref{fig:many choreographies}~(c) and (p)}
\label{fig:core}
\end{figure}

\item It was stated in the definition of $C(n,k/\ell)$ that $\ell$ is coprime to $k$.  If $(k,\ell)>1$ then the element $(I,e,\nfrac{1}{(k,\ell)})$ is in $G$, contradicting Proposition~\ref{prop:kernels}.
Moreover, the group $C(n,k/\ell)$ is conjugate to $C(n,k/(k-\ell))$ (by the element $(\kappa,e,0)$, and this together with the coprimality, allows us to restrict attention to $\ell=1$ for $k\leq4$ and  $1\leq \ell < k/2$ if $k\geq5$.  The same restrictions also apply to $D(n,k/\ell)$. 
\item If $n$ is even the analogues of the exceptional subgroups would involve collisions and so do not arise.
\item The subgroups denoted $C$ are all non-reversing symmetry groups while those denoted with $D$ are reversing (that is, of \emph{cyclic} and \emph{dihedral} type respectively, in the language of \cite{FT04}).
\item We will see that all these symmetry types are of interest for choreographies in the $n$-body problem (at least with a strong force interaction---see Theorem~\ref{thm:SF}), except those where $n$ divides $k$ (see Example~\ref{eg:n divides k}).
\item In the two regular families, an element preserves orientation in $\bb{T}$ if and only if it preserves orientation in $\bb{R}^2$, while for the exceptional groups there are elements $g$ with $\tau(g)\in\Sone$ but $\rho(g)\not\in\SO(2)$, or conversely with $\rho(g)\in\SO(2)$ but $\tau(g)\not\in\Sone$.
\item  For any choreography with symmetry type $D'(n,1)$ or $D'(n,2)$, the particles pass through the origin. This is because the presence of the time-reversing element $(R_\pi,s_1,\zerobar)\in G$ implies that $z_1(0)=R_\pi z_1(0)$, so that $z_1(0)=0$.  It then follows that each of the other particles also passes through 0.  
\item Choreographies of type $C'(n,2)$ may or may not pass through the origin, but if one does not then its winding number around the origin is 0.  This is because the symmetry $(\kappa,\sigma_2, \nfrac{1}{2n})$ reverses the orientation of the plane, but preserves the orientation of the curve, so takes the winding number to its opposite. This is illustrated in  Fig.~\ref{fig:exceptionals}(a,b) where the origin (the barycentre) is clearly to the left of the symmetric crossing point. 
\end{enumerate}
\end{remarks}

\begin{figure} [t]   
\centering
  \subfigure[$C'(7,2)$ at $t=0$]%
{\includegraphics[scale=0.2]{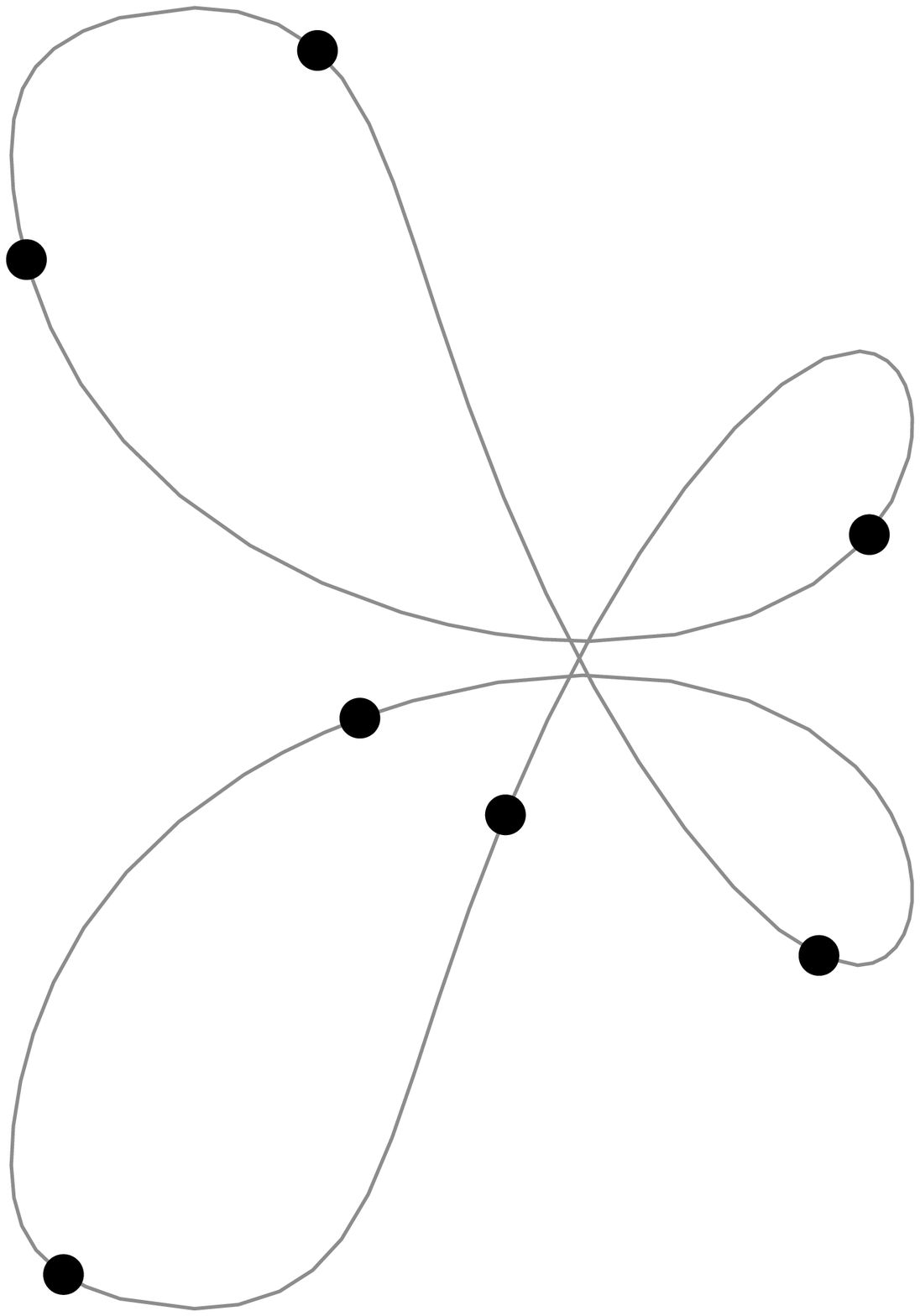}}
 \qquad \qquad
  \subfigure[$C'(7,2)$ at $t=1/14$]%
{\includegraphics[scale=0.2]{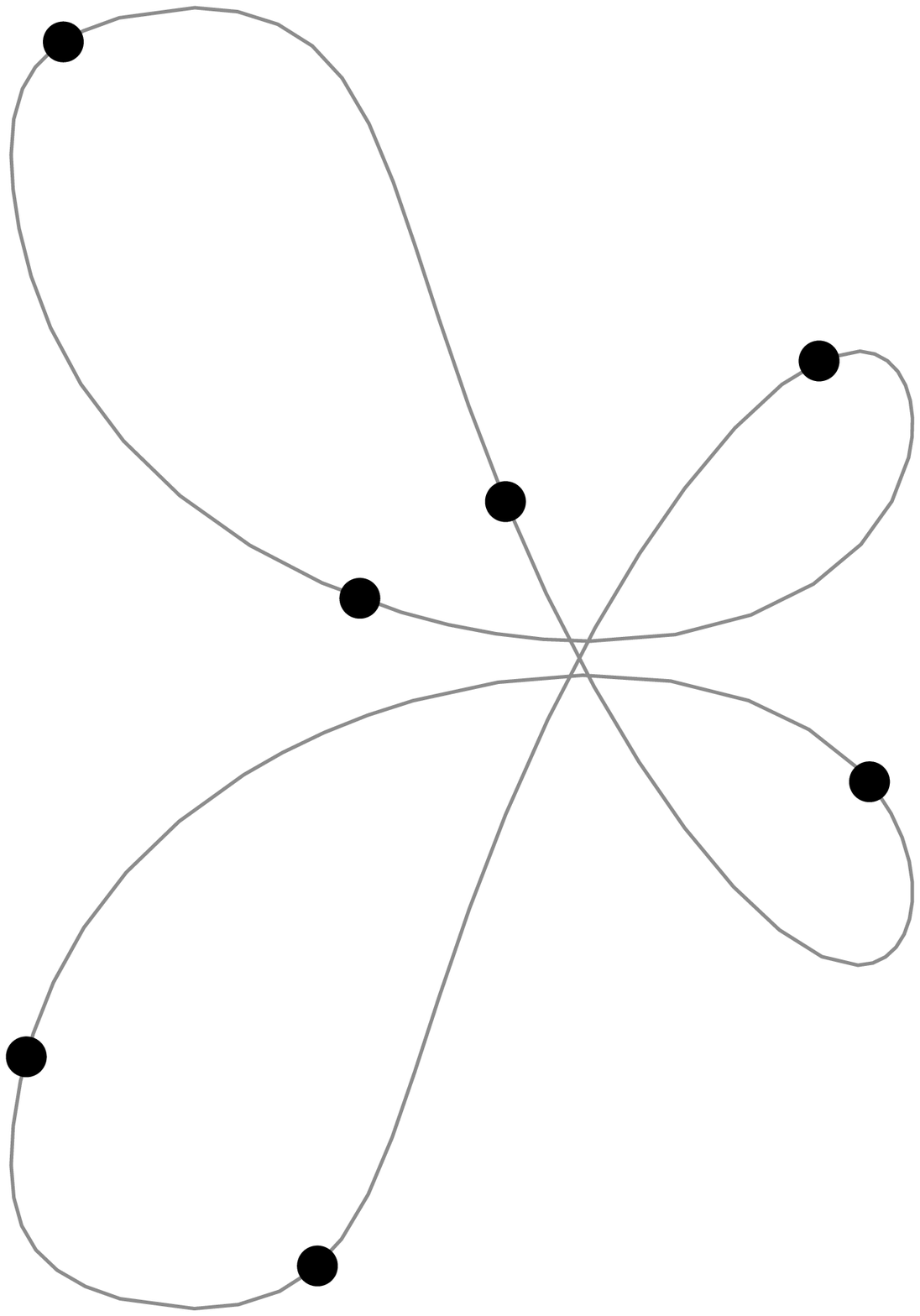}}
 
\subfigure[$D'(9,1)$ (at $t=0$)]%
 {\includegraphics[scale=0.2]{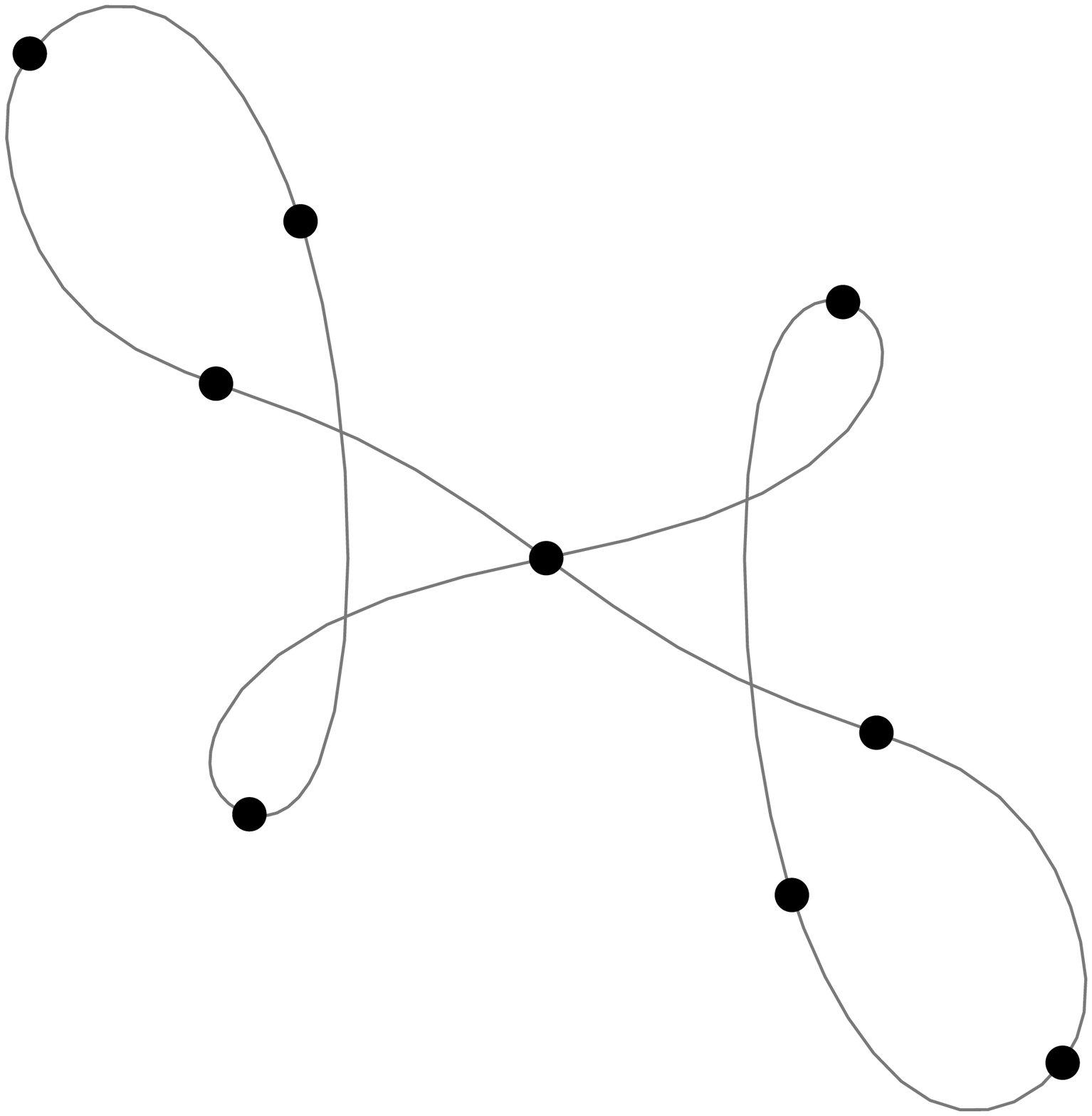}}
\qquad \qquad
\subfigure[$D'(7,2)$ (at $t=0$)]%
 {\includegraphics[scale=0.2]{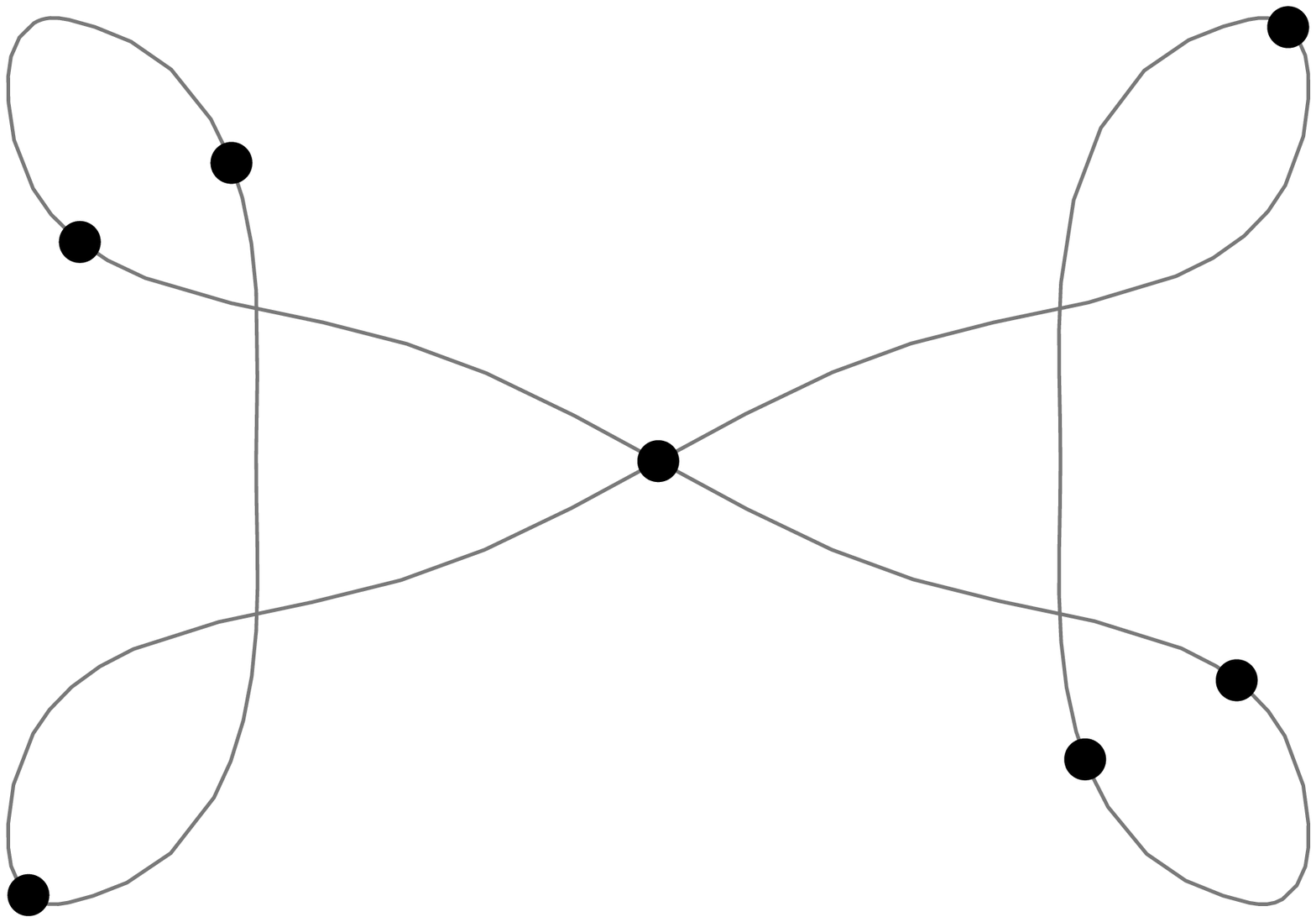}}
\caption{Exceptional symmetry groups. 
  (a,b) A choreography for the 7-body problem, with
  symmetry $C'(7,2)$, demonstrating the symmetry generator
 $(\kappa,\,\sigma_2,\,\mbox{\small 1/14})$, where here  
  $\sigma_2=(1\;5\;2\;6\;3\;7\;4)$. 
 (c) A choreography with $D'(9,1)$ symmetry:  rotation by $\pi$ is combined with a time reversal. (d) A choreography with $D'(7,2)$ symmetry which possesses both the other symmetries: a rotation by $\pi$ combined with time-reversal as well as a reflection which is not time reversing; another with this symmetry, in a different connected component, would be the figure-8 curve with 7 particles. 
  See  \cite{JMweb} for animations. 
}
\label{fig:exceptionals}
\end{figure}

\begin{proposition} \label{prop:vanishing angular momentum}
Choreographies with symmetry type $C'(n,2),D'(n,1)$ or $D'(n,2)$ have zero angular momentum.
\end{proposition}

\begin{proof}
The symmetries we consider each have a well-defined effect on angular momentum. The rotations $\SO(2)$, the permutations $S_n$ and the time translations $\Sone$ preserve the angular momentum, while the reflections in $\OO(2)$ and the time reversing elements of $\Sonehat$ change angular momentum to its opposite. It is clear that any choreography whose symmetry group contains an element which changes the sign of the angular momentum must have angular momentum equal to zero.  This is the case for all the exceptional symmetry groups, as $(\kappa,\sigma_2,-\nfrac1{2n})$ and $(R_\pi,s_1,\zerobar)$ reverse the angular momentum (as do all conjugate elements). It is not the case for the regular families.
\end{proof}

\subsection{Fourier series} \label{sec:Fourier}

Write $z(t)$ for the parametrized curve defining the choreography $u(t)$, with $t\in\bb{T}=\bb{R/Z}$. So 
$$z_j(t) = z\left(t+(j-1)/n\right),\quad j=1,\dots,n.$$
Using complex coefficients we can write $z$ as a Fourier series,
\begin{equation}\label{eq:Fourier}
z(t) = \sum_{r\in\bb{Z}}\zeta_r\exp(2\pi\ii rt).
\end{equation}

The fact that the centre of mass is at the origin translates into the following constraint on the coefficients, as observed by Sim\'o \cite{Simo01a}:

\begin{lemma} 
If $u$ is a choreography with $n$ particles and  $r$ is a multiple of $n$ then $\zeta_r=0$ in (\ref{eq:Fourier}).
\end{lemma}

\begin{proof}
Let \(z(t)=\sum_r\zeta_{r}\exp(2\pi\ii rt)\) and \(z_j(t)=z(t+(j-1)/n)\). The centre of mass (as a function of time) is
\begin{eqnarray*}
0&=&\frac1n \sum_r \zeta_{r}\left[\sum_{j=1}^n \exp(2\pi\ii r(t+j/n))\right] \\
& =&\frac1n \sum_r \exp(2\pi\ii rt)\left[\zeta_{r}\sum_{j=1}^n \exp(2\pi\ii rj/n)\right].
\end{eqnarray*}
This is satisfied if and only if
\[ \zeta_{r}\sum_{j=1}^n \exp(2\pi\ii rj/n) = 0\quad(\forall r)\] 
and the result then follows since the sum over \(j\) vanishes if \(n\) is not a divisor of \(r\), otherwise it is equal to \(n\) .
\end{proof}

If the choreography only has choreographic symmetry $\Chor_n=C(n,1)$  (as in Fig.~\ref{fig:choreo-egs}(d)) then there is no further restriction on the Fourier series of the underlying curve. If there is just one reflection giving a time-reversing symmetry, then $z(-t)=\overline{z(t)}$, and this translates into the condition $\zeta_{n}\in\RR$, which in practice means that $x(t)$ has a cosine expansion and $y(t)$ a sine expansion (where $z=x+\ii y$).

\begin{proposition} \label{prop:Fourier}
The symmetry of a choreography $u$ translates into the conditions on the Fourier coefficients shown in Table \ref{table:Fourier}. 
\end{proposition}

\begin{proof}
This is a simple calculation for each group.
\end{proof}

\begin{table}
\centering
\begin{tabular}{@{}c l@{}}
\toprule
Symmetry & Conditions on Fourier coefficients\\ \midrule\midrule
$C(n,k/\ell)$ & $\zeta_r=0\;\forall r\not\equiv\ell\bmod k$ \\ \midrule
$D(n,k/\ell)$ & as $C(n,k/\ell)$, with $\zeta_r\in\bb{R}$ \\ \midrule
$\zeta_r\in\bb{R}$ \\ \midrule\midrule
$C'(n,2)$ &  $\zeta_{-r}=(-1)^r\overline{\zeta_r}\,,\;\forall r$ \\ \midrule
$D'(n,1)$ & $\zeta_{-r}=-\zeta_r\,\;\forall r$\\ \midrule
$D'(n,2)$ & $\zeta_{-r}=-\zeta_r\,,\;\forall r$ with $\begin{cases}\zeta_r\in\bb{R}& \mbox{if } r \mbox{ is even}\cr \zeta_r\in\ii\,\bb{R}& \mbox{if } $r$ \mbox{ is odd}\end{cases}$\\ 
\bottomrule
\end{tabular}
\caption{Conditions on Fourier coefficients in (\ref{eq:Fourier}) for each symmetry type.} \label{table:Fourier}
\end{table}

\subsection{Isotropy subgroup lattice}\label{sec:lattice}

For two subgroups $G,H$ of $\Gamma$, we write $H\prec G$ to mean that $H$ is subconjugate to $G$---that is, $H$ is conjugate to a subgroup of $G$.  It is a transitive relation.  

\begin{proposition}
The subgroups listed above satisfy the following subconjugacy relations (recall that $C(n,k/1)=C(n,k)$ and $D(n,k/1)=D(n,k)$)
\begin{enumerate}
\item $C(n,k/\ell)\prec D(n,k/\ell)$ 
\item $C(n,k/\ell)\prec C(n,k'/\ell')$ iff $k\divides k'$ and $\ell\equiv \pm\ell'\bmod k$
\item $D(n,k/\ell)\prec D(n,k'/\ell')$ iff $k\divides k'$ and $\ell\equiv \pm\ell'\bmod k$
\item $C'(n,2)\prec D'(n,2)$ and $D'(n,1)\prec D'(n,2)$
\item $C(n,1)\prec C'(n,2)$ and $C(n,1)\prec D'(n,1)$
\end{enumerate}
These are illustrated in Figure~\ref{fig:lattice}. Note that (2) and (3) hold with $k'=\infty$.
\end{proposition}

An example illustrating part (2) is that $C(4,5)\prec C(4,10/\ell')$ for $\ell'=1$ and $4$ while $C(4,5/2)\prec C(4,10/\ell')$ only for $\ell'=3$ as $\ell'=2$ is not coprime to $k'=10$. Part (3) is similar.

\begin{proof} These all follow from the definitions of the groups (also from the Fourier series representations above). That there are no other subconjugacies is a simple case-by-case analysis, using for example that $G=C(n,k/\ell)\prec G'=C(n,k'/\ell')$ requires $|G|$ divides $|G'|$, and so $k\divides k'$. Moreover, the generator of $C'(n,2)$ is not conjugate to any element of any $C(n,k/\ell)$, or indeed of any element of $D(n,\infty/\ell)$. 
\end{proof}

The importance of this lattice of isotropy subgroups for choreographies lies in the fact that $H\prec G$ is a necessary condition for there to be a sequence of choreographies with symmetry $H$ converging to a choreography with symmetry $G$.  For example, for $D(n,\infty/\ell)\prec D(n,k/\ell)$, the choreography determined by
$$z_1(t)=\ee^{2\pi \ii\ell t}(1+\eps\,\cos(2\pi k \ii t)),$$
which has symmetry $D(n,k/\ell)$ converges as $\eps\to0$ to the circular choreography given in Example~\ref{ex:circular motion}, which has symmetry $D(n,\infty/\ell)$. Note also that for $\eps\neq0$ the particles follow each other in numerical order with time lag $1/n$, while for $\eps=0$ and $\ell>1$ particle $j$ follows particle $j+m$ where $m\ell\equiv 1\bmod n$, but now with time lag $1/\ell n$.

\begin{figure} 
\begin{center}
\psset{yunit=1.2}
\begin{pspicture}(-4,-0.2)(4,4.5)
\psset{nodesep=2pt}
\rput(-2,4){\rnode{A}{$D(n,\infty/\ell')$}}
\rput(-3,2){\rnode{C}{$C(n,k'/\ell')$}} \rput(-2,3){\rnode{B}{$D(n,k'/\ell')$}}
\rput(-2,1){ \rnode{E}{$C(n,k/\ell)$}}  \rput(-1,2){\rnode{D}{$D(n,k/\ell)$}}
\ncline[linestyle=dashed]{-}{A}{B} \ncline{-}{B}{C} \ncline{-}{B}{D}
\ncline{-}{C}{E} \ncline{-}{D}{E}

\rput(1,1.5){\rnode{K}{$C'(n,2)$}} \rput(3.2,1.5){\rnode{L}{$D'(n,1)$}}
\rput(2.2,3){\rnode{M}{$D'(n,2)$}}
\rput(0,-0.2){\rnode{X}{$C(n,1)=\Chor_n$}}
\ncline{-}{E}{X} \ncline{-}{K}{X} \ncline{-}{L}{X} 
\ncline{-}{L}{M} \ncline{-}{M}{K} 
\end{pspicture}
\end{center}
\caption{Lattice of isotropy subgroups, with $k\divides k'$ and $\ell'\equiv\pm\ell\bmod k$.}
\label{fig:lattice}
\end{figure}
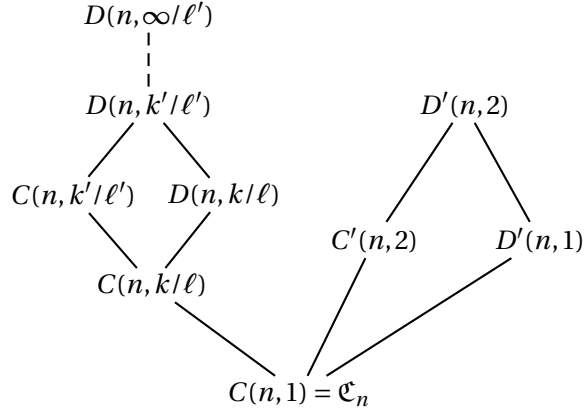

\subsection{Rotating circle condition} \label{sec:rcc}
In their very interesting paper \cite{FT04}, Ferrario and Terracini introduce the \emph{rotating circle condition}, and show by a clever perturbation argument that if the action of $G$ satisfies this condition then a collision path in the loop space cannot be a local minimum of the restriction of the Newtonian action functional  to the set of loops with symmetry $G$.  This condition is defined for choreographies in $\RR^d$, and in $\RR^2$ it reduces to the following:

A symmetry group $G$ satisfies the rotating circle condition if for each $t\in[0,1]$, one has 
\begin{enumerate}
\item $\rho(G_{t}) < \SO(2)$, and 
\item $\rho(G_{t,i})=\trivial$, for at least $n-1$ of the indices $i$.  
\end{enumerate}
Here $G_t$ is the subgroup of $G$ that fixes $t$ (under $\tau$), and $G_{t,i}$ is the subgroup of $G_t$ that fixes the index $i$ (under $\sigma$). Of course $\trivial$ is the trivial group.

\begin{proposition}\label{prop:rcc}
The groups $C(n,k/\ell)$, $C'(n,2)$ and $D'(n,1)$ all satisfy the rotating circle condition, while the remaining groups $D(n,k/\ell)$ and $D'(n,2)$ do not.
\end{proposition}

\begin{proof}
For $G=C'(n,2)$ this is immediate as $G_t$ is trivial for all $t$. The same is true of $C(n,k/\ell)$ when $n,k$ are coprime. More generally, for $G=C(n,k/\ell)$, $G_t$ is the core generated by $g_c=(R_{2\pi\ell/c},\sigma_1^{n/c},0)$ where $c$ is the order of the core ($\ker\tau$)---see Remarks~\ref{rmks:symmetry groups}. The first condition is clearly met as $\rho(g_c)=R_{2\pi\ell/c}\in\SO(2)$, while the second follows from the fact that $\ker\tau\cap\ker\sigma$ is trivial (Proposition~\ref{prop:kernels}). For $D'(n,1)$ we have $G_t$ is of order at most 2, with $G_0$ generated by $g_0=(R_\pi,s_1,\zerobar)$.  Here $\rho(g_0)\in\SO(2)$, and $\sigma(g_0)=s_1$ which fixes precisely one index (namely, 1 --- recall $n$ is odd for $D'(n,1)$).

For the two remaining types, we have elements $(-\kappa, s_1\sigma_2, \overline{1/2n})\in D'(n,2)$ and $(\kappa,s_1,\zerobar)\in D(n,k/\ell)$, both of which violate the rotating circle condition (here $-\kappa=R_\pi\kappa$ is the reflection in the vertical axis).
\end{proof}

\section{Proof of Classification Theorem} \label{sec:main proof}

In this section we prove Theorem \ref{thm:classification} and we rely extensively on the notation introduced in Sec.~\ref{sec:notation}.  Let $u$ be a choreography of period 1, with finite symmetry $G$. We also assume $u$ does not have period less than 1 (if it has minimal period $T<1$ then replace $u(t)$ by $u(t/T)$ which then has period 1, and if necessary relabel the particles).  By Proposition~\ref{prop:kernels}(4), $G<\OO(2)\times\Sigma_n^+\times\Sonehat$, and  $\Chor_n\lhd G$.   The proof is in two halves: first we assume that the symmetry group has trivial core (so $\tau$ is injective) and then we reduce the general case to the first, by considering the quotient by a free group action.  As usual $\tau:G\to \Sonehat$ is the projection

\subsection{Trivial core} \label{sec:no core proof}
Recall that the core of the symmetry group is $\ker\tau$.  So first we assume $G$ is such that $\tau$ is injective.  This means that $G\simeq\tau(G)<\Sonehat$, so that $G$ is isomorphic to a cyclic or a dihedral group. For $\OO(2)$ and $\Sonehat$ there are two non-conjugate of subgroup of order 2: the first is the order two subgroup of $\SO(2)$ or $\Sone$, which we denote $C_2$ and call cyclic, while the second is generated by a `reflection' in $\OO(2)$ or $\Sonehat$, which we refer to as dihedral and denote it $D_1$.

\subsubsection{Non-reversing symmetry groups}
We are assuming $\tau(G)\subset \Sone$. 
Since $\tau$ is injective, it follows that $G$ is isomorphic to a subgroup of $\Sone$ so is a cyclic group and since it contains $\Chor_n$ its order is a multiple of $n$. Suppose then that $G\simeq\bb{Z}_{kn}$, where $k$ is a positive integer.

Let $g_0$ be the element of $G\simeq\bb{Z}_{kn}$ with $\tau(g_0) = \nfrac{1}{nk}$: it is a generator of $G$. Then 
$$g_0 = ( \rho(g_0), \sigma(g_0), \nfrac{1}{nk})$$ 
for some $\sigma(g_0)\in\Sigma_n^+$ and $\rho(g_0)\in\OO(2)$.  Consider $g_0^k$. Now $g_0^k=(\rho(g_0)^k, \sigma(g_0)^k, \nfrac{1}{n})$, and since $\tau$ is injective this must be equal to $\chor$. 
It follows that $\rho(g_0)^k=I$ and $\sigma(g_0)^k=\sigma_1$. 

The equation $\sigma(g_0)^k=\sigma_1$ has a solution if and only if $(n,k)=1$.  In that case the solution is unique and is, by definition (see Sec.~\ref{sec:notation}), $\sigma(g_0)=\sigma_k$.

Now consider $\rho(g_0)$. 
Since $\ker\rho=\Chor_n$ (Corollary \ref{cor:kerrhochoreo}) we have a short exact sequence,
$$ \trivial \to \Chor_n \lra \bb{Z}_{kn} \stackrel{\rho}{\lra} \bb{Z}_k\to\trivial,$$
where  $\bb{Z}_k < \OO(2)$.
\begin{itemize}
\item If $k=1$, then $\rho(g_0)=I$ so $G = \Chor_n = C(n,1)$.
\item If $k>2$, then we have $\bb{Z}_k<\OO(2)$ is generated by $R_{2\pi/k}$. 
The element $\rho(g_0)$ generates $\bb{Z}_k$, and so must be equal to $R_{2\pi\ell/k}$ for some $\ell$ coprime to $k$.
In this case, we have $G = C(n,k/\ell)$ with $k$ coprime to both $n$ and $\ell$.
\item In the case where $k=2$, we have two possibilities. Either $\bb{Z}_k$ is $C_2$, generated by $R_\pi$, in which case we have $\rho(g_0) = R_\pi$ and $G = C(n,2)$, or it is dihedral $D_1$, generated by $\kappa$ (or a conjugate), in which case $\rho(g_0) = \kappa$ and we have $G = C'(n,2)$.
\end{itemize}

\subsubsection{Reversing symmetry groups} 

This is a bit more involved. Since we are assuming the core is trivial, $G\simeq\tau(G)$, and $G$ is therefore a dihedral group.  Since $\tau(G)$ contains $\nfrac1{n}\in\Sone$ it must be $D_{kn}$ for some $k$.  Let $G_0 = \tau^{-1}(\ZZ_{nk})$ be the index-2 subgroup of $G$ consisting of the non-reversing elements.  Let $r\in G$ be an order 2 element satisfying $rg=g^{-1}r$ for $g\in G_0$, so $r$ together with $G_0$ generates $G$. Up to conjugacy by $\Sone$, we can suppose $\tau(r)=\zerobar$. Furthermore, since $r$ is of order 2, it follows that $\rho(r)$ and $\sigma(r)$ are of order 1 or 2.

For $\sigma(r)$, the dihedral condition implies $r$ anticommutes with $\chor$ and so in particular,
$$\sigma(r) \sigma_1 \sigma(r)^{-1} = \sigma_1^{-1}.$$
As $\sigma_1$ generates an abelian group of order $n>2$, $\sigma(r)$ cannot be in this group. Hence it must be a reflection on the points.  Thus for odd $n$, $\sigma(r)\in\Sigma_n^+$ must be conjugate to $s_1$, which fixes one particle (namely, $z_1$), while for even $n$, there are two possibilities:
$$\sigma(r) = 
  \begin{cases} s_{12} & \textrm{fixing no particles} \\ 
                  s_1 & \textrm{fixing } z_1 \textrm{ and } z_{1+n/2},  \end{cases}
$$
but with $\sigma_1$ both generate the same dihedral group $\Sigma_n^+$.

For the $\rho$ component, since $\rho(r)$ is an element of order at most 2, we must have $\rho(r) = I, R_\pi$ or a reflection, and this is to be combined with $G_0=C(n,k/\ell)$ or, if $n$ is odd, $C'(n,2)$, from the first---non-reversing---part of the proof.   However, $\rho(r)=I$ would give an element of $\ker\rho$ whose temporal component is not in $\Sone$, contradicting Proposition~\ref{prop:kernels}(2). There remains to consider $\rho(r)=R_\pi$ or a reflection.

The dihedral condition on $r$ gives,
\begin{equation}\label{eq:dihedral on rho}
\begin{cases}
\rho(r) R_{2\pi\ell/k} \rho(r)^{-1} = R_{-2\pi\ell/k}& \mbox{if }G_0=C(n,k/\ell)\\
\rho(r)\kappa\rho(r)^{-1} = \kappa& \mbox{if }G_0=C'(n,2).
\end{cases}
\end{equation}

Consider each case in turn.  First suppose $G_0=C(n,k/\ell)$. 

\begin{itemize}

\item If $\rho(r)=R_\pi$, it would commute with $R_{2\pi\ell/k}$ in the above which is only possible if $k\leq2$ (recall that  $(k,\ell)=1$). Then $G$ contains the element $g=(R_\pi,s,\zerobar)$. If $k=2$ then combining $g$ with $(R_\pi,e,\nfrac12)\in C(n,2)$ produces the element $(I, s, \overline{\nfrac12})$ which again contradicts Proposition~\ref{prop:kernels}(2). 
If on the other hand $k=1$ we adjoin the element $(R_\pi,e,\zerobar)$ to $C(n,1)=\Chor_n$ which gives the group $D(n,1)$.

\item If $\rho(r) = \kappa$ then $G$ contains $(\kappa,s,\zerobar)$, with $s=s_1$ or $s_{12}$ as before, in which case $G$ is of type $D(n,k/\ell)$: note that together with $C(n,k/\ell)$ the element $(\kappa,s,\zerobar)$ with the two possible values of $s$ generates conjugate subgroups of $\Gamma\times\Sonehat$.
\end{itemize}

Now suppose $G_0=C'(n,2)$ which is generated by $g'=(\kappa, \sigma_2, \nfrac{1}{2n})$ with $n$ odd. By (\ref{eq:dihedral on rho}), $\rho(r)$ commutes with $\kappa$, and so is either $R_\pi$ or one of the reflections $\kappa$ or $\kappa':=R_\pi\kappa$ (reflection in the line orthogonal to $\Fix(\kappa)$). Moreover, since $n$ is odd $\sigma(r)=s_1$ (up to conjugacy, or relabelling the particles). 
\begin{itemize}
\item If $\rho(r)=R_\pi$ then $r=(R_\pi,s_1,\zerobar)$ so giving the group $G=D'(n,2)$.
\item If $\rho(r)=\kappa$ then $rg'\in\ker\rho$ but $\tau(rg')\not\in\Sone$ so contradicting  Proposition~\ref{prop:kernels}(2).
\item Finally, if $\rho(r)=\kappa'$ then $\rho(r\kappa)=\kappa'\kappa=R_\pi$ and we are in the case $G=D'(n,2)$ again.
\end{itemize}

We have therefore considered every case with a trivial core, and shown each either gives one of the groups of the classification in Section~\ref{sec:symmetry types} or by contradicting Proposition~\ref{prop:kernels} leads to a collision so the trivial core case is complete. We next consider the cases with non-trivial core.

\begin{figure}[t]
\centering
\subfigure[$D(4,6)$ which has core of order 2, and \dots]%
 {\includegraphics[scale=0.2]{Pics/D461.eps}}
 \qquad \qquad
\subfigure[\dots its image in $X_*^{(2)}$]%
 {\includegraphics[scale=0.18]{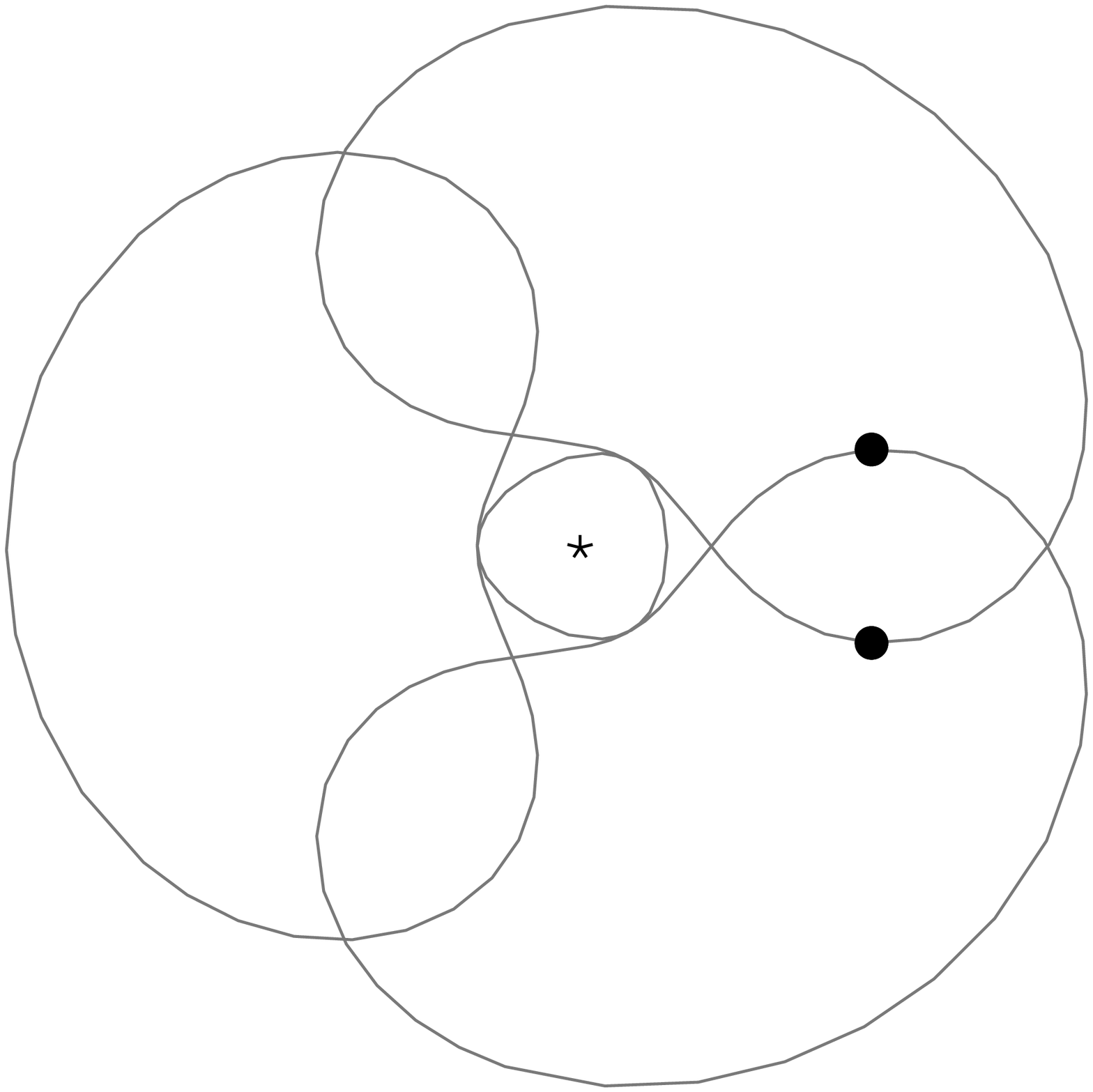}}
\caption{A choreography $u$ for the 4-body problem with 6-fold symmetry, which has core of order 2, and its image $\uhat$ under the map $\Psi$. 
}
\label{fig:D(4,6,1)}
\end{figure}

\subsection{Non-trivial core}
\label{sec:non-trivial core}

Suppose now the choreography $u$ has symmetry group $G$ with non-trivial core $K:=\ker\tau$. It follows that $u(t)\in\Fix(K,\Xn)$ for all $t$.  By Proposition \ref{prop:kernels},  $K$ is a cyclic group of order $c$ say and is generated by an element of the form $g_c:=(R_{2\pi r/c},\, \sigma_1^{\hat n},\, 0)\in\Gamma\times\Sonehat$ for some $r$ coprime to $c$ (and we may assume $1\leq r \leq c/2$ up to conjugacy).   Throughout the paper we put $\nhat = n/c$ and for given $K$ we write $Y=Y_K=\Fix(K,\Xn)$.   For the choreography to have core $K$ it follows that, for all $j$,
\begin{equation} \label{eq:isotropy}
z_{j+\nhat} = \ee^{2\pi\ii r/c}z_j.
\end{equation}

We begin with an elementary lemma. Recall that the centre of mass is fixed at the origin.

\begin{lemma} \label{lemma:avoiding origin}
Any choreography with non-trivial core does not pass through the origin. 
\end{lemma}

\begin{proof}
Let $g_c= (R_{2\pi r/c},\, \sigma_1^{n/c},\, 0)$ generate the core, and suppose for contradiction that the curve does pass through the origin: say $z_1(t_0)=0$.  Then $z_{\nhat+1}(t_0)=\ee^{2\pi\ii r/c}0 = 0$, resulting in a collision.
\end{proof}

The classification of choreographies in $Y$ proceeds by identifying $Y$ as a smooth finite cover of $\Xnhat$, the configuration space of $\nhat$ distinct particles in the \emph{punctured} plane $\CC\setminus\{0\}$, thereby reducing the classification of choreographies with core to those with no core, but now in $\Xnhat$.  Explicitly, the covering is given by
\begin{equation}\label{eq:Psi}
\begin{array}{rcl}
\Psi:\Fix(K,\Xn)&\lra& \Xnhat\\
(z_1,\dots,z_n)&\longmapsto& (z_1^c,\dots,z_{\nhat}^c)
\end{array}
\end{equation}
Note that it follows from the lemma that the $z_j\neq0$.  Solving \eqref{eq:Psi} and using \eqref{eq:isotropy} shows that $\Psi$ is a smooth covering of degree $c^{\nhat}$.

Let $u(t)=(z_1(t),\dots,z_n(t)$ be a choreography with core $K$, and define the loop $\uhat$ in $\Xnhat$ by
$$\uhat(t) = (w_1(t),\dots,w_{\nhat}(t)) = \Psi(u(t/c)),$$
so that $w_j(t) = z_j(t/c)^c$.  Combining \eqref{eq:isotropy} with the choreography symmetry of $u$ shows each $w_j$ is of period 1, and the loop $\uhat$ is indeed a choreography.  Furthermore, if the core of $u$ is precisely $K$ then the core of $\uhat$ is trivial, so we can apply the classification we have so far obtained (\S\ref{sec:no core proof}) to $\uhat$.

We can immediately rule out the exceptional groups: $D'(n,1)$ and $D'(n,2)$ as they necessarily pass through the origin (which is excluded by the lemma above).  We can also rule out $C'(n,2)$ as its winding number is necessarily zero (see Remarks~\ref{rmks:symmetry groups}(8)), while we have:

\begin{lemma}
Any choreography $u$ with non-trivial core has non-zero winding number around the origin.
\end{lemma}

\begin{proof} Let $u(t)=(z_1(t),\dots,z_n(t))$, and consider the underlying curve $z(t)=z_1(t)$ say. The winding number is given by integrating $dz/z$ around the curve. First integrate $dz/z$ from $t=0$ to $t=1/c$.  This is equivalent to $\ln(z(1/c) - \ln(z(0)) \bmod 2\pi\ii$, which is $2\pi\ii(r/c \bmod 1) \neq0$, using the choreography symmetry and \eqref{eq:isotropy}.  The integral from $t=0$ to $t = 1$ is therefore $2\pi\ii(r\bmod c) \neq0$, and so the winding number itself is equivalent to $r\bmod c$, so we are done.
\end{proof}

There remain the possibilities of $\uhat$ having symmetry $C(\nhat,\khat/\ellhat)$ or  $D(\nhat,\khat/\ellhat)$, for some integers $\khat,\ellhat$ with $(\nhat,\khat)=(\khat,\ellhat)=1$. So the question remains, given one of these symmetries for $\uhat$, what are the possible symmetries of $u$?  

Suppose first that $\uhat$ has symmetry $C(\nhat,\,\khat/\ellhat)$, for $\ellhat$ coprime to $\khat$. We claim that $u$ has symmetry (conjugate to) $C(n,k/\ell)$ for some $\ell$ necessarily coprime to $k=c\khat$ and with $\ell=\pm\ellhat\bmod\khat$.  Note that since $c\divides n$ and $(\nhat,\khat)=1$ it follows that $(n,k)=c$, consistent with the fact that the core is of order $c$. 

Now $C(\nhat,\,\khat/\ellhat)$ is generated by the choreography element $\hat\chor=(I,\sigmahat_1, -\nfrac1{\nhat})$ and $\widehat g_0 = (R_{2\pi\ellhat/\khat}, e, \nfrac1{\khat})$.  Explicitly, for each $j$ and each $t$,
$$\ee^{2\pi\ii \ellhat/\khat}w_j\left(t-\tfrac1{\khat}\right) = w_j(t).$$
Taking $c^{\textrm{th}}$ roots and replacing $t$ by $t/c$ gives
$$ 
  \ee^{2\pi\ii\ellhat/k}z_j\left(t-\tfrac{c}{nk}\right) = \ee^{2\pi\ii h/c}z_j(t), 
$$
for some integer $h$ (independent of $j$ and $t$ by the choreography symmetry), with $k=c\khat$. That is, the symmetry group of $u$ contains, in addition to $K$ and $\Chor_n$, the group generated by 
$$g=(R_{2\pi(\ellhat-h\khat)/k},\,e,\,\nfrac{c}{nk}).$$
This is the symmetry group $C(n,k/\ell)$ with $\ell=\ellhat+h\khat$ as required. Note that $r$ in the core is related to $\ell$ by $r\equiv \ell\bmod c$.

Finally suppose $\uhat$ has symmetry $D(\nhat,\,\khat/\ellhat)$. Then it also has symmetry $C(\nhat,\,\khat/\ellhat)$, and so $u$ has symmetry $C(n,k/\ell)$ as above.  Moreover, $\uhat$ has symmetry $(\kappa,s_1,\bar 0)$. That is,
$$\bar w_{2-j}(-t) = w_j(t).$$
Here the index $2-j$ is modulo $\nhat$.  Lifting this to $u$ (taking $c^{\textrm{th}}$ roots and replacing $t$ by $t/c$ as above) gives 
$$\bar z_{2-j+f\nhat}(-t) = z_j(t),$$
for some $f$. This leads to a subgroup conjugate to $D(n,k/\ell)$, and concludes the proof of Theorem~\ref{thm:classification}. \hfill$\Box$


\section{Connected components of spaces of symmetric loops}
\label{sec:components of loops space}

Let us fix notation.  Throughout the remainder of this paper, $X$ will denote a smooth manifold, and $\Gamma$ will be a Lie group acting smoothly on $X$.  (Although we assume $X$ to ba a manifold, the results hold more generally.) 
The loop space $\Lambda X$ on $X$ is the set of all continuous maps $u:\bb{T}\to X$, where we identify the circle $\bb{T}$ with $\bb{R}/\ZZ$.  All loop spaces are given the compact-open topology.
There is then a natural action of $\Gamma\times\Sone$ on the loop space $\Lambda X$ given by
$$((g,\theta)\cdot u)(t) = g\cdot u(t-\theta).$$
  We will consider the natural action of $\Gamma \times\Sonehat$ in a later paper.  For the remainder of this paper we change notation slightly compared to the earlier sections.  
Instead of writing $G< \Gamma\times\Sone$ as hitherto, we now let $G$ be a subgroup of $\Gamma$, and the corresponding subgroup of $\Gamma\times\Sone$ is given as the graph of a homomorphism $\tau:G\to \Sone$; we denote this subgroup by $G_\tau$.

In this section we analyse the sets of connected components of the spaces $\Fix(\Gt,\;\Lambda X)$ for subgroups $\Gt$ of $\Gamma\times\Sone$.  In the following section we apply these results to spaces of choreographies with $C(n,k/\ell)$ and $C'(n,2)$ symmetry. In a future paper we will consider the reversing symmetry groups in $\Gamma\times\Sonehat$, with applications to the remaining symmetry types of choreography.  

The topological analysis in this section is based on \cite{MMRS05}, but expressed in more ``equivariant'' terms.

\subsection{Symmetric loops}
We assume throughout that $X$ is a manifold with an action of a Lie group $\Gamma$, and we let $\tau:G\to \Sone$ be a given homomorphism, and write $K=\ker\tau\lhd G$. The graph $\Gt$ of $\tau$ is of course the subgroup
$\{(g,\,\tau(g))\in\Gamma\times\Sone\mid g\in G \}$.  A loop $u\in\Lambda X$ is said to have symmetry $G_\tau$ if $u\in \Fix(G_\tau,\Lambda X)$.  This means $(g,\tau(g))\cdot u = u$ for all $g\in G$, or $g\cdot u(t) = u(t+\tau(g))$.

The following lemma is easy to check. 

\begin{lemma}
The action of\/ $G_\tau$ on $\Lambda X$ restricts to an action of\/ $(G/K)_\tau$ on $\Lambda(X^K)$, where $X^K=\Fix(K,\,X)$.  Moreover $\Fix(G_\tau,\Lambda X) = \Fix((G/K)_\tau,\Lambda(X^K))$.
\end{lemma}

For $g\in \Gamma$ let 
$$\Lambda^gX = \left\{\gamma:[0,1]\to X\mid \gamma(1)=g\cdot \gamma(0)\right\}$$
be the space of \defn{relative loops with phase} $g$, or more briefly of \defn{$g$-loops}. In particular, $\Lambda^eX=\Lambda X$. 

For each $g\in \Gamma $ and each positive integer $r$, define a map $\phi_r:\Lambda^gX\to\Lambda^{g^r}X$ by concatenating $\gamma$ with its image under $g$ and then with its image under $g^2$ etc, up to $g^{r-1}$. That is,
\begin{equation}\label{eq:defn of phi_r}
\phi_r(\gamma)(t) = \begin{cases}\gamma(rt)& \text{if } t\in[0,\,1/r] \\
        g\gamma(rt-1) & \text{if }  t\in[1/r,\,2/r] \\
         \quad\vdots\\
        g^{r-1}\gamma(rt-r+1) & \text{if }  t\in[(r-1)/r,\,1].
\end{cases}
\end{equation}
Note that indeed $\phi_r(\gamma)(0)=\gamma(0)$ and $\phi_r(\gamma)(1)=g^r\gamma(0)$. Furthermore, if $g^r=e$ then $\phi_r(\gamma)\in\Fix(H,\Lambda X)$, where $H$ is the cyclic group generated by $(g,\nfrac{1}{r})\in \Gamma \times \Sone $.

\begin{proposition} \label{prop:fix is Lambda^g}
Suppose $\tau(G)$ is of order $r$, so $\tau(G)=\ZZ_r < \Sone $. Then the map
$$\phi_r:\Lambda^gX^K\lra \Fix(G_\tau,\,\Lambda X),$$ 
defines a homeomorphism for any $g\in G$ satisfying $\tau(g)=\nfrac{1}{r}$.
\end{proposition}

Note that if $g,g'$ satisfy $\tau(g)=\tau(g')$ then they differ by an element of $K$ and therefore they act the same way on $X^K$ and so $\Lambda^gX^K=\Lambda^{g'}X^K$.

\begin{proof}
Let $\gamma\in\Lambda^gX^K$ and $u=\phi_r(\gamma)$.  Since $\gamma(t)\in X^K$ for all $t$, so is $u(t)$. Since $g^r\in K$ it follows that $u$ is indeed a loop. That $u$ has symmetry $G_\tau$ is straightforward.  Moreover $\phi_r$ is a homeomorphism as the inverse is simply the restriction $u\mapsto [t\mapsto u(t/r)]$ (with $t\in[0,1]$) and both maps are continuous in the compact-open topology.
\end{proof}

\begin{definition}\label{def:generating path}
Given a loop $u\in\Lambda X$ with symmetry $G_\tau$ as above, we call $\gamma=\phi_r^{-1}(u)$ a \defn{generator} of $u$.  Explicitly,
\begin{equation} \label{eq:generating loop}
\gamma(t) = u(t/r),\quad t\in[0,1]. 
\end{equation}
\end{definition}

\subsection{Equivariant loop space} \label{sec:equivariant loops}
Let $X$ be a manifold and $\Gamma$ a group acting on $X$. Denote by $\paths(X)$ the space of all paths on $X$---that is continuous maps $[0,1]\to X$, and by $\paths(X,x)$ the space of those paths $\gamma$ for which $\gamma(0)=x$.  We follow standard notation and denote by $\Lambda X$ the space of (free) loops in $X$, so the set of those $\gamma\in\paths(X)$ for which $\gamma(1)=\gamma(0)$,  and we let $\Omega(X,x)$ denote the space of loops based at $x$, so $\gamma(0)=\gamma(1)=x$. All spaces of paths are given the compact-open topology. 

We define the \defn{equivariant free loop space} to be
$$\Lambda^\Gamma X:=\left\{(\gamma,g)\in \paths(X)\times \Gamma  \mid \gamma(1)=g\cdot\gamma(0)\right\}.$$
Similarly for $x\in X$, $\Omega^\Gamma (X,x)$ consists of those equivariant loops $(\gamma,g)$ with $\gamma(0)=x$ and $\gamma(1)=g\cdot x$. We will denote by $\beta$ the projection $\paths(X)\times \Gamma \to \Gamma $ given by $\beta(\gamma,g)=g$.  We also denote the restrictions of $\beta$ to $\Lambda^\Gamma X$ and $\Omega^\Gamma(X,x)$ by $\beta$.  In the last section we defined, for each $g\in \Gamma $,
$$\Lambda^gX =\left\{\gamma\in\paths(X) \mid \gamma(1)=g\cdot\gamma(0)\right\} = \beta^{-1}(g) \subset \Lambda^\Gamma X.$$

In the same way that the fundamental group $\pi_1(X,x)$ is defined to be $\Omega(X,x)/{\sim}$, where $\sim$ means homotopy of paths relative to the end-points, we define
$$\pi_1^\Gamma (X,x) = \Omega^\Gamma (X,x)/{\sim},$$
where again $\sim$ is relative to end-points, and with $g$ kept fixed.
This set $\pi_1^\Gamma (X,x)$ has a natural group structure given by
$$(\gamma,g)(\delta,h) = (\gamma*(g\delta),gh)$$
where $g\delta$ is the image of $\delta$ under the action of $g$ on $X$, and $*$ is the usual concatenation product in homotopy. [Note that our notation does not distinguish between homotopy classes and their representatives: one could be more careful and write for example $([\gamma],g)([\delta],h) = ([\gamma*g\delta],gh)$. We hope this will not be a source of confusion.]  
Some details and examples of this group can be found in the second author's thesis \cite[Sec.~2.4]{Steckles}. The group was introduced by Rhodes \cite{Rhodes} where he calls it the fundamental group of the transformation group, and denotes it $\sigma(X,x,\Gamma )$.   
The group is independent of base point $x$ up to isomorphism, provided $X$ is path connected, or more 
generally\footnote{Rhodes \cite{Rhodes} proves this only for Abelian groups, but it is easily seen to be true in general: indeed the map $(\gamma,g)\mapsto(h\gamma,hgh^{-1})$ provides a homeomorphism $\Omega^\Gamma (X,x)\to\Omega^\Gamma (X,hx)$ which descends to an isomorphism $\pi_1^\Gamma (X,x)\to\pi_1^\Gamma (X,hx)$.} provided $\Gamma $ acts transitively on the set of connected components of $X$.  
Some other properties, immediate from the definition,  are as follows.  We write all groups multiplicatively, so the trivial (homotopy) group is denoted $\trivial$. 

\begin{itemize}
\item There is a short exact sequence,
\begin{equation}\label{eq:EFG ses}
\trivial \to \pi_1(X,x) \longrightarrow \pi_1^\Gamma (X,x) \stackrel{\beta}{\longrightarrow} \Gamma  \to \trivial,
\end{equation}
where $\beta$ is the natural projection $(\gamma,g)\mapsto g$ as above, and more generally if $N\lhd \Gamma $ (normal subgroup) then
\begin{equation}\label{eq:EFG-normal subgp}
\trivial \to \pi_1^N(X,x) \longrightarrow \pi_1^\Gamma (X,x) \longrightarrow \Gamma /N \to \trivial.
\end{equation}
\item If $H<\Gamma $ then $\pi_1^H(X,x)=\beta^{-1}(H) < \pi_1^\Gamma (X,x)$.
\item If $x$ is fixed by $\Gamma $ then $\pi_1^\Gamma (X,x) \simeq \pi_1(X,x)\rtimes \Gamma $, where the action of $\Gamma $ on $\pi_1(X,x)$ is the natural one.
In particular if the action of $\Gamma $ on $X$ is trivial then $\pi_1^\Gamma (X,x)\simeq\pi_1(X,x)\times \Gamma $.
\item If $f:X\to Y$, is an equivariant map then there is a natural homomorphism $f_*:\pi_1^\Gamma (X,x)\to\pi_1^\Gamma (Y,f(x))$. More generally, if $\Gamma <H$ and $f:X\to Y$ is $\Gamma $-equivariant, then there is a natural homomorphism 
 $$f_*:\pi_1^\Gamma (X,x)\lra  \pi_1^H(Y,f(x)).$$
\end{itemize}

At the level of homotopy, the homeomorphism $\phi_r$ of Proposition~\ref{prop:fix is Lambda^g} becomes 
\begin{equation}\label{eq:r-th power}
 (\phi_r(\gamma),g^r) = (\gamma,g)^r,
\end{equation}
using the group structure in $\pi_1^\Gamma $.

\begin{example} \label{ex:pi_1 S^1}
Let $\Sone $ act on a manifold $X$. Then there is a homomorphism $\alpha:\pi_1(\Sone ,e)\to\pi_1(X,x)$ generated by the homotopy class of the orbit $\Sone \cdot x$. It is well-known that the image of $\alpha$ lies in the centre of $\pi_1(X,x)$.  Then we claim
$$\pi_1^{\Sone }(X,x) \simeq \pi_1(X,x)\times_{\alpha}\RR,$$
where $\times_{\alpha}$ means we identify $(\gamma,r+n)$ and $(\alpha(n)\gamma,r)$ for $n\in\mathbb{Z}\simeq\pi_1(\Sone ,e)$; in other words, we factor $\pi_1(X,x)\times\RR$ by the normal subgroup generated by $(\alpha(1),-1)$, where the 1 is the generator of $\pi_1(\Sone ,e)\simeq \mathbb{Z}<\RR$.  The isomorphism is given as follows.  Let $\theta\in\RR$ and $p\in X$, and denote by $\theta_p$  the path
\begin{equation}\label{eq:theta_x}
\theta_p(t) = (t\theta)\cdot p,
\end{equation}
which is a path from $p$ to $\theta\cdot p$ following the $\Sone$-orbit. Then for $(\gamma,\theta)\in\pi_1(X,x)\times\RR$, the corresponding element of $\pi_1^{\Sone}(X,x)$ is $(\gamma*\theta_x,\; \theta\bmod 1)$. This descends to a well-defined homomorphism $\pi_1(X,x)\times_\alpha\RR \to \pi_1^{\Sone}(X,x)$, which has inverse 
$$(\gamma,\theta) \mapsto  [\gamma*(-\theta)_{\theta\cdot x},\;\theta].$$
Note that $(-\theta)_{\theta\cdot x}$ is the reverse path of $\theta_x$. (Everything is up to homotopy of course.)

The projection $\beta$ makes $\pi_1^{\Sone }(X,x) $ into a bundle over $\Sone $ with fibre $\pi_1(X,x)$ and monodromy $\alpha$.
If we restrict the action to the cyclic group $\ZZ_k<\Sone $, then the corresponding subgroup of $\pi_1(X,x)\times_{\alpha}\RR$ is
$$\pi_1^{\ZZ_k}(X,x) \simeq \pi_1(X,x)\times_{\alpha}\ZZ(1/k),$$ 
where $\ZZ(1/k)$ is the subgroup of $\RR$ consisting of integer multiples of $1/k$.
\end{example}

{A fundamental property of $\pi_1^\Gamma(X,x)$ is the following.  The first part in fact follows from a result of Rhodes \cite[Theorem 4]{Rhodes}, although he proves a more general statement about orbit spaces of not necessarily free actions. The argument in the free case is more straightforward, so we give it here (it appears also in the second author's thesis \cite[Proposition 2.4.3]{Steckles}).  

\begin{proposition}\label{prop:EFG2}
If\/ $N$ is finite and acts freely on a manifold $X$ then $\pi_1^N(X,x)\simeq \pi_1(X/N,\bar x)$; more generally, if $N$ is a finite normal subgroup of\/ $\Gamma $ and acts freely on $X$ then
$$\pi_1^\Gamma (X,x) \simeq \pi_1^{\Gamma /N}(X/N,\bar x).$$
\end{proposition}

\begin{proof}
The map $X\to X/N$ is a covering so enjoys the homotopy lifting property (see for example, Hatcher \cite[Proposition 1.30]{Hatcher}).  Consider the natural map given by projecting the path: 
\begin{equation} \label{eq:equivariant fundamental loops}
\Omega^N (X,x) \lra \Omega(X/N ,\bar{x})
\end{equation}
It follows from the homotopy lifting property that this map is a homeomorphism, and that it induces the required isomorphism
$$\pi_1^N(X,x) \lra \pi_1(X_N ,\overline{x}).$$  

The more general statement follows from the snake lemma applied to the diagram,
\begin{center}\psset{arrowscale=1.3,yunit=0.8}
\begin{pspicture}(-4,-1.5)(4,1.5)
\rput(-3,1){$\pi_1(X,x)$} \psline{->}(-2,1)(-1,1) \rput(0.3,1){$\pi_1^\Gamma (X,x)$} \psline{->}(1.5,1)(2.5,1) \rput(3,1){$\Gamma $}
\psline{->}(-3,0.6)(-3,-0.6)\psline{->}(0.3,0.6)(0.3,-0.6)\psline{->}(3,0.6)(3,-0.6)\rput(-3,-1){$\pi_1(X/N,\bar x)$} \psline{->}(-2,-1)(-1,-1) \rput(0.3,-1){$\pi_1^{\Gamma /N}(X/N,\bar x)$} \psline{->}(1.5,-1)(2.5,-1) \rput(3,-1){$\Gamma /N$}
\end{pspicture}
\end{center}
The first column is injective with cokernel $N$ since $\pi_1^N(X,x)\simeq \pi_1(X/N,\bar x)$, and the final column is surjective with kernel $N$; since the resulting homomorphism $N\to N$ is an isomorphism, it follows that the middle column is an isomorphism.
\end{proof}

A simple consequence of this is that if $\Gamma $ is a finite group, $\pi_1^\Gamma (X,x)$ is the \emph{equivariant fundamental group} of $X$:

\begin{corollary} \label{coroll:pi1G  is EFG}
If\/ $\Gamma $ is a finite group acting on a manifold $X$, then $\pi_1^\Gamma (X,x)$ is
isomorphic to the fundamental group of\/ $X_\Gamma :=X\times_\Gamma  \EGamma $.
\end{corollary}

Here $\EGamma$ is the total space of the universal classifying bundle for $\Gamma$, it is a contractible space on which $\Gamma$ acts freely.  The space $X\times_\Gamma \EGamma = (\EGamma\times X)/\Gamma$ is the (Borel construction for) the homotopy orbit space for the $\Gamma$-action on $X$, see for example \cite{GS-Supersymmetry}. 

\begin{proof}
Let $*\in E$ be a base point, and for each $g\in \Gamma $ choose a path $\omega_g$  from $*$ to $g\cdot*$ (since $\EGamma$  is contractible, $\omega_g$ is unique up to homotopy).  Consider the injective map $i:\Omega^\Gamma (X,x) \to \Omega^\Gamma (X\times \EGamma , (x,*))$ defined by $i(\gamma,g) = ((\gamma,\omega_g),g)$. 

We claim that the induced homomorphism
$$i_*:\pi_1^\Gamma (X,x) \lra \pi_1^\Gamma (X\times \EGamma, (x,*))$$
is an isomorphism. That it is injective is clear, for if $(\gamma,g)$ is in the kernel then $g=e$ and $\gamma$ is trivial. It is surjective because $\EGamma$ is contractible, so any $\Gamma$-loop $((\gamma,\delta_g),g)$ is homotopic to $((\gamma,\omega_g),g)$ which is in the image of the original map $i$.  Therefore
$$\pi_1^\Gamma (X,x) \simeq \pi_1^\Gamma (X\times \EGamma, (x,*)) \simeq \pi_1(X_\Gamma,\overline{(x,*)}),$$
where the latter isomorphism follows from the proposition above.
\end{proof}
}

\subsection{Connected components}
We are interested in the topology of the space $\Lambda^\Gamma X$ of equivariant loops, and in particular of each $\Lambda^gX=\beta^{-1}(g)$.   To this end we adapt the usual argument showing that the connected components of the free loop space $\Lambda X$ correspond to conjugacy classes in the fundamental group $\pi_1(X,x)$.

First we need the topology of the \emph{based} equivariant loop space $\Omega^\Gamma (X,x)$, or rather $\Omega^g(X,x)$ for each $g$.  This is derived by giving $\Gamma $ the \emph{discrete} topology and using the fibration $\paths(X,x)\times \Gamma \to X,\; (\gamma,g)\mapsto g^{-1}\gamma(1)$ with fibre $\Omega^\Gamma (X,x)$. Since $\paths(X,x)$ is contractible, one finds from the long exact sequence that for $k\geq1$, 
$$\pi_k(\Omega^\Gamma (X,x),(\gamma,g))\simeq \pi_{k+1}(X,x).$$  
And by definition $\pi_0(\Omega^\Gamma (X,x),(\gamma,g)) = \pi_1^\Gamma (X,x)$ (as $\Gamma $ has the discrete topology).

Consider now the fibration 
\begin{equation}\label{eq:G-loop fibration}
{\begin{pspicture}(-3,0)(2,1)\psset{yunit=0.7,arrowscale=1.3}
\rput(-2.3,1){$\Omega^\Gamma (X,x)$} \rput(-1,0.9){${\lhook\joinrel\longrightarrow}$} 
\rput(0,1){$\Lambda^\Gamma X$} \rput(2,1){$(\gamma,\,g)$}
\psline{->}(0,0.5)(0,-0.5) \psline{|->}(2,0.5)(2,-0.5)
\rput(0,-1){$X$} \rput(2,-1){$\gamma(0)=g^{-1}\gamma(1)$}
\end{pspicture}}
\bigskip
\bigskip
\end{equation}
The fibre over the point $x$ is $\Omega^\Gamma (X,x)$.  The long exact sequence associated to this fibration ends with
\begin{equation} \label{eq:fibration l.e.s}
 \begin{array}{c}
\hbox to 0pt{\hss$\cdots$}\lra \pi_2(\Omega^\Gamma (X,x),\,(\gamma,g)) \lra \pi_2(\Lambda^\Gamma X,\,(\gamma,g)) \lra \pi_2(X,x) \lra \\[6pt]
\qquad\lra \pi_1(\Omega^\Gamma (X,x),\,(\gamma,g)) \lra \pi_1(\Lambda^\Gamma X,\,(\gamma,g)) \lra \pi_1(X,x) \lra \\[6pt]
\qquad\qquad\lra \pi_0(\Omega^\Gamma (X,x),\,(\gamma,g)) \lra \pi_0(\Lambda^\Gamma X,\,(\gamma,g)) \lra \pi_0(X,x) = \trivial,
\end{array}
\end{equation}
where we have assumed $X$ is path connected.

Since the topology on $\Gamma $ is discrete, we have $\pi_0(\Omega^\Gamma (X,x),\,(\gamma,g))= \pi_1^\Gamma (X,x)$, so that the last few terms of the sequence above become,
\begin{equation}\label{eq:4-term les}
\cdots \lra \pi_1(\Lambda^\Gamma X,\,(\gamma,g)) \lra \pi_1(X,x) \stackrel{\eps}{\lra} 
 \pi_1^\Gamma (X,x) \lra \pi_0(\Lambda^\Gamma X,\,(\gamma,g)) \lra  \trivial. 
\end{equation}

\begin{lemma} \label{lemma:alpha-map}
The map $\eps$ is given by conjugation in $\pi_1^\Gamma (X,x)$ by elements of $\pi_1(X,x)$:
\begin{equation} \label{eq:alpha}
\eps(\eta) = (\eta,e)^{-1}(\gamma,g)(\eta,e) = (\bar\eta*\gamma*(g\eta),\,g). 
\end{equation} 
where $\bar\eta$ is the reverse path of $\eta$.
\end{lemma}

\begin{proof}
The map $\eps:\pi_1(X,x) \to \pi_0(\Omega^\Gamma (X,x),\,(\gamma,g))=\pi_1^\Gamma (X,x)$ is the effect of lifting a loop $\eta\in\Omega(X,x)$ in the fibration (\ref{eq:G-loop fibration}). Let $t_0\in[0,1]$ and put $y=\eta(t_0)$.  The $g$-loop $\bar\eta_y*\gamma*(g\cdot \eta_y)$, where $\eta_y(t) = \eta(t_0t)$ for $t\in[0,1]$ (see figure below), provides a deformation of $\gamma$, continuous in $t_0$,  with base-point $y$. 
Letting $t_0$ increase until $t_0=1$ gives the required expression.
\begin{center}
\begin{pspicture}(0,-0.5)(4,2)
\psset{arrowscale=1.8}
\psdot(0,0) 
\rput(0,-0.5){$x$} \rput(6,-0.5){$g\cdot x$}
\psdot(-0.5625,1.125) \rput(-0.8,1){$y$}
\psbezier(0,0)(2,1)(4,-1)(6,0) \psline{->}(2.6,0.1)(3,0) \rput(3,-0.3){$\gamma$}
\psbezier[linestyle=dashed](0,0)(-2,2)(2,2)(0,0) \psline{->}(-0.1,1.48)(0,1.5)
\parametricplot{0}{0.3}{6 t mul 1 t sub mul 2 t mul 1 sub mul 6 t mul 1 t sub mul}
 \rput(0,1.75){$\eta$}
\rput(6,0){
\psdot(0,0) 
\psdot(-0.5625,1.125) \rput(-1,1){$g\cdot y$}
\psbezier[linestyle=dashed](0,0)(-2,2)(2,2)(0,0) \psline{->}(-0.1,1.48)(0,1.5) \parametricplot{0}{0.3}{6 t mul 1 t sub mul 2 t mul 1 sub mul 6 t mul 1 t sub mul}
\rput(0,1.75){$g\cdot\eta$}
}
\end{pspicture}
\end{center}
\end{proof}

It follows from the exact sequence (\ref{eq:4-term les}) that two $g$-loops $\gamma,\delta\in\Omega^g(X,x)$ are in the same connected component of $\Lambda^gX$ if and only if there is a loop $\eta\in\Omega(X,x)$ and a homotopy
$$\delta \;\sim\; \bar\eta*\gamma*(g\eta).$$

If we assume that $X$ is \emph{aspherical}, which will be true of the applications to choreographies, then we can easily deduce more (a space is aspherical if its universal cover is contractible; these are also known as Eilenberg-MacLane $K(\pi,1)$ spaces).  In this case $\pi_k(X,x)=\trivial$ for $k\geq 2$, and it follows that $\pi_k(\Omega^\Gamma (X,x))=\trivial$ for $k\geq1$ and the long exact sequence above then implies $\pi_k(\Lambda^\Gamma X)=\trivial$ for $k\geq2$.  In particular, the `$\cdots$' at the start of (\ref{eq:4-term les}) can be replaced by $\trivial$, and so $\pi_1(\Lambda^\Gamma X,\,(\gamma,g)) $ is isomorphic to the kernel of $\eps$.

The preceding discussion can be summarized in the following statement, which is a restatement in terms of the equivariant fundamental group of two theorems in \cite{MMRS05}.  The theorem will be applied via Proposition \ref{prop:fix is Lambda^g} to spaces of symmetric loops. 

\begin{theorem} \label{thm:connected components}
Suppose $X$ is a manifold with an action of a group $\Gamma $.  For each $g\in\Gamma$, the relative loop space $\Lambda^gX$ enjoys the following properties.
\begin{enumerate}
\item The connected components of\/ $\Lambda^gX$ are in 1--1 correspondence with the set of orbits of the action of\/ $\pi_1(X,x)$ on $\beta^{-1}(g)\subset \pi_1^\Gamma (X,x)$ by conjugation.
\item If in addition $X$ is aspherical then the connected component of\/ $\Lambda^gX$ containing\, $\gamma$ is aspherical with fundamental group isomorphic to the `partial centralizer'
$$Z_{\pi_1^\Gamma (X,x)}((\gamma,g)) \;\cap \; \pi_1(X,x),$$
where $Z_{\pi_1^\Gamma (X,x)}((\gamma,g)) $ is the centralizer of the element $(\gamma,g)$ in $\pi_1^\Gamma (X,x)$.
\end{enumerate}
\end{theorem}

\begin{proof}
(1) The exact sequence (\ref{eq:4-term les}) together with Lemma \ref{lemma:alpha-map} shows that two $g$-loops $\gamma,\delta\in\Omega^g(X,x)$ belong to the same connected component of $\Lambda^g X$ if and only if there exists $\eta\in\pi_1(X,x)$ such that 
$$(\delta,g) = (\eta,e)^{-1}(\gamma,g)(\eta,e).$$
That is, if and only if $(\gamma,g)$ and $(\delta,g)$ lie in the same orbit of the action of $\pi_1(X,x)$ on the coset $\beta^{-1}(g)\subset \pi_1^\Gamma (X,x)$. Different connected components therefore correspond to different orbits of this action.

(2) Since $X$ is aspherical, we have $\pi_k(\Omega^\Gamma (X,x),\,(g,\gamma))=0$ for $k\geq1$, so the long exact sequence (\ref{eq:fibration l.e.s}) shows that $\pi_k(\Lambda^\Gamma X,\,(\gamma,g))=0$ for $k\geq2$, and $\pi_1(\Lambda^\Gamma X,\,(\gamma,g))$ is isomorphic to $\ker\eps$. By Lemma \ref{lemma:alpha-map} this kernel is the required `partial centralizer'.
\end{proof}

Before turning to the case in hand of $n$-body choreographies, we illustrate the theorem with two examples.

\begin{example} (Taken from \cite{MMRS05}.)
Consider $\Gamma =\ZZ_2=\left<\kappa\mid\kappa^2=e\right>$ acting on $X=\bb{T}$ by $\kappa\theta = -\theta$. Then using the fixed point $\theta=0$ as base point, one sees that $\pi_1^\Gamma (\bb{T} ,0)=\ZZ\rtimes\ZZ_2$, with $(a,\kappa)(b,\kappa) = (a-b,e)$. Thus $(a,e)(b,\kappa)(a,e)^{-1} = (b+2a,\kappa)$.  It follows that, with $\tau(\kappa)=\nfrac12$ we have $\Fix(\Gamma _\tau,\Lambda X)\simeq\Lambda^\kappa X$ and this has two connected components, determined by the parity of $b\in\ZZ$ in $(b,\kappa)$; both components are contractible. 
\end{example}

\begin{example}
For a more interesting example, let $X$ be the cylinder $\bb{T}\times\RR$ with a single puncture at $(\nfrac12,0)$. The fundamental group is the free group on two generators $F_2$. Choosing $x=(0,0)$ as a base point, the generators are the two loops $a$ and $b$, one above the puncture and one below, both chosen to go round the circle $\bb{T}$  once in the positive direction.  Let $\Gamma =\ZZ_2$ act by reflection in the equator, so $\kappa\cdot(\theta,r)=(\theta,-r)$.  Since $x$ is fixed by this action, we have $\pi_1^\Gamma (X,x)\simeq F_2 \rtimes\ZZ_2$, with $\ZZ_2$ acting on $F_2$ by $\kappa a=b$ and $\kappa b=a$.
If we let $\tau:\Gamma \to \Sone $ be the only non-trivial homomorphism (so $\tau(\kappa)=\nfrac12$), we have 
$$\Fix(\Gamma_\tau,\LX) \simeq \Lambda^\kappa X,$$
and the latter space has connected components in 1--1 correspondence with the orbits of the action of $F_2$ on $\beta^{-1}(\kappa)<\pi_1^\Gamma (X,x)$. This action is
$$(w,e)(z,\kappa)(w^{-1},e) = (wz\bar w^{-1}, \kappa)$$
where $w$ is any element of $F_2$, so any word in $a,b$, and $\bar w$ is the same word but with $a$ replaced everywhere by $b$ and vice versa. Thus two $\kappa$-loops $z,z'$ based at $x$ are in the same connected component of $\Fix(\Gamma _\tau,\LX)$ if and only if there is a $w\in F_2$ such that $z'=wz\bar w^{-1}$. For example, all $\kappa$-loops of the form $w\bar w^{-1}$ are in the same connected component as the trivial loop at $x$. Moreover, that connected component is contractible by Theorem~\ref{thm:connected components}(2) since it is aspherical with fundamental group $\{w\in F_2\mid w\bar w^{-1}=e\}$, which is the trivial group. On the other hand, the connected component containing the $\kappa$-loop $a$ is not contractible, as $ab(a)\bigl(\overline{ab}\bigr)^{-1} = ab(a)a^{-1}b^{-1}=a$, so all powers of $ab$ lie in the partial centralizer $Z_{F_2\rtimes Z_2}((a,\kappa)) \cap F_2$. 
\end{example}

\begin{remark} \label{rmk:Reidemeister} 
The theorem above was written in terms of `twisted conjugacy' in \cite{MMRS05} (also known as Reidemeister conjugacy).  Let $\pi$ be a group and $\phi$ an endomorphism of $\pi$. Two elements $\gamma,\gamma'\in\pi$ are said to be $\phi$-twisted conjugate if there is a $\delta\in\pi$ such that $\gamma'=\delta\gamma\phi(\delta^{-1})$. In the first example above, $\phi(a)=-a$ (additively) and in the second $\phi(w)=\bar w$.

The description in \cite{MMRS05} is related to the present approach as follows. For $g\in \Gamma $ let $\omega=\omega_g$ be a fixed path from $x$ to $g\cdot x$. Then we can identify the coset $\beta^{-1}(g)\subset\pi_1^\Gamma (X,x)$ with $\pi_1(X,x)$ by making $\zeta\in\pi_1(X,x)$ correspond to $(\zeta,e) (\omega_g,g) = (\zeta*\omega_g,g)$ in $\beta^{-1}(g)$. The expression for $\eps$ in (\ref{eq:alpha}) above using this identification is, for $\eta\in\pi_1(X,x)$, 
$$\eps(\eta) = \bar\eta*(\zeta*\omega_g)*(g\eta)*\overline{\omega_g}.$$
Thus $\eps(\eta) = \eta^{-1}\zeta\phi(\eta)$ which is precisely Reidemeister conjugation of $\zeta$, where $\phi$ is the automorphism
$$\phi(\eta) = \omega_g*(g\eta)*\overline{\omega_g}\,.$$
Further details can be found in \cite[Sec.~2.5]{Steckles}.  
\end{remark}

\section{Choreographies and braids}
\label{sec:choreography components}

We now apply the results of the previous section to the case in hand of $n$ distinct points in the plane and in particular to the choreographic loops.  We are interested in the action of finite subgroups $G$ of $\Gamma=\OO(2)\times S_n$ on $\Xn$ and we use some of the properties above to find the equivariant fundamental group $\pi_1^G(\Xn,x)$. (We ignore the temporal part $\Sone$ here as its action on $\Xn$ is trivial.)

\subsection{Configuration space and braids}
First we recall some facts about braid groups, and introduce some notation. Two useful references are the books by Kessel and Turaev \cite{KTbook} and Farb and Margalit \cite{FarbMargalit}. It was first observed by Fox and Neuwirth \cite{FoxN} that the fundamental group of the space $X^{(n)}$ is (isomorphic to) the pure braid group $P_n$, and it was moreover proved by Fadell and Neuwirth \cite{FadellN} that $\Xn$ is aspherical.  Recall that the braid groups sit in a short exact sequence
\begin{equation}\label{eq:braid ses}
\trivial \to P_n\lra B_n\stackrel{\pi}{\lra} S_n\to \trivial,
\end{equation}
where $B_n$ is the full braid group on $n$ strings, and $S_n$ the permutation group. Furthermore, $S_n$ acts free  ly on $\Xn$ and the fundamental group of the quotient space $\Xn/S_n$ is the full braid group $B_n$ (as observed in \cite{FoxN}).  It follows that the equivariant fundamental group $\pi_1^{S_n}(\Xn,x)$ is isomorphic to the braid group $B_n$, and that (\ref{eq:EFG ses}) becomes (\ref{eq:braid ses})  (i.e., $\beta$ becomes $\pi$). Here we are taking the usual base point considered for braid groups, namely $x$ is any point in $\Xn$ where the $z_j$ are placed in sequence along the real line. 

Let us denote\footnote{these are usually denoted $\sigma_1,\dots,\sigma_{n-1}$, but this would conflict with our use of $\sigma_i$ as permutations} the generators of the braid group $B_n$ by  $\braid_1,\braid_2,\dots,\braid_{n-1}$, where $\braid_i$ represents the crossing of string $i+1$ over (in front of) string $i$. In terms of motion of points in the plane, $\br_1$ represents the clockwise interchange of points $i$ and $i+1$.  The centre of both groups $B_n$ and $P_n$ is infinite cyclic, generated by the \emph{full twist}  $\Delta^2$, where 
$$\Delta = (\braid_{1}\braid_{2} \dots \braid_{n-1}) (\braid_{1}\braid_{2} \dots \braid_{n-2}) \cdots (\braid_{1}).$$ 
The full twist represents a clockwise rotation through $2\pi$ of the whole braid.

Denote by $\delta$ the element $\delta=\braid_1\braid_2\dots\braid_{n-1}$ 
(see Figure \ref{fig:braid delta5}).  In the exact sequence (\ref{eq:braid ses}) we have $\pi(\delta)=\sigma_1 = (1\;2\;\dots\;n)$. It  follows that $\delta^n\in P_n$ and in fact $\delta^n=\Delta^{2}$ (as is not hard to see geometrically).

\begin{remark}\label{rmk:braids to permutations}
  Some care should taken to express correctly the permutation associated to a given braid, since we are using the usual \emph{left} action of the permutation group. Now, as homotopy classes of curves in $\Xn$, the product $\br\br'$ is the class $[\br*(\sigma\cdot\br')]$, where $\sigma=\pi(\br)$.  It follows that $\pi(\br\br')=\pi(\br)\circ\pi(\br')$ (so first applying $\pi(\br')$ on $\Xn$ and then $\pi(\br)$ --- this is consistent with the usual relation between free group actions and fundamental groups, as described for example in Spanier \cite[Chap.~2]{Spanier}). In particular, the permutation associated to the braid $\delta=\br_1\br_2\br_3\br_4\in B_5$ shown in Fig.~\ref{fig:braid delta5} is $\pi(\delta)=(1\;2\;3\;4\;5)$ (rather than the usual $(5\;4\;3\;2\;1)$). 
\end{remark}

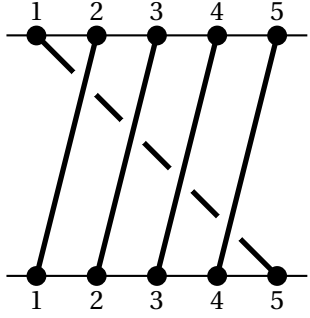
\begin{figure}
\begin{center}\psset{unit=0.8}
 \begin{pspicture}(-2,-2)(2,2.5)
\psline(-2.5,2)(2.5,2) 
  \rput(-2,2.4){1} \rput(-1,2.4){2} \rput(0,2.4){3} \rput(1,2.4){4} \rput(2,2.4){5}
 \psdots[dotsize=5pt](-2,2)(-1,2)(0,2)(1,2)(2,2)
  \psline[linewidth=2pt](2,-2)(-2,2)  
\psdots[dotsize=8pt,dotstyle=o,linecolor=white,fillcolor=white](1.2,-1.2)(0.4,-0.4)(-0.4,0.4)(-1.2,1.2) 
  \psline[linewidth=2pt](-2,-2)(-1,2)  
  \psline[linewidth=2pt](-1,-2)(0,2)  
  \psline[linewidth=2pt](0,-2)(1,2)  
  \psline[linewidth=2pt](1,-2)(2,2)  
 \psline(-2.5,-2)(2.5,-2) 
 \psdots[dotsize=5pt](-2,-2)(-1,-2)(0,-2)(1,-2)(2,-2)  
  \rput(-2,-2.4){1} \rput(-1,-2.4){2} \rput(0,-2.4){3} \rput(1,-2.4){4} \rput(2,-2.4){5}
\end{pspicture}
\end{center}
\caption{The braid $\delta$ for $n=5$. One has $\pi(\delta)=(1\;2\;3\;4\;5)=\sigma_1$.}
\label{fig:braid delta5}
\end{figure}

There is a homomorphism $\chi:B_n\to\ZZ$, generated by $\chi(\braid_i)=1$ (for all $i$), and called the \emph{exponent sum} of the braid; it measures the algebraic number of crossings in the braid.  For example  $\chi(\delta)=n-1$ while $\chi(\Delta)=\half n(n-1)$.  The important property in the present context is that $\chi$ is invariant under conjugation, as is easy to see.   (More abstractly, the first homology group of $B_n$ is $\ZZ$, and $\chi$ is the natural projection $B_n\to H_1(B_n)\simeq B_n/[B_n,B_n]$.)

We define two subgroups of the braid group which will be useful when we consider choreographies with non-trivial core below.

\begin{definitions} \label{def:winding braids}
Firstly, we let $B_{n,1}$ be the group
\begin{equation}\label{eq:B_{n,1}}
B_{n,1} = \{\br\in B_{n+1}\mid \pi(\br)(0)=0\}
\end{equation}
where the $n+1$ strings are labelled $0,1,\dots,n$. Its generators are described in the review by Vershinin \cite{Vershinin-loops} (where $B_{n,1}$ is denoted $Br_{1,n+1}$).

Secondly, we define a subgroup  $P_{n,1(c)}$ of $P_{n+1}$ as follows: again label the strings $0,\dots, n$. Then for each pure braid $p\in P_{n+1}$ there are integers $w_{ij}(p)$ given by the number of times string $i$ winds around string $j$.  Then put
$$P_{n,1(c)} = \{p\in P_{n+1} \mid w_{j0}(p) \equiv 0 \bmod c,\;\forall j=1\dots,n\}.$$
That is, the number of times each string winds around the 0-string is an integer multiple of $c$. (This also has a homological interpretation, as the winding numbers of a loop are the components of the image of the element of $P_n$ in its first homology group: $H_1(P_n) \simeq \ZZ^{n(n-1)/2}$.)
\end{definitions}

Recall that $\sigma_1\in S_n$ is the permutation $\sigma_1(j)=j+1\bmod n$.

\begin{proposition}\label{prop:aspherical Y}
Suppose $c$ divides $n$ with $1<c<n$, and $\ell$ is coprime to $c$. Let $K$ be the subgroup of\/ $\Gamma$ generated by $(R_{2\pi\ell/c}, \,\sigma_1^{\nhat})$, where $\nhat=n/c$, so $K$ is cyclic of order $c$.  The space $Y=\Fix(K,\Xn)$ is connected and aspherical, with fundamental group
$\pi_1(Y,y) \simeq P_{\nhat,1(c)}<P_{\nhat +1}$. %
\end{proposition}

\begin{proof}
We saw in Sec.~\ref{sec:non-trivial core}  that $Y$ is a cover of $\Xnhat$, of degree $c^{\nhat}$.    The homotopy type of $\Xnhat$  is found from the fibration $X^{(\nhat+1)}\to\bb{C}$ given by 
$$(z_0,z_1,\dots,z_{\nhat})\longmapsto z_0.$$ 
The fibre over $0$ is $\Xnhat$ and since the base is contractible, the total space and the fibre are homotopic. That is, $\Xnhat \sim X^{(\nhat +1)}$, and the latter is aspherical with fundamental group $P_{\nhat+1}$ as already pointed out.  Now consider the covering $\Psi:Y\to\Xnhat$ of degree $c^{\nhat}$ given in \eqref{eq:Psi}, which is in fact the quotient map for the action of $\ZZ_c^{\nhat}$ acting by multiplication by $c^{\mathrm{th}}$ roots of unity:
\begin{equation}\label{eq:Zcnhat action}
 (\omega_1,\dots,\omega_{\nhat})\cdot(z_1,\dots,z_n) = (\omega_{1}z_1,\dots,\omega_{\jhat}z_j,\dots,\omega_{\nhat}z_n) \quad (\omega_j\in\ZZ_c),
\end{equation}
where $\jhat=j\bmod\nhat$. 
This gives the short exact sequence,
$$\trivial\to\pi_1(Y,x) \longrightarrow \pi_1(\Xnhat,\Psi(x)) \longrightarrow \ZZ_c^{\nhat}\to\trivial.$$
The projection to $\ZZ_c^{\nhat}$ is the winding number mod $c$ of each of the strings (particles) around the origin, so in order for a loop to be in $\pi_1(Y,x)$ these winding numbers must vanish, modulo $c$. 
\end{proof}

For the following theorem, we need to define an involution $\br\mapsto \bar \br$ on $B_n$:  given any braid $\br\in B_n$, then $\bar\br$ is the same braid but with every overcrossing changed into an undercrossing and vice versa. This is a homomorphism, and clearly this leaves the resulting permutation unaffected: $\pi(\bar\br) = \pi(\br)$. Algebraically, let $\br$ be any word written in terms of the generators $\braid_j$, then $\bar\br$ is the same word with each instance of a generator replaced by its inverse. It follows that the exponent sum satisfies $\chi(\bar\br)=-\chi(\br)$. For example, if $\br=\braid_i\braid_j^{-2}$ then $\bar\br = \braid_i^{-1}\braid_j^2$. One can show easily that $\bar\Delta=\Delta^{-1}$. 

Recall the notation $\times_{\alpha}$ and the curve $(-\theta)_{g\cdot x}$ defined in Example~\ref{ex:pi_1 S^1}. Recall also that $R_{2\pi\theta}$ is the rotation through $2\pi\theta$ in the plane, and $\kappa$ is the reflection in the $x$-axis (complex conjugation).

\begin{theorem} \label{thm:EFG for Gamma}
For\/ $\Gamma=\OO(2)\times S_n$ acting on $X^{(n)}$ with $x$ the usual base point for braids, there is an isomorphism
\begin{eqnarray*}
\qquad  \pi_1^\Gamma(\Xn,x) &\stackrel{\sim}{\lra}&     
  (B_n\times_{\alpha}\RR)\rtimes\mathbb{Z}_2\\[6pt]
    (\gamma,\;g) 
  &\longmapsto& 
  \left\{\begin{array}{ll}
    ([\gamma*(-\theta)_{g\cdot x},\, \theta], e) & \mbox{if }\; g=(R_{2\pi \theta},\,\sigma)\\
    ([\gamma,\, 0], \kappa) & \mbox{if }\; g=(\kappa,\sigma),
  \end{array}\right.
\end{eqnarray*}
where $\sigma\in S_n$. The image of the remaining elements, so those of the form $g=(R_{2\pi\theta}\kappa,\,\sigma)$, can be determined using the semidirect product structure coming from $\kappa[\br,s] = [\bar\br,-s]$.  The homomorphism $\alpha:\ZZ\to B_n$ is given by $\alpha(n)=\Delta^{2n}$.  In terms of this isomorphism, the projection $\beta:\pi_1^\Gamma\to\Gamma$ is given by
$$\beta([\br,\theta],\kappa^r) = (R_{2\pi \theta}\kappa^r,\,\pi(\br))\in\OO(2)\times S_n,$$
where $\pi:B_n\to S_n$ is the usual projection. 
\end{theorem}

\begin{proof} We use Proposition~\ref{prop:EFG2} and Eq.~\eqref{eq:EFG-normal subgp}, with $N = S_n \lhd \OO(2)\times S_n$.  Firstly, $S_n$ acts freely on $\Xn$ so that
\begin{equation}\label{eq:EFG for S_n action}
\pi_1^{S_n}(\Xn,x) \simeq \pi_1(\Xn/S_n,\,\bar x) \simeq B_n.
\end{equation}

Now apply Proposition~\ref{prop:EFG2} to the $\SO(2)$-action to obtain (we identify $\Sone\stackrel{\sim}{\lra}\SO(2)$ by $\theta\mapsto R_{2\pi\theta}$)
\begin{equation}\label{eq:EFG for SO(2) x S_n}
\pi_1^{\SO(2)\times S_n}(\Xn,\,x) \simeq \pi_1^{\SO(2)}(\Xn/S_n,\hat x) \simeq \pi_1(\Xn/S_n, \hat x) \times_{\alpha}\RR.
\end{equation}
The second isomorphism, by Example~\ref{ex:pi_1 S^1}, is $(\hat\gamma,\,\theta) \mapsto [\hat\gamma*(-\theta)_{\theta\cdot x},\;\theta]$, where $\hat\gamma$ is a path in $\Xn/S_n$ with phase $\theta$.  Let $\gamma$ be the unique lift of the path $\hat\gamma$ to $\Xn$ with base point $x$. Since $\hat\gamma(1) = \theta\cdot\hat x$, we have that $\gamma(1)=\sigma\cdot\theta\cdot x = \theta\cdot\sigma\cdot x$ for some $\sigma\in S_n$. That is, $\gamma(1)=g\cdot x$ for $g=(R_{2\pi\theta},\sigma)\in\SO(2)\times S_n$. Consequently, $\gamma*(-\theta)_{g\cdot x}$ is a path from $x$ to $\sigma\cdot x$ and so represents a braid, as required.

The final result then follows by using $\kappa\in\OO(2)$---the reflection in the $x$-axis (real line) which fixes the base point and makes every overcrossing into an undercrossing and vice versa, so changing each $\braid_j$ into $\braid_j^{-1}$.

For the homomorphism $\beta$, the element $s\in\RR$ corresponds to rotation through an angle of $2\pi s$ and $\kappa$ to the reflection in the $x$-axis.  The choice of semi-direct product structure then determines the first component of $\beta$.  The second component $\pi(b)$ is determined by the projection associated to Eq.\,\eqref{eq:EFG for S_n action} above.
\end{proof}

For future reference, if $p\in P_n$, which we consider as the element $(p,0,e)\in (B_n\times_{\alpha}\RR)\rtimes\ZZ_2$, then conjugation by $p$ is given by,
\begin{equation}\label{eq:conjugation by pbraids}
p(\br ,\theta,e)p^{-1} = (p\br p^{-1},\theta,e)\quad\mbox{and}\quad 
p(\br ,\theta,\kappa)p^{-1} = (p\br \bar p^{-1},\theta,\kappa).
\end{equation}
The first of these just involves the usual conjugacy in the braid group, while the second uses a twisted (Reidemeister) conjugacy: in the notation of Remark~\ref{rmk:Reidemeister} it uses $\phi(p)=\bar p$.

\begin{example} \label{ex:EFG for trivial core} 
The equivariant fundamental groups for the non-reversing symmetry groups without core are subgroups of the group $\pi_1^\Gamma(\Xn,x)$ found in Theorem~\ref{thm:EFG for Gamma}. For the cyclic group $\Sigma_n<S_n$ (the `spatial component' of the choreography group $\Chor_n$), one has
\begin{equation} \label{eq:CB_n}
 \pi_1^{\Sigma_n}(\Xn,x) \simeq CB_n,
\end{equation}
a group we call the \defn{cyclic braid group}, and which is the subgroup of $B_n$ generated by $P_n$ and $\delta$; it is of course the inverse image under $\pi:B_n\to S_n$ of the cyclic group $\Sigma_n$.  The homomorphism $\beta:\pi_1^{\Sigma_n}\to \Sigma_n$ here is just the restriction of $\pi$ to this cyclic braid group. 
\end{example}

Combining this example with the $\SO(2)$ action we can find $\pi_1^G(\Xn,x)$ for $G=C(n,k/\ell)$, using Example~\ref{ex:pi_1 S^1}. Recall that provided $(n,k)=1$, $G$ is cyclic of order $nk$ and generated by $g=(R_{2\pi a\ell/k}, \sigma_1^b)$, where $an-bk=1$.

\begin{corollary}\label{coroll:EFG for trivial core} 
(1) Let  $G=C(n,k/\ell)$ with $(n,k)=1$.  There is an isomorphism
\begin{equation}\label{eq:EFG for coprime n,k}
\pi_1^G(\Xn,x) \stackrel{\sim}{\lra} CB_n\times_{\alpha}\ZZ(1/k),
\end{equation}
where $\ZZ(1/k)$ is the subgroup of\/ $\RR$ consisting of integer multiples of $1/k$. Given $(\gamma,g)\in\pi_1^G(X,x)$ for $g=(R_{2\pi a\ell/k},\sigma_1^b)$, the corresponding element of the right hand side is
$$[\gamma*(-\theta)_{g\cdot x},\; a\ell/k],$$ 
for $\theta=a\ell/k\in\ZZ(1/k)$ (the path $\theta_p$  is defined in Example~\ref{ex:pi_1 S^1}).  Under this isomorphism, the projection $\beta:\pi_1^G\to G<\SO(2)\times\Sigma_n$ is given by
$$\beta([\br, \theta]) = (R_{2\pi\theta}, \pi(\br)).$$  

\medskip

(2) Let $G=C'(n,2)$, which is cyclic of order $2n$ and its projection to $\Gamma$ is generated by $(\kappa,\sigma_2)\in\Gamma$.  There is an isomorphism
\begin{equation}\label{eq:EFG for C'}
\pi_1^G(\Xn,x) \stackrel{\sim}{\lra} CB_n\rtimes \ZZ_2
\end{equation}
where the semidirect product comes from $\kappa\cdot\br=\bar\br$.  Any relative loop $\gamma$ with $g=(\kappa, \sigma_2)$ maps under the isomorphism to $(\gamma,\kappa)$, and $\beta(\gamma,\kappa^r) = (\kappa^r,\pi(\gamma))$.
\end{corollary}

\begin{proof}
Since the groups are subgroups of $\Gamma$, these statements are all immediate consequences of the more general Theorem~\ref{thm:EFG for Gamma}. 
\end{proof}

\subsection{Choreographies}

We now return to the case of choreographies, where the equivariant loops are those arising from the classification of Section~\ref{sec:classification}. Since we are looking at non-reversing symmetry groups in this section, there are two cases: $C(n,k/\ell)$ and, if $n$ is odd, $C'(n,2)$, defined in Sec.~\ref{sec:symmetry types}. (We consider the time-reversing symmetries in a separate paper.)

Recall that for $C(n,k/\ell)$ we denote the core by $K$: it is the pointwise isotropy of a loop with this symmetry and is the cyclic group of order $c=(n,k)$ generated by $(R_{2\pi\ell/c},\,\sigma_1^{\nhat})\in\Gamma$.  If $c=1$ or if $G=C'(n,2)$ then $K$ is trivial.

Combining Propositions~\ref{prop:fix is Lambda^g}, \ref{prop:aspherical Y}, and Theorem~\ref{thm:connected components} we deduce the main result of this section.  For the part with non-trivial core, it will be useful to use the subgroup $B_{\nhat,1}$ defined in Eq.~\eqref{eq:B_{n,1}}. Note that the  projection $\pi:B_{\nhat+1}\to S_{\nhat+1}$ restricts to $B_{\nhat,1}\to S_{\nhat}$. 
Define the element $\deltahat =  \braid_0^{2}\braid_1\braid_2\dots\braid_{\nhat-1}\in B_{\nhat,1}$ (see Figure~\ref{fig:deltahat}), and note that $\pi(\deltahat)=\sigmahat_1 = (1\;2\;\dots\;\nhat)\in S_{\nhat+1}$. In fact, though we don't use it, it is well known that $(\deltahat)^{\nhat-1}=\Delta^2$.

\begin{figure}[t] 
\begin{center}\psset{unit=0.8}
 \begin{pspicture}(-3,-2)(3,2.5)
\psline(-2.5,2)(2.5,2) 
  \rput(-2,2.4){0} \rput(-1,2.4){1} \rput(0,2.4){2} \rput(1,2.4){3} \rput(2,2.4){4}
 \psdots[dotsize=5pt](-2,2)(-1,2)(0,2)(1,2)(2,2)
  \pscurve[linewidth=2pt](2,-2)(-1,-0.5)(-2.5,0.8)(-1,2)  
\psdots[dotsize=6pt,dotstyle=o,linecolor=white,fillcolor=white](-2,0.1)(-0.7,-0.7)(0.2,-1.15)(1.1,-1.58) 
  \psline[linewidth=2pt](-1,-2)(0,2)  
  \psline[linewidth=2pt](0,-2)(1,2)  
  \psline[linewidth=2pt](1,-2)(2,2)  
  \psline[linewidth=2pt](-2,-2)(-2,1.2)  
  \psline[linewidth=2pt](-2,1.6)(-2,2)  
   \psline(-2.5,-2)(2.5,-2) 
 \psdots[dotsize=5pt](-2,-2)(-1,-2)(0,-2)(1,-2)(2,-2)  
  \rput(-2,-2.4){0} \rput(-1,-2.4){1} \rput(0,-2.4){2} \rput(1,-2.4){3} \rput(2,-2.4){4}
\end{pspicture}
\end{center}
\caption{The braid $\deltahat\in B_{\nhat,1}$ for $G=C(n,k/\ell)$ with $(n,k)=c>1$, shown here for $\nhat=4$; note that  $\pi(\deltahat)=(1\;2\;3\;4) =\sigmahat_1$. See Remark~\ref{rmk:braids to permutations}} 
\label{fig:deltahat}
\end{figure}

\begin{theorem} \label{thm:connected components of choreographies}
Let $G$ be any of the symmetry groups $C(n,k/\ell)$ or $C'(n,2)$. \\
(1) The set of connected components of the space $\Fix(G,\Lambda\Xn)$ is in 1-1 correspondence with the following sets:
\begin{itemize}
\item for $G=C(n,k/\ell)$ with $(n,k)=1$, the set of\/ $P_n$-conjugacy classes in the coset $\delta^bP_n\subset CB_n$, where $bk\equiv -1 \bmod n$;
\item for $G=C(n,k/\ell)$ with $(n,k)=c>1$, the set of\/ $P_{\nhat,1(c)}$-conjugacy classes in the coset $\deltahat^bP_{\nhat,1(c)}\subset B_{\nhat,1}$ where $\deltahat\in \braids_{\nhat,1}$ is defined above, and $bk\equiv -c\bmod n$; 
\item for $G=C'(2,n)$, the set of orbits of the twisted conjugacy action of\/ $P_n$ on the coset $\delta'P_n$, where $\delta'=\delta^{(n+1)/2}$ and the twisted conjugacy action is $p\cdot \br = p\br\bar p^{-1}$, for $p\in P_n$, and the definition of $\bar p$ precedes Theorem~\ref{thm:EFG for Gamma}.
\end{itemize}

(2) The connected component containing $u\in\Fix(G,\Lambda X)$ is aspherical with fundamental group isomorphic to the group $\Pi_u$ defined as follows:
\begin{itemize}
\item for $G=C(n,k/\ell)$ with $(n,k)=1$ then 
$$\Pi_u=\{p\in P_{n} \mid p\br =\br p\}$$
where $\br \in\delta^b P_n<B_n$ is the braid corresponding to the path $t\mapsto u(t/nk)$ for $t\in[0,1]$ and $bk\equiv -1\bmod n$;  
\item for $G=C(n,k/\ell)$ with $(n,k)=c>1$ it is the analogue, with $\Xn$ replaced by $Y$ and so $p\in P_{\nhat,1(c)}$ in place of\/ $P_n$, 
$$\Pi_u=\{p\in P_{\nhat,1(c)} \mid p\br =\br p\}$$ 
where $\br \in\deltahat^bP_{\nhat,1(c)}<B_{\nhat,1}$ is the braid corresponding to the path $t\mapsto u(ct/nk)$ for $t\in[0,1]$ and $bk\equiv -c\bmod n$;  
\item  for $C'(n,2)$ it is 
$$\Pi_u=\{p\in P_n \mid p\br  = \br \bar p\},$$
where $\br \in B_n$ is the braid corresponding to the path $t\mapsto u(t/2n)$ for $t\in[0,1]$. 
\end{itemize}
\end{theorem}

For example, for choreographies with symmetry $\Chor=C(n,1)$, we have that $k=1$, so  $b=-1$. The relevant coset is therefore $\delta^{-1}P_n$. 

\begin{remarks}\label{rmk:cosets}
(a) It is easy to see that each of the cosets occurring in part (1) contains infinitely many conjugacy classes, since multiplication by powers of $\Delta^2$ preserves the coset, but changes the conjugacy class as is readily seen by considering the exponent sum (which is invariant under conjugacy).

\noindent(b) Note that in the theorem the $a,b$ are not unique.  However, different choices of $b$ differ by multiples of $n$, and $\delta^{b+rn}= \delta^b\Delta^{2r}$.  So firstly the cosets $\delta^bP_n$ and $\delta^{b+rn}P_n$ coincide, and secondly while the conjugacy classes are not the same, they are in (natural) 1--1 correspondence, because $\Delta^2$ is central.

\noindent(c) Finding centralizers of braid elements is an interesting problem in braid theory, see for example \cite{G-MW04}. However the known results do not address the twisted conjugacy problem relevant to $C'(n,2)$, although some general results do exist \cite{G-MV11}.
\end{remarks}

\begin{proof}
By Proposition~\ref{prop:fix is Lambda^g}, we know $\Fix(G,\Lambda\Xn)$ is homeomorphic to $\Lambda^g\Xn$ if the core is trivial, and to $\Lambda^gY = \Lambda^gY_K$ if the core $K$ is non-trivial, for some suitable $g\in \Gamma$ discussed below. In each case, given a loop $u\in\Fix(G,\Xn)$ one defines a suitable $g$-loop $\gamma$ according to Proposition~\ref{prop:fix is Lambda^g}. The set of connected components is then given in part (1) of Theorem~\ref{thm:connected components}, while the fundamental group of each component is given in part (2) of the same theorem.  We proceed by treating each class of symmetry group in turn, first the two cases with trivial core, and then the remaining groups with non-trivial core, which requires a more involved argument.

First suppose $G=C(n,k/\ell)$ with $(n,k)=1$. Here $G$ is cyclic of order $nk$ and its projection to $\Gamma$ is generated by  $g=(R_{2\pi a\ell/k},\, \sigma_1^b)\in \Gamma$, with $an-bk=1$.  Now $\pi_1^G$ is given in Corollary~\ref{coroll:EFG for trivial core}, and $(\gamma,g)$ represents an element $[\br , r/k]\in CB_n\times_{\alpha}\ZZ(1/k)\simeq\pi_1^G$ for some $\br \in \braids_n$ and $r\in\ZZ$. Then $\beta([\br ,\,r/k])=(R_{2\pi r\ell/k},\pi(\br ))\in \SO(2)\times S_n$ so $r=a$ and  $\pi(\br )=\sigma_1^b$ so that $\br $ belongs to the coset $\delta^bP_n$.  The final statement in this case then follows from the expression for conjugation given in \eqref{eq:conjugation by pbraids}.

Next suppose $u$ has symmetry $G=C'(n,2)$, which is the cyclic group of order $2n$ whose projection to $\Gamma$ is generated by $g'=(\kappa,\sigma_2)$ (recall $n$ is odd, so $\sigma_2=\sigma_1^h$ with $h={(n+1)/2}$).  Then $\beta^{-1}(g')\subset \pi_1^\Gamma(\Xn,x)$ is the coset $(\delta^hP_n,0,\kappa)<\pi_1^\Gamma(\Xn,x)$ (see Theorem~\ref{thm:EFG for Gamma} for notation). It then follows from the second part of \eqref{eq:conjugation by pbraids} that the conjugation by $p\in P_n$ is given by $\br\mapsto p\br\bar p^{-1}$.

Finally suppose $G$ has non-trivial core, so $G=C(n,k/\ell)$ with $(n,k)=c>1$. Then by Proposition~\ref{prop:fix is Lambda^g}  we have $\Fix(G,\Lambda\Xn)\simeq \Lambda^g Y$ for $g = (R_{2\pi a\ell/k},\, \sigma_1^b)$ with $an-bk=c$ ---see \eqref{eq:c/nk generator}. By Proposition~\ref{prop:aspherical Y} the fundamental group of $Y=\Fix(K,\Xn)$ is $P_{\nhat,1(c)}$, so that the connected components of $\Lambda^gY$ are in 1--1 correspondence with the $P_{\nhat,1(c)}$-conjugacy classes in the coset $\beta^{-1}(g)\subset\pi_1^{N/K}(Y,x)$, where $N$ is any group containing $K$ and $g$ and acting on $Y$.  It remains to identify a suitable group $N$ so that the commutation is as given in the theorem.

We take the group $N$ as follows. The core $K$ is generated by $(R_{-2\pi\ell(c)},\,\sigma_1^{\nhat})$, and $\sigma_1^{\nhat}$ is a product of $\nhat$ disjoint cycles of length $c$. For $j=1,\dots,\nhat$ let $\pi_j\in S_n$ be the cycle of length $c$
$$\pi_j = (j\;\;(\nhat+j)\;\;(2\nhat+j)\;\dots\;\;((c-1)\nhat+j)),$$
that is, $\pi_j(j+r\nhat) = j+(r+1)\nhat\bmod n$, for $r=0\dots c-1$.  Then as a product of disjoint cycles $\sigma_1^{\nhat} = \pi_1\pi_2\dots\pi_{\nhat}$.
Let $N<N_\Gamma(K)$ be 
$$N=\SO(2)\times(\ZZ_c^\nhat\rtimes S_{\nhat}).$$ 
The action on $Y$ is by $\SO(2)$ acting on the plane as usual, the $j^{\mathrm{th}}$ component of $\ZZ_c^{\nhat}$ acting by powers of the cycle $\pi_j$ (equivalently, on $Y$ by  multiplication by $c^{\mathrm{th}}$ roots of unity as in \eqref{eq:Zcnhat action}), and $S_{\nhat}$ acting by permuting the $\nhat$ disjoint cycles in $\sigma_1^{\nhat}$.  (The permutation part of $N$ is the wreath product $\ZZ_c\wr S_{\nhat}$.)  Clearly $K<N$ and there is therefore an action of $N/K$ on $Y$. Moreover, the subgroup $\ZZ_c^\nhat$ acts freely, and the quotient $Y/\ZZ_c^{\nhat}$ can be identified with $\Xnhat$ via the map $\Psi$ defined in \eqref{eq:Psi}.

Let $\Gammahat=\SO(2)\times S_{\nhat}$ act in the usual way on $\Xnhat$, and let $\psi:N/K\to\Gammahat$ be the surjective homomorphism $\psi(R,\sigma)=(R^c,\sigmahat)$, where $\sigmahat(j)=\sigma(J)\bmod \nhat$. The kernel of $\psi$ is precisely $\ZZ_c^{\nhat}$ and the map $\Psi$ is then equivariant with respect to $\psi$, meaning $\Psi(g\cdot x) = \psi(g)\cdot\Psi(x)$.

Since the $\ZZ_c^\nhat$-action is free, by Proposition~\ref{prop:EFG2} we have an isomorphism 
\begin{equation} \label{eq:EFG for Y}
\pi_1^{N/K}(Y,x) \simeq \pi_1^{\widehat{\Gamma}}(\Xnhat,x) = B_{{\nhat},1}\times_{\alpha}\RR, 
\end{equation}
where $\alpha(1)=\widehat{\Delta}^2$, where $\widehat{\Delta}^2$ is the full twist in $P_{\nhat+1}\subset B_{\nhat,1}$ (the proof of the final equality in \eqref{eq:EFG for Y} is the same as that for Theorem~\ref{thm:EFG for Gamma}, bearing in mind that only $S_{\nhat}$ acts and not $S_{\nhat+1}$, and we are not including the reflection $\kappa$).
Now with $\deltahat\in B_{\nhat,1}$ defined just before the statement of the theorem, we have $\pi(\deltahat) = \sigmahat_1$ so (similar to the case above) 
$$\beta^{-1}(R_{2\pi a\ell/k},\sigma_1^b) = \deltahat^b P_{\nhat,1(c)},$$
as required, where $\beta:\pi_1^{N/K}(Y,x)\to N/K$.

The fundamental group of the component containing $u$ is then deduced in the same way as before.
\end{proof}

\begin{figure}  
\centering
\subfigure[$D(6,4)$]{\includegraphics[scale=0.2]{Pics/D641a_0.eps}}
 \qquad \qquad
\subfigure[$D(10,5/2)$]{\qquad\includegraphics[scale=0.22]{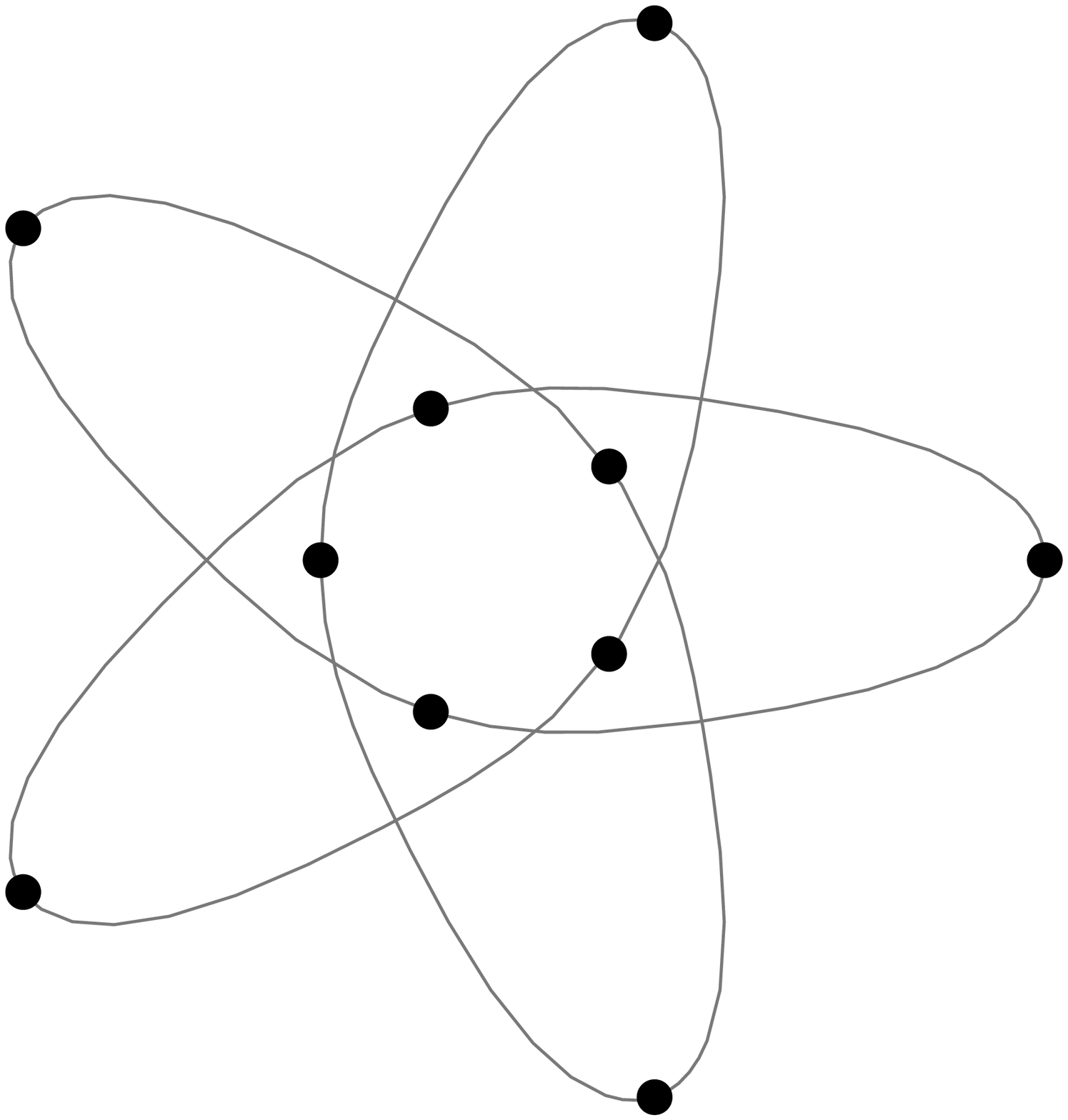}}\caption{Two examples illustrating different connected components of each symmetry type compared with those in Figure~\ref{fig:core}, see also Remark~\ref{rmk:figures}. Many examples of 5-body choreographies with a single time-reversing reflectional symmetry---so type $D(5,1)$ ---in different connected components have been found (numerically) by Sim\'o \cite[ Figure 3]{Simo01a}.}  
\label{fig:core2}
\end{figure}

Choreographies with the same symmetry but in different connected components can often be distinguished by their sequence of winding numbers: since the particles move without collision, in the full periodic orbit each pair of particles winds around each other some integer number of times. For a given choreography, the winding number for particles $i$ and $j$ depends only on the difference $|i-j|$, because of the choreography symmetry. For the two choreographies depicted in Fig.~\ref{fig:core2} the sequence of winding numbers is $(-3,-3,1)$ and $(-3, -3, 2, -3, -3)$ respectively, while for the choreographies of Fig.~\ref{fig:core} they are $(1,-3,1)$ and $(-3, 2, 2, 2, -3)$ respectively. The two with symmetry $D(6,4)$ have different winding numbers and so are not homotopic; similarly for the two $D(10, 5/2)$-choreographies. Here the notation is that the $k^{\mathrm{th}}$ term of the sequence is the winding number of particle $j$ around particle $j+k$. It is not hard to show that, for choreographies of type $C(n,k,\ell)$, the winding numbers are all equivalent to $\ell\bmod k$.

\begin{example}\label{eg:component with circular choreography}
Consider the circular choreography $u$ with speed $\ell=1$ and $n$ particles (Example~\ref{ex:circular motion}). If $z_0=(\omega, \omega^2, \dots, \omega^{n-1}, 1)\in\Xn$, where $\omega=\ee^{2\pi\ii/n}$, then $u(t)=\ee^{2\pi\ii t}z_0$. For this example, it is convenient to use $z_0$ as the base point. The generators of the braid group $\braid_i$ now represent the clockwise exchange of particle $i$ and $i+1$.   The braid $\delta=\braid_1\dots\braid_{n-1}$ becomes the rigid rotation by $2\pi/n$ of the configuration in the clockwise direction, represented by the path $\delta(t) = \ee^{-2\pi\ii t/n}z_0$.  The braid $\Delta^2=\delta^{n}$ is the full twist, given by a clockwise rotation by $2\pi$. 

Let $k$ be coprime to $n$ and let $(g,\tau(g))=(R_{2\pi a/k},\, \sigma_1^b,\,\nfrac{1}{nk})\in G_\tau=D(n,\infty)$, the symmetry group of $u$, where $an-bk=1$. Let $\gamma(t)=u(t/nk)$ for $t\in[0,1]$ (so $\gamma$ is the generator of $u$ ---Definition~\ref{def:generating path}). Then $\gamma\in\Lambda^g\Xn$ so it is natural to ask which connected component it belongs to. The components of $\Lambda^g\Xn$ are indexed by the $P_n$-conjugacy classes of $\delta^b P_n<CB_n$ as described in the theorem above. A calculation involving the isomorphism of Theorem~\ref{thm:EFG for Gamma} shows that the corresponding braid is precisely $\delta^b$. 

Since $b$ is coprime to $n$, it follows from results of \cite{Bessis-Digne-Michel} (see also \cite{G-MW04})  that the centralizer of $\delta^b$ in $P_n$ is the cyclic group generated by $\Delta^2$.  It then follows from Theorem~\ref{thm:connected components of choreographies} that the space of loops with symmetry $C(n,\infty)$ (or symmetry type $D(n,\infty)$) is homotopic to $S^1$, a fact which also follows from the direct argument in Example~\ref{eg:n divides k} below.
\end{example}

\begin{example}\label{eg:C(3,4)}
Consider the choreographies with 3 particles and 4-fold rotational symmetry, so $G=C(3,4)$. Then $G\simeq \ZZ_3\times\ZZ_4\simeq\ZZ_{12}$, which is generated by 
$$(g,\tau(g)) = (R_{3\pi/4},\,\sigma_1^{-1},\,\nfrac1{12})\in\Gamma\times\Sone$$
(see \eqref{eq:coprime generator}, with $an-bk=1$ for $a=b=-1$), and we have
$$\pi_1^G(X^{(3)},x) \simeq CB_3\times_{\alpha}\ZZ(1/4),$$
where $\alpha(1)=\Delta^2$, and the set of connected components is in 1--1 correspondence with the $P_3$-conjugacy classes in the coset $\delta^{-1} P_3\subset B_3$ (in fact in $CB_3$ the cyclic braid group).  

There are two examples of elements of $C(3,4)$ we have seen: the circular choreography with symmetry-type $D(3,\infty)$, see Examples~\ref{ex:circular motion} and~\ref{eg:component with circular choreography} above, and the choreography in Fig.~\ref{fig:D(3,4)} (and note that of course $D(3,4)<C(3,4)$).  The generators and braids of these are illustrated in Fig.~\ref{fig:C(3,4) paths}. The first of these corresponds to the $P_3$-conjugacy class containing $\delta^{-1}=\braid_2^{-1}\braid_1^{-1}$, as shown in the example above, while the second corresponds to the one containing $\br_2^2\delta^{-1}=\br_2\br_1^{-1}$.

It was shown in the previous example that the $P_3$-centralizer of $\delta=\br_1\br_2$ is the subgroup of $P_3$ generated by $\Delta^2$, so that the set of loops with symmetry $C(3,\infty)$ has the homotopy type of a circle (as was shown explcitly in Example~\ref{eg:component with circular choreography}.  On the other hand, 
the element $\br_2\br_1^{-1}$ is a pseudo-Anosov braid (by exclusion: it is neither periodic nor reducible \cite{G-M03}, or see \cite[Sec~15.1]{FarbMargalit}), and it follows from work of Gonz\'alez-Meneses and Weist \cite{G-MW04} that its $P_3$ centralizer is isomorphic to $\ZZ^2$ (in fact generated by $(\br_2\br_1^{-1})^3$ and $\Delta^2$).  Assuming the functional is coercive (which occurs for example for the strong force, see Sec.~\ref{sec:variational}), it follows by using Morse theoretic arguments, that it must have at least two $\Sone$-orbits of critical points on this component, for  otherwise it would be homotopic to a circle. 
\end{example}

\begin{figure}
\centering
\begin{pspicture}(-6,-2.2)(6,2)
\psset{unit=1.6} 
\rput(-2,0){%
  \psarc{->}(0,0){1}{0}{30}   \psarc[linestyle=dashed]{->}(0,0){1}{30}{115} 
  \rput(1.1,0.3){$\gamma$}   \rput(0.8,1){$(-\theta)_{g\cdot x}$}
  \psarc{->}(0,0){1}{120}{150} \psarc[linestyle=dashed]{->}(0,0){1}{150}{235} 
  \psarc{->}(0,0){1}{240}{270} \psarc[linestyle=dashed]{->}(0,0){1}{270}{355} 
  \psdots(1,0)(-.5, .866)(-.5, -.866)
  \rput(1.2,0){3}\rput(-0.6,-1){2}\rput(-0.6,1){1}
  \psdots[dotstyle=o](0.866,0.5)(-0.866,0.5)(0,-1)
}
\rput(2,0){%
  \psdots(1,0)(-.5, .866)(-.5, -.866)
  \rput(-0.6,1){1}  
  \psarc{->}(0,0){1}{120}{150} 
  \psarc[linestyle=dashed]{->}(0,0){1}{150}{235} 
  \rput(-0.6,-1){2} 
  \psecurve[showpoints=false](-0.5,-0.866)(-0.5,-0.866)(0,-0.3)(0.25,-0.15)(0.5,-0.25)(0.55,-0.4)(0.5,-0.5)  
  \psecurve[showpoints=false]{->}(0.5,-0.5)(0.45,-0.6)(0.25,-0.8)(0.05,-0.95)(0,-1) 
  \psecurve[showpoints=false](0.17,-0.17)(0.15,-0.25)(0.25,-0.5)(0.4,-0.55)(0.5,-0.5)(0.57,-0.45)(0.8,-0.25)(1,0)(1,0) 
  \psecurve[showpoints=false]{<-}(0.866,0.5)(0.85,0.45)(0.3,0)(0.2,-0.1)(0.17,-0.17) 
  \rput(1.2,0){3} 
  \psarc[linestyle=dashed]{->}(0,0){1}{270}{355} 
  \psarc[linestyle=dashed]{->}(0,0){1}{30}{115} 
  \rput(0.5,0.3){$\gamma$} \rput(0.8,1){$(-\theta)_{g\cdot x}$}
  \cput(1,-1){$\br_2^2$}
  \psdots[dotstyle=o](0.866,0.5)(-0.866,0.5)(0,-1)
}
\end{pspicture}
\caption{Two braids arising from loops with $D(3,4)$ symmetry. The first is the circular choreography, the second the one depicted in Fig~\ref{fig:D(3,4)}.  The black dots represent the base point $z_0\in X^{(3)}$,  and the open dots the point $\gamma(1)=g\cdot z_0$, where $g=(R_{-\pi/2},\,\sigma_1^{-1})$, and $\sigma_1^{-1}=(1\;3\;2)$. The solid curves represent $\gamma$ and the dashed ones $(-\theta)_{g\cdot x}$. 
(A break in the solid path in the second figure represents the one traversed later: in other words this is a top-down view of the motion, with time increasing downwards and the `hidden' curve broken as usual.)
} 
\label{fig:C(3,4) paths}
\end{figure}
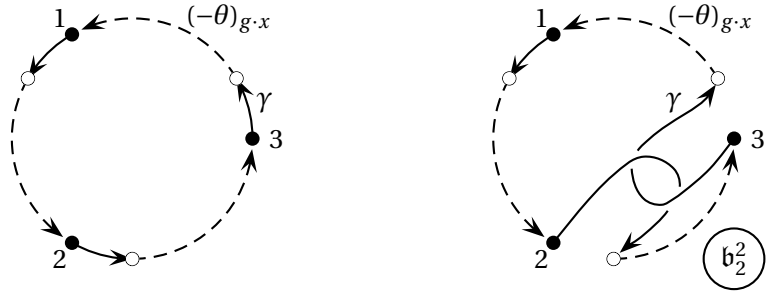

Finally, we consider symmetry groups where the number of particles divides the order of symmetry of the curve, and show directly that such loops are always homotopic to the circular choreography, and indeed in the case of the Newtonian potential the only choreographic solution is the circular one. 

\begin{example}\label{eg:n divides k}
Suppose $G=C(n,k/\ell)$ with $k$ a multiple of $n$.  Then $c=(n,k)=n$, and the fixed point space $\Fix(K,\Xn)\simeq X^{(1)}_*$, the punctured plane (see Lemma~\ref{lemma:avoiding origin}).  Indeed, the position of the first particle determines the others as they lie at the vertices of a regular $n$-gon centred at the origin.  The motion will preserve this property and so the solutions will be `homothetic'.  The fundamental group of this fixed point space is $\ZZ$, and the equivariant fundamental group $\pi_1^G(\CC^*,x) \simeq \ZZ\times \ZZ_k$, with $\beta(p,q)=(p\bmod n,\,q)\in G=\ZZ_n\times\ZZ_k$. The conjugacy is trivial, so there are countably many connected components (parametrized by $\pi_1(\CC^*)=\ZZ$), and each connected component is homotopic to a circle.  

More explicitly, let $u(t)=(z_1(t),\dots,z_n(t))$ be a choreography with symmetry $C(n,k/\ell)$ with $k=an$. The symmetry implies in particular that for all $t$, $z_i(t)/z_j(t)$ is never a positive real number. Define a new choreography $v$ homotopic to $u$ within the set of loops with symmetry $C(n,k/\ell)$ by projecting each point radially to the unit circle:
$$v(t) = \left(\frac{z_1(t)}{|z_1(t)|},\dots, \frac{z_n(t)}{|z_n(t)|}\right).$$
This choreography involves all the points moving equally spaced around the unit circle, and so in fact has symmetry $D(n,\infty/\ell)$. It follows that every component of $C(n,k/\ell)$ with $n|k$ contains a loop with symmetry $D(n,\infty/\ell)$. Furthermore, each of the $\ZZ$'s worth of connected components of $C(n,k/\ell)$ mentioned above contains the multiple coverings of the basic one. 

As far as the dynamics is concerned, the potential function on $\Xn$ restricts to a similar potential function on this fixed point space.  In particular, if the potential is homogeneous of degree $d$, then the restriction is also homogeneous of degree $d$, and the resulting dynamics coincides with that for a single mass moving in the corresponding central force.  For example, if the system is Newtonian ($d=-1$) the restricted dynamics will involve the vertices of the regular $n$-gon moving in ellipses, parabolae or hyperbolae.  However, if the motion is not circular then distinct vertices will move on distinct curves, so the only choreographic solution will be the circular motion.  Similar results should be available for other homogeneous potentials.
\end{example}

In the next section, we extend this example by asking which components of $C(n,k)$ and $C(n,k/\ell)$ do not contain loops with greater symmetry.

\section{Adjacencies and components}
\label{sec:adjacencies}

There is a familiar question in the application of topological methods to the calculus of variations which asks whether a given loop is homotopic to a loop which is multiply covered. If it is not, then existence theorems (such as existence of geodesics) guarantee there is a solution in the given homotopy class (connected component of loop space) which is not just a multiple covering of a simpler solution, so guaranteeing a `new' solution.  The answer is simply that a loop is not homotopic to a multiply covered loop if and only if the corresponding element of the fundamental group is \emph{primitive}.  An element of a group is said to be primitive if it cannot be written as a positive power of another element; this property is invariant under conjugacy.

From the symmetry perspective, even if there is no symmetry group on the configuration space $X$, there is always an $\Sone$-action on the loop space $\Lambda X$ (we don't discuss the $\Sonehat$ action in this section). A trivial loop is one with isotropy precisely $\Sone$. A loop has symmetry $\ZZ_r<\Sone$ if and only if it is $r$-times covered. So, if a loop $u$ has symmetry $\ZZ_r<\Sone$ then the curve $v:t\mapsto u(t/r)$ for $t\in[0,1]$, is also a loop, and the homotopy classes satisfy $[u]=[v]^r$, explaining what was said above. In particular, $u$ is freely homotopic to a multiply covered curve (i.e., one with greater symmetry) if and only if $[u]\in\pi_1(X,x)$ is not a primitive element.

The natural extension of this question to the symmetric setting is to ask whether a particular connected component of the space of loops with symmetry $G$ contains loops with greater symmetry (of which multiple coverings are a particular case). We saw an instance of this in Example~\ref{eg:n divides k} where $G=C(n,k/\ell)$ with $n\divides k$: all such components contain loops with symmetry $D(n,\infty/\ell)$. 

In the context of choreographies, we saw in Section~\ref{sec:lattice} that $C(n,k/\ell) \prec C(n,k',\ell')$ iff $k\divides k'$ and $\ell\equiv\ell'\bmod k$.  Some components of the loops with symmetry (at least) $C(n,k/\ell)$ contain loops with symmetry $C(n,k',\ell')$ and some do not.  Similarly, if $n$ is odd, $C(n,1)\prec C'(n,2)$ and the same question arises.  In this section we develop a method to determine which components do contains loops with greater symmetry and which do not, in terms of conjugacy classes in the equivariant fundamental group $\pi_1^\Gamma(\Xn,x)$. 
If the action of $\Gamma$ on $X$ is free, this is a simple extension of the classical result above, the requirement being that the relevant element $(\gamma,g)$ be a primitive element of $\pi_1^G$---see Corollary~\ref{coroll:primitives} below. 

\subsection{Maximal symmetry type}\label{sec:max symmetry type}

Returning to the general context of actions of a group $\Gamma$ on a connected manifold $X$, suppose a loop $u$ has symmetry equal to $G<\Gamma\times\Sone$, with $\tau(G)=\ZZ_r<\Sone$ and suppose $g\in G$ is such that $\tau(g)=\nfrac1r$. Such a $g$ is uniquely defined modulo $K=\ker\tau\lhd G$ and moreover, since $u$ has isotropy $G$, $g\in N=N_\Gamma(K)$, the normalizer of $K$ in $\Gamma$.  As usual define in this case $\gamma(t)=u(t/r)$ for $t\in[0,1]$, the generator  of $u$, so $u=\phi_r(\gamma)$ in the notation of \eqref{eq:defn of phi_r}. Then $\gamma\in\Lambda^gX^K$.

Let $C(u)$ denote the connected component of $\Fix(G,\Lambda X)$ containing $u$, or what is equivalent by Proposition~\ref{prop:fix is Lambda^g}, the connected component of $\Lambda^gX^K$ containing $\gamma$.  The question is whether $C(u)$ contains any points with symmetry group strictly containing $G$.  We say the component $C(u)$ is of \defn{maximal symmetry type} if there are no loops in $C(u)$ with symmetry strictly greater than $G$.

\begin{theorem}
Suppose the group $\Gamma$ acts on a manifold $X$, and let $u$ be a loop in $X$ with isotropy $G< \Gamma\times\Sone$ with $K=\ker\tau$ trivial,  and let $\gamma\in\Lambda^gX$ be a generator of $u$  The connected component $C(u)$ of\/ $\Fix(G,\Lambda X)$ containing $u$ is of maximal symmetry type if and only if both the following conditions hold:
\begin{enumerate}
\item  $(\gamma,g)\in\pi_1^\Gamma(X,x)$ is a primitive element, and
\item there does not exist a non-trivial isotropy subgroup $K$ of the $\Gamma$-action on $X$ for which  $(\gamma,g)\in \pi_1^{N}(X^{K},x')$, where $N=N_\Gamma(K)$ and $x'\in X^K$.
\end{enumerate}
More generally, the same is true if $u\in \Lambda Y$ for $Y=X^{K_0}$ with $K_0=\ker\tau\lhd G$ non-trivial, provided $\Gamma$ is replaced by $\Gamma_0:=N_\Gamma(K_0)/K_0$, $X$ is replaced by $Y$ and for condition (2), $K$ strictly contains $K_0$, $x$ and $x'$ should be in the same connected component of $X^{K_0}$ and $N$ is replaced by 
$$N' = (N_\Gamma(K_0)\cap N_\Gamma(K))/K_0.$$
\end{theorem}

Condition (2) implicitly uses the identification of $\pi_1^{N}(X^K,x')$ with its image in $\pi_1^\Gamma(X,x)$ (since $X^{K}\subset X$ and $N<\Gamma$). More generally it uses the identification of $\pi_1^{N'}(X^{K},x)$ with its image in $\pi_1^{\Gamma_0}(X^{K_0},x)$.  Note that this condition (2) is trivially satisfied if $K$ is a maximal isotropy subgroup of the $\Gamma$-action on $X$.

In the more general version of condition (1) the group $\pi_1^{\Gamma_0}(Y,\,x)$ arises from the natural action of $N_\Gamma(K_0)/K_0$ on $Y=X^{K_0}$, and is isomorphic to the quotient
$$\trivial \lra K_0 \lra \pi_1^{N_\Gamma(K_0)}(Y,\,x) \lra \pi_1^{\Gamma_0}(Y,\,x) \lra \trivial$$
where the first inclusion is $k\mapsto(\bar x, k)$, where $\bar x$ is the constant loop at $x$. 

\begin{proof}
We prove the `basic' version; the more general version follows by considering the $\Gamma_0$-action on $X^{K_0}$.

First we prove the `only if' statement. Suppose firstly that $(\gamma,g)\in\pi_1^\Gamma(X,x)$ is not primitive, so there is a $(\delta,h)\in\pi_1^\Gamma(X,x)$ and $p>1$ such that $(\gamma,g)=(\delta,h)^p$.  Let $v$ be the loop corresponding to $\delta$. Then $v$ is homotopic to $u$ in $\Fix(G,\Lambda X)$ so $C(u)$ is not maximal. Secondly suppose condition (2) is violated.  Since $X$ is connected, the groups $\pi_1^\Gamma(X,x)$ and $\pi_1^\Gamma(X,x')$ are isomorphic, so let $(\gamma',g)$ be the corresponding element in the latter group (same $g$).   Then   there is a $(\delta,h)\in\pi_1^N(X^K,x')$ such that $(\gamma',g)=(\delta,h)$ ($h$ is only defined modulo $K$, but we can choose $h=g$). It follows that $u$ and $v$ are homotopic.

For the converse, suppose that the component $C(u)$ is not of maximal symmetry type, and let $v\in C(u)$ be a loop with greater symmetry. That is, $v\in\Fix(H,\Lambda X)\cap C(u)$ for some group $H<\Gamma\times\Sone$ with $H\gneqq G$.  Let $\tau':H\to\Sone$ be the projection, and let $K$ be the kernel.  If $K$ is non-trivial, then condition (2) fails.   

On the other hand, if $K$ is trivial, then the image of $\tau'$ is a subgroup of $\Sone$ strictly containing $\tau(G)$, so has order $pr$ for some integer $p>1$.  Let $h$ satisfy $\tau'(h)=\nfrac{1}{pr}$ and define  $\delta(t)=v(t/pr)$ for $t\in[0,1]$. Then $\delta\in\Lambda^hX$, and $(\gamma,g)=(\delta,h)^p$ so that $(\gamma,h)$ is not primitive.
\end{proof}


As mentioned above, if the action of $\Gamma$ on $X$ is free the theorem simplifies considerably:

\begin{corollary} \label{coroll:primitives}
Suppose the action of\/ $\Gamma$ on $X$ is free and let $u$ be a loop with symmetry  $G<\Gamma\times S^1$.  The connected component $C(u)$ of\/ $\Fix(G, \Lambda X)$ is of maximal symmetry type if and only if the corresponding element $(\gamma,g)\in\pi_1^{\Gamma}(X,x)$ is primitive, where $\gamma$ is the generator of\/ $u$.
\end{corollary}

\subsection{Application to choreographies}

Since we know from Theorem~\ref{thm:classification} all the isotropy subgroups of $\Gamma\times\Sone$ appearing for choreographies we can simplify the application of the theorem above. There are two types of case to consider.  The first, more direct and only for odd $n$, is that $C(n,1)=\Chor_n\prec C'(n,2)$ and the other, which involves many sub-cases, is that $C(n,k/\ell)\prec C(n,k'/\ell')$ if and only if $k\divides k'$ and $\ell'\equiv\pm\ell\bmod k$.

Suppose first that $n$ is odd and $u$ has symmetry $C(n,1)$. 
By Theorem~\ref{thm:connected components of choreographies}, the set of connected components of $C(n,1)$ is in 1--1 correspondence with the $P_n$-conjugacy classes in the coset $\delta^{-1} P_n\subset CB_n$ (the cyclic braids are defined in Eq.~\eqref{eq:CB_n}).  The connected component $C(u)$ containing $u$ therefore corresponds to the $P_n$-conjugacy class containing $\delta^{-1} p_u$ for some $p_u\in P_n$ 

Suppose $v\in C(u)$ has symmetry $C'(n,2)$ and let $\eta(t)=v(t/2n)$, so that $\eta\in\Lambda^{(\kappa,\sigma_2)}\Xn$ (see Sec.~\ref{sec:notation} for notation).  Since there is no rotational part to this symmetry group, the relevant equivariant fundamental group is
$$\pi_1^{\ZZ_2\times \Sigma_n}(\Xn,x) \simeq CB_n\rtimes\ZZ_2,$$
with product $(\br_1,\kappa)(\br_2,g) = (\br_1\bar \br_2,\, \kappa g)$ for $g\in\{I,\kappa\}$, and $\bar\br$ is the braid obtained from $\br$ by changing all overcrossings to undercrossings and vice versa (see Theorem~\ref{thm:EFG for Gamma}).  The map $\phi_2:\Lambda^{(\kappa,\sigma_2)}\Xn\to \Lambda^\chor\Xn$ is 
$$(\gamma,\,(\kappa,\sigma_2)) \longmapsto (\gamma*\bar\gamma,\,\chor).$$
Furthermore, connected components correspond to twisted conjugacy classes, and indeed under this map, for arbitrary $r\in P_n$,
$$r*\gamma*\bar r^{-1} \longmapsto (r*\gamma*\bar r^{-1})*(\bar r*\bar\gamma* r^{-1} ) = r*(\gamma*\bar\gamma)*r^{-1},$$
up to homotopy, so mapping connected components to connected components (as it must!).  Moreover, with $(\gamma,(\kappa,\sigma_2))$ corresponding to $(\br_0,\kappa)\in B_n\rtimes\ZZ_2$ we have $\br_0 = \delta'p$ for some pure braid $p$ (see Theorem \ref{thm:connected components of choreographies}). The corresponding braid under $\phi_2$ is $\br_0\bar\br_0 = \delta' p \overline{\delta'}\bar p = \delta'\overline{\delta'}p'\bar p = \delta p''$, as  required for braids with $C(n,1)$ symmetry.
Thus we have, for odd $n$,

\begin{proposition}
A loop $u$ with symmetry $C(n,1)$ is homotopic to a loop with symmetry $C'(n,2)$ if and only if the corresponding element $\br$ of the braid group can be written as $\br = \br_0\overline{\br_0}$.
\end{proposition}

Note that $\pi(\br)=\pi(\br_0)^2=\sigma_1$ which implies that $\pi(\br_0)=\sigma_2$ ($\sigma_2$ is the unique square root of $\sigma_1$ in $S_n$), so any $\br_0$ satisfying $\br=\br_0\overline{\br_0}$ corresponds to a loop with symmetry $C'(n,2)$.  

On the other hand, as pointed out earlier, the exponent sum satisfies $\chi(\bar\br)=-\chi(\br)$. It follows that $\chi(\br\bar\br)=0$, and consequently the braid associated to any component of the set of choreographies containing loops with $C'(n,2)$ symmetry must have zero exponent sum.

An open question arising here is whether, for $g'$ the generator of $C'(n,2)$, the map
$$\pi_0(\Lambda^{g'} \Xn) \lra \pi_0(\Lambda^\chor\Xn)$$
is injective. Failure of this would amount to the existence of two braids $\br,\br' \in \delta'P_n$ in distinct twisted conjugacy classes, for which $\br\overline{\br}$ is conjugate to $\br'\overline{\br'}$.  This would have implications for the number of critical points of the action functional in a connected component of $C(n,1)$, as there would be at least two critical points in that component with symmetry (conjugate to) $C'(n,2)$. 

\medskip

Now consider the case $C(n,1)\prec C(n,k/\ell)$, with $(n,k)=1$. Since there are no fixed points for the $\Gamma$-action, we are in the setting of Corollary~\ref{coroll:primitives}. Let $g=(R_{2\pi a\ell/k},\sigma_1^b)$, and $\delta\in\Lambda^g\Xn$, where as usual $an-bk=1$. Then the map $\phi_k:\Lambda^{g}\Xn\to\Lambda^\chor\Xn$, as in Eq.~\eqref{eq:defn of phi_r}, and at the level of homotopy,
$$\phi_k(\eta,g)=(\eta,g)^k = (\gamma,\chor),$$
with $\gamma=\phi_k(\eta)$. Representing $(\eta,g)\in\pi_1^\Gamma(\Xn,x)$ as $[\br_0,a\ell/k]\in B_n\times_\alpha\RR$ and $(\gamma,\chor)$ as $[\br,0]$ we require $[\br_0,a\ell/k]^k=[\br,0]$.  This becomes
$$[\br_0^k,a\ell]=[\br,0]$$
so that $\br = \Delta^{-2a\ell}\br_0^k$. This implies the following.

\begin{proposition}
A loop $u$ with symmetry $C(n,1)$ and corresponding braid $\br\in\delta^{-1} P_n$ is homotopic to a loop with symmetry $C(n,k/\ell)$ with $(n,k)=1$ if and only if  $\Delta^{2a\ell}\br$ has a $k^\mathrm{th}$ root, where $an\equiv 1\bmod k$.
\end{proposition}

Having a $k^{\mathrm{th}}$ root means of course that it can be written as $\Delta^{2a\ell}\br = \br_0^k$ for some braid $\br_0$.  Since $\chi(\Delta^2) = n(n-1)$, this requires $2a\ell n(n-1)+\chi(\br)$ to be a multiple of $k$, and hence the exponent sum satisfies
$$\chi(\br)\equiv - (n-1)\ell\bmod k.$$ 
This restricts the possible values of $k,\ell$.
We aim to consider the full question, allowing for the core, in a future paper.

Finally we mention a result of Gonz\'alez-Meneses \cite{G-M03} who shows that the $k^{\mathrm{th}}$ root of a braid, if it exists, is unique up to conjugacy in $B_n$.  Moreover, if the braid is of \emph{pseudo-Anosov type}, then a $k^{\mathrm{th}}$ root, if it exists, is unique. (The \emph{type} of $\br$ coincides with the type of $\Delta^2\br$ as is readily checked.)   This suggests that the map
$$\pi_0(\Lambda^g\Xn)\lra\pi_0(\Lambda^\chor\Xn)$$
may be injective. However, the conjugacy in \cite{G-M03} is by all possible elements of $B_n$, while $\pi_0$ is determined by conjugacy with respect to elements of $P_n$. So this is inconclusive, but again would have repercussions for estimates of numbers of critical points.

\paragraph{Acknowledgements}  We would like to thank Mark Roberts, Jelena Grbic and Peter Rowley for helpful discussions, and Dan Gries for creating the HTML5/Javascript code which displays the animations on the website \cite{JMweb}.  We would also like to thank the two anonymous referees whose suggestions have helped improve the presentation. The second author was supported during her PhD studies by Forrest Recruitment, and would like to express her gratitude to John and Stephanie Forrest for their support, and the interest they took in this research.

\small

\def\cprime{$'$}

\vskip 1cm

\noindent\begin{minipage}{0.3\textwidth}
\obeylines\it
School of Mathematics
University of Manchester
Manchester M13 9PL
UK
\end{minipage}

\bigskip
\bigskip

\noindent \emph{Corresponding author}: \texttt{ j.montaldi@manchester.ac.uk}

\end{document}